\documentclass[11pt]{article}
\usepackage[margin=1in]{geometry}
\usepackage[english]{babel}
\usepackage{chngcntr}
\usepackage{amsmath}
\usepackage{amsthm}
\usepackage{amssymb}
\usepackage{bm}
\usepackage{boldline}
\usepackage[shortlabels]{enumitem}
\allowdisplaybreaks
\usepackage{amsfonts}
\usepackage{comment}
\usepackage{setspace}
\usepackage[dvipsnames]{xcolor}
\usepackage{tikz-cd}
\usepackage{tikz}
\usepackage{scrextend}
\usepackage{graphicx}
\usepackage{calc}
\usepackage{tabu}
\usepackage{xcolor}
\usepackage{amsthm}
\usepackage{titlesec}
\usepackage{makecell}
\usepackage{thmtools}
\usepackage[
colorlinks=true,
urlcolor=purple,
linkcolor=purple!87!black,
citecolor=blue!87!black,
pdfborder={0 0 0}
]{hyperref}
\titleformat*{\section}{\bfseries}
\titleformat*{\subsection}{\bfseries}
\usepackage[capitalize,nameinlink]{cleveref}
\newtheorem{theorem}{Theorem}[subsection]
\counterwithin*{theorem}{section}

\renewcommand*{\thetheorem}{
  \ifnum\value{subsection}<1 
    \thesection
  \else
    \thesubsection
  \fi
  .\arabic{theorem}
}

\newtheorem{corollary}[theorem]{Corollary}
\newtheorem{lemma}[theorem]{Lemma}
\newtheorem{proposition}[theorem]{Proposition}
\newtheorem{definition}[theorem]{Definition}
\newtheorem{example}[theorem]{Example}
\newtheorem{remark}[theorem]{Remark}
\newtheorem{conjecture}[theorem]{Conjecture}
\newtheorem{myfigure}[theorem]{Fig.}
\newtheorem{mytable}{Table}

\usetikzlibrary{shapes,arrows}
\newcommand{\no}{\noindent}
\newcommand{\pp}{\,\,\,}
\usepackage{array}
\usepackage{arydshln}
\newlength\dlf
\newcommand\alignedbox[2]{
  &
  \begingroup
  \settowidth\dlf{$\displaystyle #1$}
  \addtolength\dlf{\fboxsep+\fboxrule}
  \hspace{-\dlf}
  \fcolorbox{Blue!20}{VioletRed!20}{$\displaystyle #1 #2$}
  \endgroup
}
\newlist{myenum}{enumerate}{3}
\setlist[myenum,1]{label=\textbf{(\arabic*)},
                   ref  =\textbf{(\arabic*)}}
\setlist[myenum,2]{label=\bfseries(\roman*),
                   ref  =\themyenumii\textbf{(\roman*)}}
\setlist[myenum,3]{label=\textbf{(\alph*)},
                   ref  =\themyenumi\textbf{(\alph*)}}
                   
\newcommand{\cpref}[2]{%
  \hyperref[#2]{\namecref{#1}~\labelcref*{#1}~\ref*{#2}}
}

\begin{document}
\title{An Indicator Formula for the Hopf Algebra $k^{S_{n-1}}\#kC_n$}

\author{\vspace{0.5cm} Kayla Orlinsky
\\ University of Southern California \\ Los Angeles, CA 90089-1113}



\noindent\textbf{\Large An Indicator Formula for the Hopf Algebra $k^{S_{n-1}}\#kC_n$}\\

\noindent\textbf{Kayla Orlinsky} University of Southern California, Los Angeles, CA 90089-1113\\ 
Email: korlinsk@usc.edu; Phone: (949) 579-0796\\

\noindent\textbf{Abstract } The semisimple bismash product Hopf algebra $J_n=k^{S_{n-1}}\#kC_n$ for an algebraically closed field $k$ is constructed using the matched pair actions of $C_n$ and $S_{n-1}$ on each other. In this work, we reinterpret these actions and use an understanding of the involutions of $S_{n-1}$ to derive a new Froebnius-Schur indicator formula for irreps of $J_n$ and show that for $n$ odd, all indicators of $J_n$ are nonnegative. We also derive a variety of counting formulas including \cref{thm:classifying indicators when t=2} which fully describes the indicators of all $2$-dimensional irreps of $J_n$ and \cref{thm:number of positive indicators of dimension} which fully describes the indicators of all odd-dimensional irreps of $J_n$ and use these formulas to show that nonzero indicators become rare for large $n$.\\



\section{Introduction}\label{sec:intro}

In this paper, we study the indicators of the simple modules (irreps) of a particular semisimple Hopf algebra, specifically the bismash product $J_n=\mathbb{C}^{S_{n-1}}\#\mathbb{C}C_n$ where $S_n$ is the symmetric group. We rely especially on the works of \cite{mo} which served as a general resource on Hopf Algebras as well as \cite{jm} which derived an indicator formula for irreps of bismash product Hopf algebras. Additionally, a key result of \cite{kmm} (\cref{thm:[KMM] inducing modules}) showed that there is a one-to-one correspondence between irreducible representations (irreps) of $J_n$ and group representations of $F_x\subset C_n$ where $F_x$ is the stabilizer of a permutation $x\in S_{n-1}$ under the matched pair action of $C_n$ on $S_{n-1}$ (see \cref{sec:Group Action Results}). Using this correspondence, we adapt the formulas of \cite{jm} and find new ways to explicitly describe and count irreps of $J_n$ with certain indicators. Finally, we use code developed in Python to provide insight into negative indicators.

The idea that indicators give insight into the structure of representations started with groups. Given an irreducible character $\chi$ of a finite group $G$, averaging the values outputted by the character evaluated at the square of each group element is called the Frobenius-Schur \cite{fs} indicator $$\nu(\chi)=\frac{1}{|G|}\sum_{g\in G}\chi(g^2).$$ First discovered by its namesakes in 1906, the magic of the Frobenius-Schur indicator is that it always takes the value $0$, $1,$ or $-1$. These three values determine whether the associated representation--thought of as $kG$-module--admits a nondegenerate $G$-invariant bilinear form ($\nu(\chi)\not=0$) and if so, if that form is symmetric or skew-symmetric ($\nu(\chi)=1$ or $-1$ respectively).

Although \cite{fgsv} extended the indicator to Kac algebras in 1999, explicitly constructing an analogue of the Frobenius-Schur indicator for semisimple Hopf algebras was first done by Linchenko and Montgomery in 2000 \cite{lm}. The key idea of \cite{lm} came from the fact that a semisimple Hopf algebra $H$ over an algebraically closed field admits a unique two-sided integral element $\Lambda$ satisfying $h\Lambda=\varepsilon(h)\Lambda=\Lambda h$ for all $h\in H.$ It was then found that by computing the (irreducible) character of the second Sweedler power of this integral $$\nu(\chi)=\chi((m\circ\Delta)(\Lambda))=\chi\left(\sum\Lambda_1\Lambda_2\right),$$ that the resulting sum would also only return $0,1,$ or $-1$. Here $m$ and $\Delta$ are the multiplication and comultiplication on the Hopf algebra. Furthermore, these three values given by $\nu(\chi)$ determine whether the associated $H$-module admits a nondegenerate $H$-invariant bilinear form ($\nu(\chi)\not=0$) and if so, if that form is symmetric or skew-symmetric ($\nu(\chi)=1$ or $-1$ respectively). It is a perfect analogue of its group-theoretic counterpart for semisimple Hopf algebras.

Frobenius-Schur indicators on Hopf algebras are of interest for a variety of reasons, three of which stand out. First, and most obviously, they determine if the associated $H$-module (representation) has an $H$-invariant, bilinear form and further if it is symmetric or skew-symmetric. In fact, it was shown by Linchenko and Montgomery \cite{lm} that $\nu(\chi)\not=0$ if and only if the associated simple module (with character $\chi$) is self dual.

Secondly, the indicator has been found to aid in the classification of finite dimensional Hopf algebras. This problem of describing all Hopf algebras of a certain dimension is in general very difficult and tedious. For example, in 2000, Kashina \cite{k} classified all semisimple Hopf alebgras of dimension $16$ and then in 2001, Kashina, Mason, and Montgomery \cite{kmm} showed that the computations of all Hopf Algebras of dimension 16 could be improved using the (second) indicator $\nu$ \cite{kmm}. Also \cite{ksz1} used the indicator to put restrictions on the dimension of a Hopf algebra based on the dimension of its simple modules. For example, they showed that if $H$ has an even dimensional simple module, then the dimension of $H$ itself must be even \cite{ksz1}.

Thirdly, and most generally, Frobenius-Schur indicators can be adapted to have broader applications to category theory. If $H$ is a semisimple Hopf algebra over $\mathbb{C}$, then we can define Rep$(H)$, which is the category of finite dimensional representation of $H$. In fact, Rep$(H)$ is a pivotal fusion tensor category meaning that it is a category with finitely many isomorphism classes of simple objects (fusion) with a notion of dualization (pivotal) and tensor product (tensor). It was shown by Mason and Ng in 2005 \cite{mn} that the Frobenius-Schur indicator is a categorical invariant. That is, if $H$ and $H'$ are two semisimple Hopf algebras over an algebraically closed field of characteristic $0$ such that Rep$(H)$ is equivalent to Rep$(H')$ as $k$-linear tensor categories, then $\{\nu(\chi)\}_{\chi\in Irr(H)}$ and $\{\nu(\chi')\}_{\chi\in Irr(H')}$ are identical.

Although we do not discuss them in this paper, there are abundant applications if one were to expand out to higher indicators. That is, by computing the $n^{\text{th}}$-Sweedler power of $\Lambda$ (as described above). For example, in the same paper where the idea is introduced, \cite{ksz2} prove a version of Cauchy's Theorem for semisimple Hopf alebras using the regular representation.

In this paper, we focus on the semisimple Hopf algebra $J_n=\mathbb{C}^{S_{n-1}}\#\mathbb{C}C_n$ which is a bismash product of the two groups $S_{n-1}$ and $C_n$. Matched pairs of groups have been known to group theorists as factorizable groups since the 1950s, but the construction of Hopf algebras from a mathced pair of groups was first proposed by Takeuchi \cite{ta} in 1981. In fact, it was shown by Masuoka \cite{ma} that the Hopf algebra of interest in this paper, $J_n$, is one of only two Hopf algebras that arise from the decomposition of $S_n$ as a matched pair of $S_{n-1}$ and $C_n$. The other Hopf alebgra being $H_n=(J_n)^*$ the dual of $J_n$ for which \cite{jm} have already shown that all indicators of simple modules are $+1$.

Similar work has been done by Timmer \cite{ti} in 2015 using different factorizations of the symmetric group $S_n=S_{n-r}G$ and considering $H=k^{G}\#kS_{n-r}$. Timmer \cite{ti} shows that the indicator is nonnegative, a result which we find to be false for $J_n$ for certain $n.$

The goal of this work was to fully describe all irreps of $J_n$ and state their second indicators $\nu(\chi)=\nu_2(\chi)$. When $n$ is odd, this can be done explicitly using an understanding of involutions in the symmetric group $S_{n-1}$ and some computation. However, when $n$ is even, more work is needed to completely describe all irreps and their indicators. Specifically, when $n$ is odd, we prove the indicator is nonnegative but for $n\ge 12$ such that $4$ divides $n$, there exists representations of $J_n$ with negative indicator which has complicated the problem of classifcation. The organization of the paper is as follows.

In \cref{sec:Group Action Results}, we discuss the group actions which determine the underlying structure of our Hopf algebra. This involves describing facorizable groups, which are groups $L$ which contain two subgroups $F,G$ such that $F\cap G=1$ and $L=GF=FG$. Note that this allows every element in $L$ to be written uniquely as a product of elements in $F$ and $G$, but does not imply that $L$ can be written as a direct or even semi-direct product of $F$ and $G$. Since we can write every element $l\in L$ as a product $xa\in GF=L$ with $a\in F$ and $x\in G$, two natural group actions arise of these subgroups on each other which describe how to write $xa$ uniquely as an element of $FG$. Note that each of these actions are group actions in the sense that $G$ acts on the set $F$ and $F$ acts on the set $G$. However, neither respects the group operation of the other. The factorizable group of interest to us will be the symmetric group of order $n$, $S_n$ which we will show can be written as a product of $C_n$ and $S_{n-1}$, where $C_n$ here is the cyclic group generated by the standard $n$-cycle $a=(1\pp 2\pp \cdots \pp n)$. We will discuss the orbits, stabilizers, and other useful sets and give some lemmas which we will need in order to prove our indicator counting formulas later on.

In \cref{sec:Hopf and rep results}, we explicitly describe the Hopf algebra $J_n$ and summarize relevant results in \cite{jm}, which include a discussion of the classification of the indicators of $H_n=(J_n)^*$ the dual of $J_n$. In \cite{jm}, a complete classification of indicators is given for $J_p$ where $p$ is prime, and also for $H_n$ for all $n.$ While we have not been able to fully describe all indicators of $J_n$ for all $n$, in this section we discuss the methods used for when $p$ is prime and provide some useful results which apply for general $n.$

In \cref{sec:reinterpretting the action}, we give a rewritten indicator formula which describes how to compute indicators from irreps induced from permutations using minimal information about the permutation. However, to obtain it, we must first completely re-interpret the actions of $C_n$ on $S_{n-1}$. We describe how viewing each permutation in the symmetric groups as a bijection on $\mathbb{Z}/n\mathbb{Z}$ and then using some modular arithmetic, we uncover new relationships within the orbits and stabilizers of the group actions described in \cref{sec:Hopf and rep results}. For example, when we consider the action of $C_n=\langle a \rangle$ on $S_{n-1}$, we show that if a permutation $x\in S_{n-1}$ (viewed as a permutation in $S_n$ which fixes $n$) has stabilizer $F_x=\langle a^t\rangle$, then $x$ exhibits a number of remarkable properties including a sort of linearity in $t$. We use these properties to derive \cref{thm:indicator reduced} which is an improved indicator formula which drastically reduces--but does not eliminate--the computation requirement. Using this, we show that for any $n$ not divisible by $4$, the indicator of any irrep of $J_n$ is always nonnegative.

In \cref{sec:explicitly counting}, we cover majority of key results for this paper. Due to the computational nature of the approach we take, we have broken this section into subsections. Throughout this section, we use the relations discovered in \cref{sec:reinterpretting the action} to fully construct (and therefore count) many sets of interest. We show that any permutation $x\in S_{n-1}$ with stabilizer $\langle a^t\rangle\subset C_n$ can in fact be constructed from some permutation in $S_{t-1}$.

In \cref{sec:revised indicator}, we use the formulas from the previous section to show that if $n$ is odd, then counting irreducible representation of $J_n$ is equivalent to counting involutions of $S_{n-1}$ with certain properties and, given an odd integer $t$, provide a formula for counting all $t$-dimensional irreps with indicator $+1$ (\cref{thm:number of positive indicators of dimension}). Although we cannot describe all indicators of $t$-dimensional irreps of $J_n$ when $t$ is even, we are able to fully describe the indicators of all $2$-dimensional irreps of $J_n$ (\cref{thm:classifying indicators when t=2}).

In \cref{sec:limit behavior}, we discuss limiting behavior using the formulas we developed previously. We show that the ratio of nonzero irreps of dimension $t$ (odd or $t=2$) with indicator $+1$ to the total number of irreps of dimension $t$ tends to $0$ as $n\to\infty$ (\cref{thm:percent +1 ind goes to zero} and \cref{thm:percent nonzero ind goes to zero}).

\section{\bf Group Action Results}\label{sec:Group Action Results}

In this section, we describe the underlying group actions which make up the structure of the bismash product Hopf algebra $J_n$. Throughout, $L,F,G$ will be finite groups. Elements $a,b\in F$ and $x,y\in G$ and $l\in L$.

Consider a (finite) group $L$ with two subgroups $F$ and $G$ which intersect trivially and whose product is the entire group $L$. We call such a group factorizable and such objects have been studied for decades due to their unique decomposition.

\begin{definition}\label{def:factorizable group}
{\it 
A group $L$ is factorizable if it has two subgroups $F$ and $G$ such that $F\cap G=1$ and $L=GF=FG$.
}
\end{definition}
 
Being able to write that $L=GF$ and knowing that $F\cap G=1$ implies that $L=FG$ and it also implies that every element $l\in L$ can be written uniquely as a product of something in $F$ and something in $G$. 

Note we cannot say that $L$ is a direct (or even semi-direct) product of $F$ and $G$ because nowhere are we assuming normality of either subgroup.

\begin{definition}[\cite{ta}]\label{def:matched pair}
{\it 
Let $F$ and $G$ be groups. Let $\lhd:G\times F\to G$ and $\rhd:G\times F\to F$ be group actions. Then $(F,G,\rhd,\lhd)$ is a matched pair if for all $x,y\in G$ and all $a,b,\in F$, \begin{align*}
    x\rhd ab&=(x\rhd a)((x\lhd a)\rhd c)\\
    xy\lhd a&=(x\lhd(y\rhd a))(y\lhd a)
\end{align*} 
}
\end{definition}

While this definition turns out to be equivalent to that of a factorizable group, \cite{ta} did not know that when he coined this term in the '80s.

\begin{lemma}\label{lem:factorizable is matched pair}
{\it Let $L=GF$ be a factorizable group. Then $(F,G,\rhd,\lhd)$ is a matched pair where for $x\in G$ and $b\in F$, the decomposition below is unique: $$xb=(x\rhd b)(x\lhd b).$$ Conversely, if $(F,G,\rhd,\lhd)$ is a matched pair, then $L$ which has underlying set $G\times F$ and multiplication defined by $$(a,x)(b,y)=(a(x\rhd b),(x\lhd b)y)$$ is a factorizable groups}
\end{lemma}

We use the same definitions as from \cite{jm}. Many of the results given next proved in \cite{jm} or follow immediately from the definitions.

\begin{proposition}\label{prop:facts about matched pair}
{\it Let $L=GF=FG$ be a factorizable group where the actions $\lhd$ and $\rhd$ are defined such that $xb=(x\rhd b)(x\lhd b)$ is unique for all $x\in G$ and $b\in F$. Then the following properties hold:
\begin{center}
\begin{minipage}{0.45\textwidth}
\begin{myenum}
    \item \label{prop:facts about matched pair (1)} $x\rhd bc=(x\rhd b)((x\lhd b)\rhd c)$ 
    \item \label{prop:facts about matched pair (2)} $xy\lhd b=(x\lhd (y\rhd b))(y\lhd b)$
    \item \label{prop:facts about matched pair (3)} $xy\rhd b=x\rhd (y\rhd b)$
    \item \label{prop:facts about matched pair (4)} $x\lhd bc=(x\lhd b)\lhd c$
    \item \label{prop:facts about matched pair (5)}$1\rhd b=b$
\end{myenum}
\end{minipage}\hspace{0.25cm}\begin{minipage}{0.45\textwidth}
\begin{enumerate}[label=\textbf{(\arabic*)}]
\setcounter{enumi}{5}
    \item \label{prop:facts about matched pair (6)} $1\lhd b=1$
    \item \label{prop:facts about matched pair (7)} $x\lhd 1=x$
    \item \label{prop:facts about matched pair (8)} $x\rhd 1=1$
    \item \label{prop:facts about matched pair (9)} $(x\rhd b)^{-1}=(x\lhd b)\rhd b^{-1}$
    \item \label{prop:facts about matched pair (10)} $(x\lhd b)^{-1}=x^{-1}\lhd (x\rhd b)$
\end{enumerate}
\end{minipage}
\end{center}}
\end{proposition}

Both \cref{lem:factorizable is matched pair} and \cref{prop:facts about matched pair} are proved from the definitions using direct computation.

\begin{definition}\label{def:orbits and stabs}
{\it 
Let $L=FG$ be a factorizable group. Then for $x\in G$ and $b\in F$, we define the following:
\begin{myenum}
    \item \label{def:orbits and stabs (1)} The orbits $\mathcal{O}_x=\{x\lhd b\,|\,b\in F\}\subset G$
    \item \label{def:orbits and stabs (2)} the stabilizers $F_x=\{b\in F\,|\, x\lhd b=x\}\subset F$
    \item \label{def:orbits and stabs (3)} The sets $F_{x^{-1},x}=\{b\in F\,|\, x^{-1}\lhd b=x\}$
\end{myenum}}
\end{definition}

\begin{lemma}[\cite{jm}, Corollary 4.3]\label{lem:orbit and stab inv equiv}
{\it Let $L=FG$. Let $x\in G$, then the following are equivalent: \begin{myenum}
    \item \label{lem:orbit and stab inv equiv (1)} $x^{-1}\in\mathcal{O}_x$
    \item \label{lem:orbit and stab inv equiv (2)} $y^{-1}\in\mathcal{O}_x$ for all $y\in\mathcal{O}_x$
    \item \label{lem:orbit and stab inv equiv (3)} $F_{x^{-1},x}\not=\varnothing$
    \item \label{lem:orbit and stab inv equiv (4)} $F_{y^{-1},y}\not=\varnothing$ for all $y\in\mathcal{O}_x$
\end{myenum}

In fact, it is always true that $|F_{x^{-1},x}|=|F_{y^{-1},y}|$ for all $y\in\mathcal{O}_x$.}
\end{lemma}

As we will see in the next section, the heart of our problem of computing second indicators will become a problem of understanding which $x\in S_{n-1}$ satisfy that $x^{-1}\in\mathcal{O}_x$.

\begin{lemma}\label{lem:sizes of stabs and inv stabs}
{\it Let $x^{-1}\in\mathcal{O}_x$. Then $|F_{y^{-1},y}|=|F_x|$ for every $y\in\mathcal{O}_x$.}
\end{lemma}

\no{\it Proof.} Since $x^{-1}\in\mathcal{O}_x$, there exists a $b\in F$ such that $x^{-1}=x\lhd b$. 

Define a function $\Lambda:F_{x^{-1},x}\to F_x$ such that $\Lambda(c)=bc$ for all$c\in F_{x^{-1},x}$. 

This function is well defined since $$x\lhd bc=(x\lhd b)\lhd c=x^{-1}\lhd c=x.$$

This map is also clearly a bijective with inverse $\Lambda^{-1}(d)=b^{-1}d$ for all $d\in F_x$. 
\qed\\

\begin{theorem}[\cite{jm}, Lemma 4.2 (4)]\label{thm:product in stab}
{\it Let $F$ be an abelian group and let $x^{-1}\in\mathcal{O}_x.$ Then for all $y\in\mathcal{O}_x$ and all $b\in F_{y^{-1},y}$, $(y^{-1}\rhd b)b\in F_x$}
\end{theorem}

The specific factorizable group of interest to us is $L=S_n$. Letting $a=(1\pp2\pp\cdots\pp n)$ be the standard $n$-cycle and considering $F=C_n=\langle a\rangle$ and $G=S_{n-1}$ we can write $L=FG$ as a factorizable group. Note that clearly $F\cap G=1$ and since $|L|=|FG|=|F||G|$ we have that $L=GF=FG$. To describe the action explicitly, we use the following theorem.

\begin{proposition}\label{prop:[M] Properties of action}
{\it For any $\lambda\in S_n,$ let $a$ be the standard $n$-cycle and write $\lambda=a^rx$ where $x=a^{-r}\lambda \in S_{n-1}$ and $r=\lambda(n)$. This is the unique way of representing $\lambda$ as a product of elements in $C_nS_{n-1}$.}
\end{proposition}

\no{\it Proof.} Clearly $C_n\cap S_{n-1}=1$ since aside from the identity, permutations in $S_{n-1}$ fix $n$ and permutations in $C_n$ do not. Furthermore, $C_nS_{n-1}\subset S_n$ and $|C_n||S_{n-1}|=n(n-1)!=n!=|S_n|$, we have that $S_n$ is a factorizable with $F=C_n$ and $G=S_{n-1}$. \qed\\

In \cref{sec:reinterpretting the action}, we will use the properties described in this section to better understand and reinterpret the group actions of $F$ and $G$ on each other. But before we can derive anything useful, we must describe the Hopf structure which is at the heart of this paper.

\section{\bf Hopf and Representation Results}\label{sec:Hopf and rep results}

Here we construct the Hopf Algebra that will be the object of study for the remainder of the paper. Throughout this section, we will take $k$ to be an algebraically closed field of characteristic $0$. When doing explicit computations, we will often take $k=\mathbb{C}$. A general resource on Hopf Algebras, \cite{mo}, was consulted often for background definitions and facts about Semisimple Hopf Algebras.

The first half of this section is largely a summary of key results of \cite{jm} since they provided the indicator formula for all semisimple Hopf algebras which are bismash products and they also classified all indicators of our main example in the prime case. We conclude that further simplifications to their formula are possible in the case where $F$ is abelian.

\begin{definition}[\cite{ta}]\label{def:Bismash Hopf algebra}
Let $(F,G,\lhd,\rhd)$ be a matched pair of finite groups. Then the bismash product $J=k^G\#kF$ has underlying set $k^G\otimes_k kF$. Let $x,y\in G$ and $a,b\in F$ and define $p_x:G\to k$ such that $$p_x(y)=\delta_{x,y}=\begin{cases} 1 & \text{ if }y=x\\ 0 &\text{ otherwise}\end{cases}$$ which is basis of $k^G$ dual to $kG$. We will typically write $p_x\otimes a=p_x\#a$ for elements of $J$.
\begin{myenum}
    \item unit: $1\# 1$
    \item multiplication: $(p_x\#a)(p_y\#b)=\delta_{y,x\lhd a}p_x\#ab$
    \item counit: $\varepsilon(p_x\#a)=\delta_{x,1}$
    \item comultiplication: $$\Delta(p_x\#a)=\sum_{y\in G}(p_{xy^{-1}}\#(y\rhd a))\otimes (p_y\#a)$$
    \item antipode: $S(p_x\#a)=p_{(x\lhd a)^{-1}}\#(x\rhd a)^{-1}$
\end{myenum}
\end{definition}
 
We will now summarize a variety of general results in representation theory which we will be using as a starting point for computing the second indicators of $J$. Additionally, we will provide results on bismash products from \cite{jm} that are on significance.

\begin{theorem}[\cite{kmm}]\label{thm:[KMM] inducing modules}
Let $J=k^G\#kF$. Let $x\in G$ have orbit $\mathcal{O}_x$ and stabilizer $F_x$. Let $V$ be an irreducible representation of $F_x$, thought of as a simple $kF_x$-module. Then $$\hat{V}=kF\otimes_{kF_x}V$$ is an irreducible $J$-module under the action of $$(p_y\#a)\cdot (b\otimes v)=\delta_{y\lhd(ab),x}(ab\otimes v)$$ for all $x\in G$, $a,b\in F$ and $v\in V$. Furthermore, every irreducible representation of $J$ is induced from an irreducible representation of subgroups of $F$ in this way.
\end{theorem}

Explicitly, if we want to construct an irreducible representation $\hat{V}$ of the Hopf algebra $J$ we follow the following steps. \begin{myenum}
    \item choose an element $x\in G$
    \item under the matched pair action of $F$ on $G$, compute the stabilizer of $x$ which is $F_x$
    \item $F_x$ is a subgroup of $F$ and so has irreducible (group) representations $V$
    \item for each of these $F_x$-group-representations $V$, there is an induced $J$-module $\hat{V}$ whose structure is described by \cref{thm:[KMM] inducing modules}.
\end{myenum}

Since all irreps of $J$ are constructed in this way, the first step in understanding the irreps of $J$ (which is ultimately how we will understand the indicators) is to understand the orbits and stabilizers of the action of $F$ on $G.$

\begin{lemma}[\cite{kmm}]\label{lem:[KMM] dimension of induced rep}
If $F$ is abelian, the dimension of $\hat{V}$ as described in \cref{thm:[KMM] inducing modules} is $[F:F_x]$.
\end{lemma}

In fact, in the original proof of \cref{thm:[KMM] inducing modules}, we know that dim$(\hat{V})=[F:F_x]$dim$(V)$. So, if $F$ is abelian, so is $F_x$ for all $x$ so $\dim(V)=1$ for all irreducible representations of $F_x$ so $\dim(\hat{V})=[F:F_x].$

We now introduce a crucial formula, the indicator formula for bimash product Hopf Algebras.

\begin{theorem}[\cite{jm}]\label{thm:[JM] indicator formula}
{\it For $J=k^G\#kF$ and $x\in G$, the indicator of $\hat{\chi}$ is given by \begin{equation}\nu(\hat\chi)=\frac{1}{|F|}\sum_{y\in\mathcal{O}_x}\sum_{a\in F_{y^{-1},y}}\hat{\chi}(p_y\#(y^{-1}\rhd a)a) \label{eq:[JM] indicator formula}\end{equation} where $F_{y^{-1},y}=\{a\in F\,|\,y^{-1}\lhd a=y\}$. Furthermore, if $\hat{V}$ is an irreducible representation, then $\nu(\hat{\chi})$ takes values in $\{1,0,-1\}$}
\end{theorem}

This formula was adapted from a formula derived by [\cite{n}, Proposition 5.5] from [\cite{kmm}, p 898]. Additionally, it was shown in \cite{lm} that the formula below--as with its group analogue--always outputs either $1,0,$ or $-1$ and these values determine if the associated module admits a nondegenerate, $J$-invariant, bilinear form (iff $\nu(\hat\chi)\not=0)$, and if it does admit such a form, if this form is symmetric (iff $\nu(\hat\chi)=1$) or skew-symmetric (iff $\nu(\hat\chi)=-1$).

\begin{theorem}[\cite{jm}]\label{thm:[JM] indicator when Fx=1}
{\it Let $L=GF=FG$ be a factorizable group and $J=k^G\#kF$. Let $x\in G$ and $F_x=\{1\}.$ Then there is a unique simple $F_x$-module $V\cong k$ and the induced simple $J$-module $\hat{V}$ has Schur indicator equal to $1$ if and only if $x^{-1}\in\mathcal{O}_x$. Otherwise the indicator is $0$.}
\end{theorem}

\begin{theorem}[\cite{jm}]\label{thm:JM indicator when Fx=F}
{\it Let $x\in G$ and $F_x=F$. Then any simple module $V$ of $kF_x$ is already an $F$-module and so $\hat{V}=V$ becomes a simple $J$-module. Then:
\begin{myenum}
    \item \label{thm:JM indicator when Fx=F (1)} if $x=1$ then the value of the Frobenius-Schur indicator of $\hat{V}$ is the same as the indicator of $V$ as a simple $kF$-module;
    \item \label{thm:JM indicator when Fx=F (2)} if $x^2\not=1$ then the value of the Frobenius-Schur indicator of $\hat{V}$ is $0$;
    \item \label{thm:JM indicator when Fx=F (3)} if $x^2=1$ (for $x\not=1$) and additionally if $F=C_{p^r}$ for $p$ and odd prime, then the Frobenius-Schur indicator of $\hat{V}$ is $1.$
\end{myenum}}
\end{theorem}

Now we describe a main result proved in \cite{jm}. This is a complete classification of the indicators of the Hopf algebra $J_p$ where $p$ is a prime.

\begin{definition}\label{def:Jn}
Let $S_n=S_{n-1}C_n=C_nS_{n-1}$ be a matched pair of finite groups. Then the bismash product $J_n=k^{S_{n-1}}\#kC_n$ where $F=C_n$ and $G=S_{n-1}$.
\end{definition}

\begin{theorem}[\cite{jm}]\label{thm:[JM] indicator Fx=F=Cp}{\it Consider $J_p=k^G\#kC_p$ where $p>2$ is prime. There are two cases: 

{\bf I: $\hat{V}$ has dimension $1$}. Here $F_x=F$, so that $x\in G^F=\{x\in G\,|\, x\lhd a=x$ for all $a\in F\}$. Then either,
\begin{enumerate}[label=\textbf{(\alph*)}]
    \item $x=1$, in which case we have the trivial module $\hat{V}_0$ with $\nu(\hat{V}_0)=1$, and $p-1$ other $1$-dim simple modules $\hat{V}$, all with indicator $0$;
    \item $x$ has order $2$, in which case we have $p$ simple modules with $\nu(\hat{V})=1$;
    \item $x$ has order $>2$, in which case we have $p$ simple modules with $\nu(\hat{V})=0.$
\end{enumerate}

{\bf II: $\hat{V}$ has dimension $p$}. Here $F_x=\{1\}$ and either
\begin{enumerate}[label=\textbf{(\alph*)}]
\setcounter{enumi}{4}
    \item $\mathcal{O}_x$ contains an element of order $2$. Then there is one simple module with $\nu(\hat{V})=1;$
    \item $\mathcal{O}_x$ contains no elements of order $2$. Then there is one simple module with $\nu(\hat{V})=0.$
\end{enumerate} }
\end{theorem}

We now cite a slightly revised theorem of \cite{jm}.

\begin{theorem}[\cite{jm}]\label{thm:[JM] counting all irrepts prime p}{\it Consider $J_p=k^G\#kC_p$ for $p>2$ prime. Let $m_{p,1}$ be the number of orbits of size $p$ (so $F_x=\{1\}$) containing at least one involution. Let $m_{p,0}$ be the number of orbits of size $p$ containing no involutions. Then \begin{myenum}
    \item \label{thm:[JM] counting all irrepts prime p (1)} $\displaystyle m_{p,1}=\frac{i_p-1}{p}-1\qquad i_p=\#\{\sigma\in S_n\,|\,\sigma^2=1\}$ 
    \item \label{thm:[JM] counting all irrepts prime p (2)} $\displaystyle m_{p,1}+m_{p,0}=(p-1)\frac{(p-2)!-1}{p}$
    \item \label{thm:[JM] counting all irrepts prime p (3)} There are $pm_{p,1}$ irreps of dimension $p$ with indicator $+1$
    \item \label{thm:[JM] counting all irrepts prime p (4)} There are $p(p-1)$ irreps of dimension $1$ and of these $p+1$ have indicator $+1$
\end{myenum}}
\end{theorem}

Let us reiterate. In their paper, Jedwab and Montgomery fully describe all indicators for the Hopf algebra $J_p=k^{S_{p-1}}\#kC_p$ where $p$ is an odd prime. Specifically, they describe completely all $x\in G=S_{p-1}$ for which an irrep $\hat{V}$ of $J_p$ induced from the (group) representation $V$ of the group $F_x=(C_p)_x$ has nonzero indicator and further show that all indicators are either $0$ or $1$. In \cref{sec:indicator formula (odd)} and \cref{sec:t=2}, we do the same for $J_n$ for $n$ odd and $n=2$ respectively.

Note that \cite{jm} show that the algebraic dual $(J_n)^*=H_n=\mathbb{C}^{C_n}\#\mathbb{C}S_{n-1}$ is totally orthogonal. That is, every simple $H_n$-module admits an $H_n$-invariant, symmetric, nondegenerate, bilinear form. However, we will see in the next section that the only $n$ for which $J_n$ is totally orthogonal is $n=2.$
 
Thus, the question of describing indicators becomes much more interesting for $J_n$.

\begin{theorem}[\cite{jm}]\label{thm:[JM] indicator of Hn}
{\it Let $H_n=\mathbb{C}^{C_n}\#\mathbb{C}S_{n-1}$ be the algebraic dual of $J_n$. Then if $\hat{V}$ is an irreducible representation of $H_n$ with character $\hat{\chi}$, $\nu(\hat{\chi})=+1.$ }
\end{theorem}

\begin{lemma}\label{lem:inv element in orbit means 0 indicator}
{\it Let $x\in S_{n-1}$. If $x^{-1}\notin\mathcal{O}_x$, then $\hat{V}$ has indicator $0$ for any representation $V$ of $F_x$.}
\end{lemma}

\no{\it Proof.} From \cref{thm:[JM] indicator formula}, if $F_{y^{-1},y}=\varnothing$, then $\nu(\hat{\chi})=0$. From \cref{lem:orbit and stab inv equiv}, if $x^{-1}\notin\mathcal{O}_x$, then $F_{y^{-1},y}=\varnothing$ for all $y\in\mathcal{O}_x$ and so $\nu(\hat{\chi})=0$. \qed\\

\cref{lem:inv element in orbit means 0 indicator} is a more general version of \cref{thm:[JM] indicator when Fx=1}.

\begin{proposition}[\cite{jm}]\label{prop:[JM] character formula}
{\it Let $L=FG$ as usual and let $V$ be an irreudcible representation of $F_x$ with character $\chi$. Then for $x\in G$, the induced $J=k^G\#kF$-module $\hat{V}$ has character $\hat{\chi}$ which is defined by $$\hat{\chi}(p_y\#a)=\sum_{\substack{b\in T\\\text{ where }b^{-1}ab\in F_x}}\delta_{y\lhd b}(x)\chi(b^{-1}ab)$$ where, the sum is taken over a set of representatives $T$ for the right cosets of $F/F_x$. }
\end{proposition}

\begin{lemma}\label{lem:character abelian case}
{\it If $F$ is abelian, then the formula for $\hat{\chi}$ From \cref{prop:[JM] character formula} reduces to $$\hat{\chi}(p_y\#a)=\begin{cases} \chi(a) & \text{ if } y\in\mathcal{O}_x \text{ and } a\in F_x\\
0 & \text{ otherwise}\end{cases}$$}
\end{lemma}

\no{\it Proof.} Fix the element $p_y\#a$. Consider the following cases.

\begin{myenum}
    \item First, if $y\notin\mathcal{O}_x$, then there does not exist any $b\in F$ for which $y\lhd b=x$. Namely, $\delta_{y\lhd b}(x)=0$ for all $b\in F$ and so $\hat{\chi}(p_y\#a)=0$.
    \item Second, if $y\in\mathcal{O}_x$ but $a\notin F_x$, then there is no $b\in F$ for which $bab^{-1}=bb^{-1}a=a\notin F_x$ ($F$ is abelian), so there is no sum and thus $\hat{\chi}(p_y\#a)=0$. 
    \item Now, if $y\in\mathcal{O}_x$ and $a\in F_x$, then by \cref{prop:[JM] character formula}, we claim that there is only one coset of $F/F_x$ for which $y\rhd b=x$.
    
    First, since $y\in\mathcal{O}_x=\{x\rhd b\,|\,b\in F\}$ there exists a $b\in F$ so $x\lhd b=y$. Namely, $x=y\lhd b^{-1}$ and so there is at least one coset ($b^{-1}F_x$) of $F/F_x$ whose representative satisfies that $\delta_{y\lhd b^{-1}}(x)=1$.
    
    Now, assume $c\in F$ is another element such that $y\lhd c=x$ (and so equivalently $x\lhd c^{-1}=y$). Then $$y=x\lhd c^{-1}=x\lhd b\implies x=x\lhd bc.$$ Thus, $bc\in F_x$ so $c\in b^{-1}F_x$ (keep in mind $F$ is abelian). 
    
    Thus, there is exactly one coset for which $\delta_{y\lhd b^{-1}}(x)\not=0$.
\end{myenum}
Finally, we obtain the result $$\hat{\chi}(p_y\#a)=\begin{cases} \chi(a) & \text{ if } y\in\mathcal{O}_x \text{ and } a\in F_x\\
0 & \text{ otherwise}\end{cases} \eqno\qed$$

\begin{theorem}\label{thm:indicator on reps of kCn}
{\it Let $C_n=\langle a\rangle$. Let $k=\mathbb{C}$ and $\zeta_n\in k$ be a primitive $n^{\text{th}}$-root of unity. Let $V_i$ be a simple $kC_n$-module equivalent to the irreducible (group) representation $\rho_i$ of $C_n$ defined by \begin{align*}
    \rho_i:C_n&\to\mathbb{C}^\times\\
    a&\mapsto \zeta_n^i.
\end{align*} Then, for any $0\le i<n$, $$\nu(V_i)=\begin{cases} +1 & \text{ if }n|2i\\
0 & \text{ otherwise} \end{cases}$$

Finally, this gives us that for $n$ even, all the representations have indicator $0$ except for exactly the trivial representation and the representation defined by $\rho_i(a)=(\zeta_n)^{\frac{n}{2}}$ both of which have indicator $+1$. 

And for $n$ odd, only the trivial representation has indicator $+1$ and the rest have indicator $0.$ }
\end{theorem}

\no{\it Proof.} Let $\chi_i$ be the character of $V_i$. First, note that $V_i\cong k$ and so writing $\rho_i(a)=\zeta_n^i$ makes sense and so Tr$(\rho_i(a))=\zeta_n^i$. 

Since $C_n=\langle a\rangle$, we can write \begin{align*}
    \nu(V_i)&=\frac{1}{|C_n|}\sum_{b\in C_n}\chi_i(b^2)=\frac{1}{|C_n|}\sum_{t=0}^{n-1}\chi_i((a^t)^2)\\
    &=\frac{1}{n}\sum_{t=0}^{n-1}\text{Tr}(\rho_i(a^{2t}))=\frac{1}{n}\sum_{t=0}^{n-1}\rho_i(a)^{2t}\\
    &=\frac{1}{n}\sum_{t=0}^{n-1}(\zeta_n^i)^{2t}=\frac{1}{n}\sum_{t=0}^{n-1}\zeta_n^{2it}\\
    &=\begin{cases}
    1 & \text{ if }\zeta_n^{2i}=1\\
    \frac{1}{n}\cdot\frac{\zeta_n^{2ni}-1}{\zeta_n^{2i}-1} & \text{ if }\zeta_n^{2i}\not=1
    \end{cases}\\
    &=\begin{cases}
    1 & \text{ if }\zeta_n^{2i}=1\\
    0 & \text{ if }\zeta_n^{2i}\not=1
    \end{cases}\\
\end{align*}  

Recall that for any $w\not=1$, $\displaystyle \frac{w^n-1}{w-1}=\sum_{t=0}^{n-1}w^t$. Since $\zeta_n$ is a primitive $n^{\text{th}}$-root of unity, $\zeta_n^{2i}=1$ if and only if $n|2i$.

Note that since $i<n$, $2i<2n$ and so $n|2i$ if and only if $n=2i$ (so $i=\frac{n}{2}$) or $i=0$. Namely, the trivial representation always has indicator $+1$, and there is a non-trivial representation with nonzero indicator if and only if $n$ is even.
\qed\\

\section{\bf Reinterpreting the Action}\label{sec:reinterpretting the action}

Consider $x\in S_n$ such that $x(n)=n$. Then, letting $C_n=\langle a\rangle$ where $a=(1\pp2\pp\cdots\pp n)$ is the standard $n$-cycle, we can multiply powers of $a$ with $x$ without issue.

Now, we note that $S_n\cong Bij(\mathbb{Z}/n\mathbb{Z})$ which is the group of bijections of $\mathbb{Z}/n\mathbb{Z}$ where permutation multiplication is now viewed as function composition. By viewing $a$ as a shifting map which sends $i$ to $i+1$, we can consider how the action of $C_n$ on $S_{n-1}$ affects each permutation $x$ pointwise.

Using this notation, we will often describe addition and multiplication of the outputs of $x$ and its shifts. This makes sense using the ring structure of $\mathbb{Z}/n\mathbb{Z}$. Unless otherwise stated, all arithmetic in this section will be done modulo $n$.

\subsection{Reinterpreting Orbits and Stabilizers}\label{sec:reinterpret orb set}

\begin{proposition}\label{prop:for python easier formulas}
{\it Let $L=FG=C_nS_{n-1}$ where $F=C_n=\langle a\rangle$ and $G=S_{n-1}$. consider the actions $\lhd$ and $\rhd$ of $F$ and $G$ on each other as from \cref{sec:Group Action Results} and let $x\in S_{n-1}$. Then,
\begin{myenum}
    \item \label{prop:for python easier formulas (1)} Let $a^r\in C_n$. Then $x\lhd a^r = a^{-x(r)}xa^r$ and $x\rhd a^r = a^{x(r)}$;
    \item \label{prop:for python easier formulas (2)} $F_x=\{a^t\,|\,a^{-x(t)}xa^t=x\}=\{a^t\,|\, x(u+t)-x(t)=x(u)$ for $u=1,...,n\}$;
    \item \label{prop:for python easier formulas (3)} $F_{x^{-1},x}=\{a^s\, |\, x^{-1}(u+s)-x^{-1}(s) =x(u) \text{ for }i=1,...,n\}$
    \item \label{prop:for python easier formulas (4)} $\mathcal{O}_x = \{a^{-x(l)}xa^l\,|\, a^l\in C_n\}$;
\end{myenum}
}
\end{proposition}

\begin{proof} Note that all arithmetic is done modulo $n$. That is, when we write $x(u+t)-x(t)=x(u)$ we really mean $(x(u+t\mod n)-x(t))\mod n=x(u)$.

\begin{myenum}
    \item \cref{prop:[M] Properties of action} tells us that $S_n$ is factorizable with subgroups $F=C_n$ and $G=S_{n-1}$ and \cref{lem:factorizable is matched pair} tells us how we can obtain a matched pair from $F$ and $G$.
\item Recall \cpref{def:orbits and stabs}{def:orbits and stabs (2)}. From (1) \begin{align*}
    F_x&=\{a^t\in C_n \, |\, x\lhd a^t=x\}\\
    &=\{a^t\in C_n\,|\,a^{-x(t)}xa^t=x\}\\
    &=\{a^t\in C_n\,|\,(a^{-x(t)}xa^t)(u)=x(u)\text{ for all }u=1,...,n\}\\
    &=\{a^t\in C_n \, |\, x(u+t)-x(t)=x(u)\text{ for all}u=1,...,n\}\
\end{align*}
\item Recall \cpref{def:orbits and stabs}{def:orbits and stabs (3)}. Then, by similar reasoning as (2), \begin{align*}
    F_{x^{-1},x}&=\{a^s\in C_n \, |\, x^{-1}\lhd a^s=x\}\\
    &=\{a^s\in C_n\,|\,a^{-x^{-1}(s)}x^{-1}a^s=x\}\\
    &=\{a^s\in C_n\,|\,(a^{-x^{-1}(s)}x^{-1}a^s)(u)=x(u)\text{ for all }u=1,...,n\}\\
    &=\{a^s\in C_n \, |\, x^{-1}(u+s)-x^{-1}(s)=x(u)\text{ for all}u=1,...,n\}\
\end{align*}
\item Recall \cpref{def:orbits and stabs}{def:orbits and stabs (1)}. Again from (1) \begin{align*}
    \mathcal{O}_x&=\{x\lhd a^l\,|\, a^l\in C_n\}\\
    &=\{a^{-x(l)}xa^l\,|\, \text{ for }l=1,...,n\}\qedhere
\end{align*}
\end{myenum}
\end{proof}

Let us consider an example.

\begin{example}\label{ex rewriting permutation as a bijection}$\,$

Let $n=6$ and $x=(1\pp 2)(3\pp4\pp5)$. Then, using permutation multiplication we get that $x\lhd a=a^{-x(1)}xa=a^{-2}xa=a^4xa$. Since $a=(1\pp2\pp3\pp4\pp5\pp6)$, direct computation by multiplying cycles gives that $x\lhd a=(1\pp5\pp4)$.

However, instead of multiplying permutations, this result can also be seen by understanding $x\lhd a$ pointwise. That is, using ideas from the proof of \cpref{prop:for python easier formulas}{prop:for python easier formulas (1)}, $$(x\lhd a)(i)=(a^4xa)(i)=x(i+1)+4.$$ This gives us a formula to write out the product as a bijection
\begin{center}
    \begin{minipage}{0.3\textwidth}
    \begin{align*}
    1&\mapsto x(2)+4=5\\
    2&\mapsto x(3)+4=2\\
    3&\mapsto x(4)+4=3\\
\end{align*}
\end{minipage}\hspace{0.5cm}\begin{minipage}{0.3\textwidth}
    \begin{align*}
    4&\mapsto x(5)+4=1\\
    5&\mapsto x(6)+4=4\\
    6&\mapsto x(1)+4=6\\
\end{align*}
\end{minipage}
\end{center} which is, of course, the same permutaiton we described via direct computation.

\end{example}

\begin{proposition}\label{prop:equivalence of sets}
{\it Let $x\in S_{n-1}.$ Let $F_x=\langle a^t\rangle$. Then \begin{myenum}
    \item \label{prop:equivalence of sets (1)} $a^{-x(l)}xa^l=a^{-x(s)}xa^s$ if and only if $x(s-l)=x(s)-x(l)$ and $s=l\mod t$ 
    \item \label{prop:equivalence of sets (2)} $\mathcal{O}_x=\{a^{-x(l)}xa^l\,|\,l=1,...,t\}$ and so if $y=a^{-x(l)}xa^l\in\mathcal{O}_x$, then $y(u)=x(u+l)-x(l)$ for all $u.$ Specifically, $F_y=F_x$ and $y(t)=x(t).$
    \item \label{prop:equivalence of sets (33)} If $x^{-1}=a^{-x(s)}xa^s\in\mathcal{O}_x$. Then $x(s)+s$ is a fixed point of $x$ and $x(s)+s=ut$ for some $u$.
    \item \label{prop:equivalence of sets (44)} If $x^{-1}=a^{-x(s)}xa^s\in\mathcal{O}_x$, and we let $y\in\mathcal{O}_x$ with $y=a^{-x(l)}xa^l$. Then $$F_{y^{-1},y}=\{a^{mt+l-s-x(l)}\,|\,m=0,1,...,\frac{n}{t}-1\}$$
\end{myenum}}
\end{proposition}

\no{\it Proof.} First, note that \cpref{prop:for python easier formulas}{prop:for python easier formulas (2)} tells us that if $a^t\in F_x$, then $x(t+u)=x(t)+x(u)$ for all $u.$

\begin{myenum}
    \item \begin{enumerate}
        \item[$\implies$] Let $a^{-x(l)}xa^l=a^{-x(s)}xa^s$ for some $l$ and $s$. Then $x=a^{x(l)-x(s)}xa^{s-l}$. Now, $$x(t)=(a^{x(l)-x(s)}xa^{s-l})(t)=x(t+s-l)+x(l)-x(s)=x(t)+x(s-l)+x(l)-x(s)$$ and so $$-x(s-l)=x(l)-x(s).$$ Thus, \begin{align*}
    x&=a^{x(l)-x(s)}xa^{s-l}\\
    &=a^{-x(s-l)}xa^{s-l}
\end{align*} so $a^{s-l}\in F_x$ so there exists a $p$ such that $s-l=pt$ so $s=l\mod t$.
\item[$\impliedby$] Assume there exists $s,l$ so $x(s-l)=x(s)-x(l)$ and $s=l\mod t$. Let $s-l=pt$ for some $p$. Then, since $a^{-pt}=(a^t)^{-p}\in F_x$, \cpref{prop:for python easier formulas}{prop:for python easier formulas (2)} gives that $a^{-x(-pt)}xa^{-pt}=x$. Using the substitution $l=s-pt$, routine computation gives that $a^{-x(l)}xa^l=a^{-x(s)}xa^s$ as desired.
    \end{enumerate}

\item From \cref{prop:for python easier formulas} we know that $\mathcal{O}_x=\{a^{-x(l)}xa^l\,|\,l=1,...,n\}$. Let $y\in\mathcal{O}_x$. Then $y=a^{-x(l)}xa^l$ for some $l$. Thus, $$y(t)=(a^{-x(l)}xa^l)(t)=x(t+l)-x(l)=x(t)+x(l)-x(l)=x(t).$$ Furthermore,  $$a^{-y(t)}ya^t=a^{-x(t)}ya^t=a^{-x(t)}(a^{-x(l)}xa^l)a^t=a^{-x(l)}a^{-x(t)}xa^ta^l=a^{-x(l)}xa^l=y.$$ Therefore, $a^t\in F_y$ and so $F_x\subset F_y$.

However, $\mathcal{O}_x=\mathcal{O}_y$ by definition of group actions the choice of representative does not change the orbit and so this shows that stabilizers are contained within each other. Thus, $F_x=F_y$ for all $y\in\mathcal{O}_x.$

Finally, let $$\mathcal{O}_x=\{x,y_1,y_2,...,y_{n-1}\}$$ where $y_l=a^{-x(l)}xa^l$. Now, since the orbit stabilizer theorem tells us that $|\mathcal{O}_x|=\frac{|F|}{|F_x|}=\frac{n}{\frac{n}{t}}=t$ If $n=t$, then there is nothing to show. However, if $t<n$, then consider $p=ut+q$ for some $q<t$. Then inductively, $x(ut+q)=x(ut)+x(q)$ and so $x(ut+q-q)=x(ut)=x(ut+q)-x(q)$. Clearly $ut+q=q\mod t$ and so (1) gives us that $y_p=y_q$. Therefore, $$\mathcal{O}_x=\{x,y_1,y_2,...,y_{n-1}\}=\{x,y_1,...,y_{t-1},x,y_1,... y_{t-1},...\}=\{x,y_1,...,y_{t-1}\}.$$ That is, $\mathcal{O}_x=\{a^{-x(l)}xa^l\,|\,l=1,...,t\}$ 

\item Let $x^{-1}=a^{-x(s)}xa^s$ for some $s\le t$. Then we have that $$x^{-1}(u)=x(u+s)-x(s)\qquad\text{ for all }u.$$ Namely, $x(x(l)+s)=x^{-1}(x(l))+x(s)=l+x(s)$ and so Then $$s=x^{-1}(x(s))=x(x(s)+s)-x(s)\implies x(x(s)+s)=x(s)+s$$ so $x(s)+s$ is a fixed point of $x.$ Furthermore, \begin{align*}
    a^{-x(x(s)+s)}xa^{x(s)+s}&=a^{-(x(s)-s)}xa^{x(s)+s}=a^{-s}(a^{-x(s)}xa^s)a^{x(s)}\\
    &=a^{-s}x^{-1}a^{x(s)}=(a^{-x(s)}xa^s)^{-1}\\
    &=(x^{-1})^{-1}=x
\end{align*} so $a^{x(s)+s}\in F_x$ so $x(s)+s=ut$ for some $1\le u\le \frac{n}{t}$.

\item Let $x^{-1}=a^{-x(s)}xa^s$ and $y\in\mathcal{O}_x$ with $y=a^{-x(l)}xa^l$. Then \begin{align*}
    y^{-1}&=(a^{-x(l)}x^l)^{-1}=a^{-l}x^{-1}a^{x(l)}\\
    &=a^{-l}a^{-x(s)}xa^sa^{x(l)}=a^{-l-x(s)}xa^{x(l)+s}
\end{align*}

Note that (3) gives that $y^{-1}=a^{-x(x(l)+s)}xa^{x(l)+s}.$

Next, let $a^r\in F_{y^{-1},y}$. Note that since we are assuming that $x^{-1}\in\mathcal{O}_x$, we know that $F_{y^{-1},y}$ is nonempty and has the same size as $F_x$ from \cref{lem:orbit and stab inv equiv}.

Then, \begin{align*}
    y&=a^{-x(l)}xa^l=y^{-1}\lhd a^r\\
    &=a^{-y^{-1}(r)}y^{-1}a^r=a^{-y^{-1}(r)}(a^{-x(x(l)+s)}xa^{x(l)+s})a^r\\
    &=a^{-x(r+x(l)+s)+x(x(l)+s)}a^{-x(x(l)+s)}xa^{x(l)+s}a^r=a^{-x(x(l)+s+r)}xa^{x(l)+s+r}\\
\end{align*}

Thus, we have that $$a^{-x(l)}xa^l=a^{-x(x(l)+s+r)}xa^{x(l)+s+r}$$ and so from (1), there exists a $p$ such that $x(l)+s+r-l=pt$. Namely, $a^r=a^{pt+l-x(l)-s}$.

Conversely, fix $1\le m\le \frac{n}{t}$ and let $r=mt+l-x(l)-s$. Then $$y^{-1}(r)=y^{-1}(mt+l-x(l)-s)=x(mt+l)-l-x(s)=mx(t)+x(l)-l-x(s)$$ and  \begin{align*}
    y^{-1}\lhd a^r&=a^{-y^{-1}(r)}y^{-1}a^r=a^{-y^{-1}(r)}(a^{-l-x(s)}xa^{x(l)+s})a^r\\
    &=a^{-mx(t)-x(l)+l+x(s)}a^{-l-x(s)}xa^{x(l)+s}a^r=a^{-x(l)-x(mt)}xa^{x(l)+s+r}\\
    &=a^{-x(l)}a^{-x(mt)}xa^{mt+l}=a^{-x(l)}xa^l\\
    &=y
\end{align*} and so $a^{mt+l-x(l)-s}\in F_{y^{-1},y}$ for all $m.$

Therefore, $$F_{y^{-1},y}=\{a^{mt+l-x(l)-s}\,|\,m=1,2,...,\frac{n}{t}\}.$$ Since we are taking all integers modulo $n$, we could equivalently consider $u=0,...,\frac{n}{t}-1$, which will be convenient to us later.\qed

\end{myenum}

\subsection{Uncovering New Permutations Properties}\label{sec:unconvered properties}

Now, using our new notation, we will explore exactly how much we can obtain from this shifting. It turns out the answer is quite a lot! First, we dive into understanding what info we can obtain from the stabilizer of $F_x\subset F=C_n$. Because the stabilizer is a subgroup and $C_n$ is cyclic, we can write $F_x=\langle a^t\rangle$ for some value $t$ dividing $n$. Throughout, we will refer to $t$ often and it will always be reserved for the (smallest) power of $a$ which generates the stabilizer of a given $x$.

\begin{lemma}\label{lem:x(t)=kt some k}
Let $x\in S_{n-1}$ and $F_x=\langle a^t\rangle$ where $a=(1\pp2\pp\cdots\pp n)$. Then,
\begin{myenum}
    \item \label{lem:x(t)=kt some k (1)} $x(ut+q)=ux(t)+x(q)$ for all $u\in \mathbb{Z}/n\mathbb{Z}$ and $q<t.$
    \item \label{lem:x(t)=kt some k (2)} $x(t)=jt$ where $1\le j< \frac{n}{t}$ is coprime to $\frac{n}{t}$
    \item \label{lem:x(t)=kt some k (33)} $F_{x^{-1}}=F_x$
    \item \label{lem:x(t)=kt some k (44)} If $x(t)=jt$ then $x^{-1}(t)=j^{-1}t$ where $j^{-1}$ is the multiplicative inverse of $j$ modulo $\frac{n}{t}$.
    \item \label{lem:x(t)=kt some k (55)} if $x^{-1}\in\mathcal{O}_x$ and $x(t)=jt$, then $j^2= 1\mod \frac{n}{t}$.
\end{myenum}
\end{lemma}

\no{\it Proof.} 
\begin{myenum}
    \item This is induction on the statment $x(u+t)=x(u)+x(t)$ for all $t$ from \cpref{prop:for python easier formulas}{prop:for python easier formulas (2)}.
    \item Here we are assuming that $F_x=\langle a^t\rangle\subset F=\langle a\rangle$. Since $|F|=n$, and $F_x$ is a subgroup, Lagrange tells us that there exists an $m$ such that $n=mt$. Therefore by (1), $$mt=n=x(n)=x(mt)=mx(t).$$ Namely, using rules of modular arithmetic, $$mx(t)=mt\mod n\implies x(t)=t\mod\left(\frac{n}{m}=t\right).$$ Thus, $x(t)$ is some multiple of $t$.

Now, for contradiction, assume that $x(t)=jt$ where $j$ is not coprime to $\frac{n}{t}$. Then, let $u=\frac{n}{t\text{gcd}(j,\frac{n}{t})}<\frac{n}{t}$ and so $ut<n$ but $$x(ut)=ux(t)=ujt=\frac{njt}{t\text{gcd}(j,\frac{n}{t})}=n\frac{j}{\text{gcd}(j,\frac{n}{t})}=n\mod n$$ which is not possible since $x\in S_{n-1}$ and so $x(n)=n$. Namely, $x(t)=jt$ where $j$ is coprime to $\frac{n}{t}$.
\item Since \cpref{prop:for python easier formulas}{prop:for python easier formulas (1)} gives us that $x(u+t)=x(u)+x(t)$ for all $u$, inductively we get that $$x(mt+u)=x((m-1)t+u+t)=x((m-1)t+u)+x(t)=\cdots=x(u)+mx(t).$$
\item First, since $x\lhd a^t=x$, then inverting both sides and using the properties of \cref{prop:[M] Properties of action} and \cpref{prop:for python easier formulas}{prop:for python easier formulas (1)}, we obtain $$(x\lhd a^t)^{-1}=x^{-1}\lhd(x\rhd a^t)=x^{-1}\lhd a^{x(t)}=x^{-1}.$$ Namely, $a^{x(t)}\in F_{x^{-1}}$. However, $x(t)=jt$ where $j$ is coprime to $\frac{n}{t}$ by (2). Thus, $a^{x(t)}=(a^t)^j$ is a generator of $F_x$ and so $F_x\subset F_{x^{-1}}$. 

Now, the reverse. Let $F_{x^{-1}}=\langle a^s\rangle$, then $x^{-1}\lhd a^s=x^{-1}$ so again using the properties of \cref{prop:[M] Properties of action} and \cpref{prop:for python easier formulas}{prop:for python easier formulas (1)}, $$(x^{-1}\lhd a^s)^{-1}=x\lhd(x^{-1}\rhd a^s)=x\lhd a^{x^{-1}(s)}=x$$ so $a^{x^{-1}(s)}\in F_x$. But (2) tells us that $x^{-1}(s)=ls$ where $l$ is coprime to $\frac{n}{s}$ and so again, $a^{x^{-1}(s)}$ is a power of the generator which is coprime to $\frac{n}{s}=|F_{x^{-1}}|$ and so it is again a generator. Thus, $F_{x^{-1}}\subset F_x$.
\item  Note that from (2), we know that $x(t)=jt$ where $j$ is coprime to $\frac{n}{t}$. Thus, $j$ has a unique multiplicative inverse modulo $\frac{n}{t}$ which we will call $j^{-1}$. Let $x^{-1}(t)=lt$ where $l$ is coprime to $\frac{n}{t}$. Then, from (1) $$t=x^{-1}(x(t))=x^{-1}(jt)=jx^{-1}(t)=jlt\implies 1=jl\mod\frac{n}{t}.$$

Thus, because inverses (if they exist) are unique, we have that $l=j^{-1}\mod\frac{n}{t}$.
\item If there is an $s$ so $x^{-1}=a^{-x(s)}xa^s\in\mathcal{O}_x$, then $x^{-1}(u)=x(u+s)-x(s)$ for all $u=1,...,n-1$. Namely, $$x^{-1}(t)=x(t+s)-x(s)=x(s)+x(t)-x(s)=x(t).$$ That is, $t=x^2(t)=x(jt)=jx(t)=j^2t$ and so $$j^2t= t\mod n\implies j^2= 1\mod\frac{n}{t}.\eqno\qed$$
\end{myenum}

\begin{remark}\label{remark2}
{\it The converse of \cpref{lem:x(t)=kt some k}{lem:x(t)=kt some k (55)} is false. That is, knowing $x$ satisfies that $x(t)=jt$ where $j^2= 1\mod\frac{n}{t}$ does not force $x^{-1}\in\mathcal{O}_x$. }
\end{remark}

\begin{corollary}\label{lem:x deterimined by remainders mod t for Fx=<a^t>}
{\it Let $x\in S_{n-1}$. Then $a^t\in F_x$ if and only if $x(mt+u)=mx(t)+x(u)$ for all $1\le m\le \frac{n}{t}$ and $0\le u<t$.}
\end{corollary}

\no{\it Proof.} 
\begin{enumerate}
    \item[$\implies$] This follows directly from \cpref{lem:x(t)=kt some k}{lem:x(t)=kt some k (1)}.
    \item[$\impliedby$] Assume $x(mt+u)=mx(t)+x(u)$ for all $1\le m\le \frac{n}{t}$ and $0\le u<t$. Let $1\le q\le n$ and write $q=mt+u$ for some $0\le u< t$. Then, \begin{align*}
    x(q+t)-x(t)&=x(mt+u+t)-x(t)=x((m+1)t+u)-x(t)\\
    &=(m+1)x(t)+x(u)-x(t)=mx(t)+x(u)\\
    &=x(mt+u)=x(q)
\end{align*} So $a^t\in F_x$ by \cref{prop:equivalence of sets}.\qed\\
\end{enumerate}

\begin{corollary}\label{cor:x(i)=ix(1) for Fx=F}
{\it Let $x\in S_{n-1}$. Then $F_x=F=C_n=\langle a\rangle$ if and only if for all $u=1,...,n$, $x(u)=ux(1)$. }
\end{corollary}

Let us pause here to emphasize the volume of information we obtain from knowing the size of the stabilizer of $x$. \cref{lem:x(t)=kt some k} tells us that if $F_x=\langle a^t\rangle$, then $x$ is in some sense linear in $t$. That is, $x(q)$ where $q>t$ is determined by where $x$ sends $q\mod t$ and where $x$ sends $t.$

This puts a heavy restriction on what $x$ can look like. All values $x(q)$ for $q>t$ are determined from the values $x(t)$ and $x(u)$ for $u<t$. Thus, if $t$ is small relative to $n$, then the number of permutations in $S_{n-1}$ which have stabilizer $\langle a^t\rangle$ is small relative to the total $(n-1)!$ which we formalize in \cref{sec:explicitly counting}.

Using our new understanding of permutations $x$ satisfying $F_x=\langle a^t\rangle$, we can now better understand the relationship between $x$ and $y$ where $y\in\mathcal{O}_x$.

\begin{lemma}\label{lem:y in orbit x}
{\it Let $x\in S_{n-1}$, $F_x=\langle a^t \rangle$, and $x^{-1}\in\mathcal{O}_x$. Write $\mathcal{O}_x=\{x,y_1,y_2,...,y_{t-1}\}$ where $y_l=a^{-x(l)}xa^l$. Then \begin{myenum}
    \item \label{lem:y in orbit x (1)} If $x^{-1}=y_s$ for $1\le s\le t$, then $y_l^{-1}=y_{x(l)+s \text{ mod } t}$ for all $l=1,...,t.$ (Here we let $y_t=x$).
    \item \label{lem:y in orbit x (2)} for some $1\le l\le t$, $l=x(l)+s\mod t$ if and only if $y_l$ has order $2$ (Here we let $y_t=x$).
\end{myenum}}
\end{lemma}

\no{\it Proof.} 
\begin{myenum}
    \item If $x^{-1}=y_s=a^{-x(s)}xa^s$ for some $s\le t$, then \begin{align*}
        y_l^{-1}&=(a^{-x(l)}xa^l)^{-1}=a^{-l}x^{-1}a^{x(l)}\\
        &=a^{-l}a^{-x(s)}xa^sa^{x(l)}=a^{-l-x(s)}xa^{s+x(l)}
    \end{align*}
    
    Now note that $x^{-1}=a^{-x(s)}xa^s$ implies that for all $j$, $x^{-1}(j)=x(j+s)-x(s)$ and so $$-x(x(l)+s)=-(x^{-1}(x(l))+x(s))=-(l+x(s))=-l-x(s)$$ and so we obtain that $$y_l^{-1}=a^{-x(x(l)+s)}xa^{x(l)+s}=y_{x(l)+s\mod t}.$$
    \item This is immediate from (1) since $y_l$ has order two if and only if $y_l^{-1}=y_{x(l)+s}=y_l$ which happens if and only if $x(l)+s= l\mod t$.\qed\\
\end{myenum}

\subsection{A Revised Indicator Formula}\label{sec:revised indicator}

\begin{proposition}\label{prop:x(s)+s=0}
{\it Let $x\in S_{n-1},$ $F_x=\langle a^t\rangle$, $x^{-1}=a^{-x(s)}xa^s\in\mathcal{O}_x$ for some $s$, $x(t)+t=u_1t$, and $x(s)+s=u_2t$. Fix an integer $i$. If $iu_1= 0\mod \frac{n}{t}$ then 
\begin{myenum}
    \item \label{prop:x(s)+s=0 (1)} if $\frac{n}{t}$ is odd, $iu_2= 0\mod\frac{n}{t}$
    \item \label{prop:x(s)+s=0 (2)} if $\frac{n}{t}$ is even, then $iu_2= \frac{n}{2t}\text{ or }0\mod\frac{n}{t}$.
\end{myenum}}
\end{proposition}

\no{\it Proof.} First, note that since $x^{-1}(i)=x(i+s)-x(s)$, letting $i=x(s)$ we get that $$x(x(s)+s)=x^{-1}(x(s))+x(s)=s+x(s).$$

Now, \begin{align*}
    a^{-x(x(s)+s)}xa^{x(s)+s}&=a^{-(x(s)+s)}xa^{x(s)+s}=a^{-s}a^{-x(s)}xa^sa^{x(s)}\\
    &=a^{-s}x^{-1}a^{x(s)}=(a^{-x(s)}xa^s)^{-1}\\
    &=(x^{-1})^{-1}=x.
\end{align*}

Therefore, $a^{x(s)+s}\in F_x=\langle a^t\rangle$ and so $x(s)+s=u_2t$ is a multiple of $t.$

Thus, $$u_2x(t)=x(u_2t)=x(x(s)+s)=x(s)+s=u_2t$$ and so $2u_2t=u_2t+u_2t=u_2x(t)+u_2t=u_1u_2t.$ This gives us the relation that $$2u_2t= u_1u_2t\mod n\implies 2u_2= u_1u_2\mod\frac{n}{t}.$$

Therefore, $$2iu_2= iu_1u_2= 0 \mod\frac{n}{t}.$$

\begin{myenum}
    \item \no{\it Assume $\frac{n}{t}$ is odd.} Then $2$ is invertible $\mod\frac{n}{t}$ so $$iu_2= 0\mod\frac{n}{t}.$$
    \item 
\no{\it Assume $\frac{n}{t}$ is even.} Then $$iu_2= 0\mod\frac{n}{2t}.\eqno\qed$$
\end{myenum}

We are at last ready to give the key theorem so far.

\begin{theorem}\label{thm:indicator reduced}
{\it Consider $J_n=\mathbb{C}^{S_{n-1}}\#\mathbb{C}C_n$. Let $C_n=\langle a\rangle$. Fix $x\in S_{n-1}$ such that $F_x=\langle a^t\rangle$ and $x^{-1}=x\lhd a^s\in\mathcal{O}_x$. Let $\rho_{x,i}:F_x\to\mathbb{C}^\times$ be an irreducible (group) representation of $F_x$ defined by $\rho_{x,i}(a^t)=\zeta_\frac{n}{t}^i$ where $\zeta_{\frac{n}{t}}$ is a primitive $\frac{n}{t}$-root of unity. Let $\chi_{x,i}$ be the character of $\rho_{x,i}$ Then if $\hat{\chi_{x,i}}$ is the character of $J_n$ induced from $\chi_{x,i}$ (as described in \cref{thm:[KMM] inducing modules}) we obtain:
\begin{myenum}
    \item \label{thm:indicator reduced (1)} if $\frac{n}{t}$ is odd, then \begin{equation}\nu(\hat{\chi}_{x,i})=\begin{cases} 1 & \text{ if }(\zeta_{\frac{n}{t}})^{iu_1}=1\\
    0 & \text { if }(\zeta_{\frac{n}{t}})^{iu_1}\not=1
    \end{cases}\label{eq:ind when n/t odd}\end{equation}
    \item \label{thm:indicator reduced (2)} if $\frac{n}{t}$ is even, then \begin{equation}\nu(\hat{\chi}_{x,i})=\begin{cases} (\zeta_{\frac{n}{t}})^{-iu_2}=1\text{ or }-1 & \text{ if }(\zeta_{\frac{n}{t}})^{iu_1}=1\\
    0 & \text { if }(\zeta_{\frac{n}{t}})^{iu_1}\not=1
    \end{cases}\label{eq:ind when n/t even}\end{equation}
\end{myenum}  where $$x(t)+t=u_1t\qquad \text{ and }\qquad x(s)+s=u_2t.$$}
\end{theorem}

\no{\it Proof.} Let $\rho_{x,i}:F_x\to\mathbb{C}^\times$ be an irreducible (group) representation of $F_x$ defined by $\rho_{x,i}(a^t)=(\zeta_\frac{n}{t})^i$ where $\zeta_{\frac{n}{t}}$ is a primitive $\frac{n}{t}$-root of unity. Let $\chi_{x,i}$ be the character of $V_{x,i}$ which is the $F_x$-module equivalent to $\rho_{x,i}$. Then if $\hat{\chi_{x,i}}$ is the character of $J_n$ induced from $\chi_{x,i}$ (as described in \cref{thm:[KMM] inducing modules})

First, recall \cref{eq:[JM] indicator formula}. Note that \cref{thm:product in stab} tells us immediately that assuming $x^{-1}\in\mathcal{O}_x$ gives that $\hat{\chi}_{x,i}(p_y\#b)=\chi_{x,i}(b)$ and so we get $$\nu(\hat\chi_{x,i})=\frac{1}{|F|}\sum_{y\in\mathcal{O}_x}\sum_{b\in F_{y^{-1},y}}\chi_{x,i}((y^{-1}\rhd b)b)$$ where $\chi_{x,i}$ is the character of $V_{x,i}$ which is the $F_x$-module which induces $\hat{\chi}_{x,i}$.

Since $x^{-1}\in\mathcal{O}_x$, let $x^{-1}=a^{-x(s)}xa^s$ for some $s\le t$. 

Let $y_l\in\mathcal{O}_x$ with $y_l=a^{-x(l)}xa^l$. Then \cpref{prop:equivalence of sets}{prop:equivalence of sets (44)} gives that we can write $b\in F_{y^{-1},y}$ as $b=a^{mt+l-s-x(l)}$ for some $1\le m\le \frac{n}{t}$. 

Note that we have that $$y_l^{-1}=(a^{-x(l)}xa^l)^{-1}=a^{-l}x^{-1}a^{x(l)}=a^{-l}(a^{-x(s)}xa^s)a^{x(l)}$$ and so $$y_l^{-1}(mt+l-s-x(l))=x(mt+l)-x(s)-l,$$

Thus, \begin{align*}
    (y_l^{-1}\rhd b)b&=a^{y_l^{-1}(mt+l-s-x(l))}a^{mt+l-s-x(l)}=a^{x(mt+l)-x(s)-l}a^{mt+l-s-x(l)}\\
    &=a^{x(mt)+mt-x(s)-s}=a^{m(x(t)+t)-(x(s)+s)}
\end{align*}

Note that $(y^{-1}\lhd b)$ is very much dependent on both $y$ and $b$. However, $(y^{-1}\lhd b)b$ is dependent on only the choice of the $s$ for which $x^{-1}=a^{-x(s)}xa^s$.

Let $x(t)+t=u_1t$ and $x(s)+s=u_2t$. Note that both these sums are necessarily multiplies of $t$ since $x(t)$ is a multiple of $t$ from \cpref{lem:x(t)=kt some k}{lem:x(t)=kt some k (2)} and $x(s)+s$ is a multiple of $t$ from \cpref{prop:equivalence of sets}{prop:equivalence of sets (44)}.

Then, \begin{align*}
    \nu(\hat\chi_{x,i})&=\frac{1}{|F|}\sum_{y\in\mathcal{O}_x}\sum_{b\in F_{y^{-1},y}}\chi_{x,i}((y^{-1}\rhd b)b)\\
    &=\frac{1}{n}\sum_{l=1}^t\sum_{m=0}^{\frac{n}{t}-1}\chi_{x,i}((y_l^{-1}\rhd a^{mt+l-s-x(l)})a^{mt+l-s-x(l)})\\
    &=\frac{1}{n}\sum_{l=1}^t\sum_{m=0}^{\frac{n}{t}-1}\chi_{x,i}(a^{m(x(t)+t)-(x(s)+s)})\\
    &=\frac{t}{n}\sum_{m=0}^{\frac{n}{t}-1}\chi_{x,i}((a^t)^{mu_1-u_2})\\
    &=\frac{t}{n}\sum_{m=0}^{\frac{n}{t}-1}(\zeta_{\frac{n}{t}})^{i(mu_1-u_2)}\\
    &=\frac{t}{n}\sum_{m=0}^{\frac{n}{t}-1}(\zeta_{\frac{n}{t}})^{imu_1}(\zeta_{\frac{n}{t}})^{-iu_2}\\
    &=\frac{t}{n}(\zeta_{\frac{n}{t}})^{-iu_2}\begin{cases} \frac{n}{t} & \text{ if }(\zeta_{\frac{n}{t}})^{iu_1}=1\\
    \frac{(\zeta_{\frac{n}{t}})^{iu_1\frac{n}{t}}-1}{(\zeta_{\frac{n}{t}})^{iu_1}-1} & \text { if }(\zeta_{\frac{n}{t}})^{iu_1}\not=1
    \end{cases}\\
    &=\begin{cases} (\zeta_{\frac{n}{t}})^{-iu_2} & \text{ if }(\zeta_{\frac{n}{t}})^{iu_1}=1\\
    \frac{t}{n}(\zeta_{\frac{n}{t}})^{-iu_2}\frac{(\zeta_{\frac{n}{t}})^{iu_1\frac{n}{t}}-1}{(\zeta_{\frac{n}{t}})^{iu_1}-1} & \text { if }(\zeta_{\frac{n}{t}})^{iu_1}\not=1
    \end{cases}\\
    &=\begin{cases} (\zeta_{\frac{n}{t}})^{-iu_2} & \text{ if }(\zeta_{\frac{n}{t}})^{iu_1}=1\\
    0 & \text { if }(\zeta_{\frac{n}{t}})^{iu_1}\not=1
    \end{cases}
\end{align*}

with the last line reducing since $\zeta_{\frac{n}{t}}$ is a $\frac{n}{t}^{\text{th}}$-root of unity, so $(\zeta_{\frac{n}{t}})^{iu_1\frac{n}{t}}=((\zeta_{\frac{n}{t}})^{\frac{n}{t}})^{iu_1}=1$ so $(\zeta_{\frac{n}{t}})^{iu_1\frac{n}{t}}-1=0.$

\begin{myenum}
    \item Finally, \cref{prop:x(s)+s=0} gives us that $iu_2$ is a multiple of $\frac{n}{t}$ if $\frac{n}{t}$ is odd, and so $(\zeta_{\frac{n}{t}})^{iu_2}=1$ in this case.
    \item 
In the other case, $iu_2$ is a multiple of $\frac{n}{2t}$ if $\frac{n}{t}$ is even. Namely, there exists an $m$ so $iu_2=m\frac{n}{2t}$. Therefore, $(\zeta_{\frac{n}{t}})^{iu_2}=(-1)^m$ which could be $-1$ if $m$ is odd.\qed\\
\end{myenum}

Note that \cref{thm:indicator reduced} agrees completely with the results of \cite{jm}. Namely, \cref{thm:[JM] indicator when Fx=1} tells us that if $F_x=\{1\}$ then $\nu(\hat{\chi})=1$ if and only if $x^{-1}\in\mathcal{O}_x$ and is $0$ otherwise. \cref{thm:indicator reduced} says that if $F_x=\{1\}$ then $t=n$ so $\frac{n}{t}=1$ which is odd. There is then one (group) irrep of $F_x$, the trivial (group) irrep and so the indicator will be $1$ if $x^{-1}\in\mathcal{O}_x$ and $0$ otherwise.  

\cpref{thm:JM indicator when Fx=F}{thm:JM indicator when Fx=F (2)} says that if $F_x=F$ and $x^2\not=1$ then all indicators are $0.$ \cref{thm:indicator reduced} says that if $F_x=F$ then $t=1$ and so $\frac{n}{t}=n$. Since $|\mathcal{O}_x|=t=1$, if $x^2\not=1$ then $x^{-1}\not\in\mathcal{O}_x=\{x\}$ and so all indicators are $0$.

\cpref{thm:JM indicator when Fx=F}{thm:JM indicator when Fx=F (1)} tells us that if $x=1$, then the indicator of \cite{jm} and the Frobenius-Schur group indicator agree. \cref{thm:indicator reduced} states that if $x=1$ then $x^{-1}=x\lhd a^t$ and so $s=t$. Namely, $(\zeta_n)^{iu_1}=(\zeta_n)^{2i}$ and implies $(\zeta_n)^{-iu_2}=(\zeta_n)^{-2i}$. Thus, whether $n$ is even or odd, if $\zeta_n^{2i}=1$ then $\zeta_n^{-2i}=1$ and so the indicator $\nu(\hat{\chi}_{x,i})=1$ if and only if $n$ divides $2i$ and $0$ otherwise. This of course agrees with the indicator $\nu(\chi_{x,i})$ by \cref{thm:indicator on reps of kCn} and so the indicators are the same.

Finally, we will see in the later sections \cref{thm:indicator reduced} agrees with the work of \cite{jm} when $n=p^r$ for $p$ and odd prime and $r>0$.

\begin{corollary}\label{cor:n odd all ind nonnegative}{\it 
If $n$ is odd, then all indicators of $J_n$ are nonnegative.}
\end{corollary}

An important question is whether the indicator is ever negative. Recall that for $p$ prime, $J_p$ has nonnegative indicator. However, the answer is yes, there exists even $n$ for which $J_n$ has negative indicator.

\begin{example}\label{ex:ind -1}
{\it Let $n=16$, one can show the element $$x=(1\pp 5 \pp 9\pp13)(3\pp 7\pp11\pp15)$$ has stabilizer $F_x=\langle a^2\rangle$ and so $t=2.$ 

Now, let $\zeta_{\frac{n}{t}}=e^{\frac{2\pi i}{8}}\in\mathbb{C}^\times$ be a primitive $8^{\text{th}}$-root of unity. 

Then, $\frac{n}{t}=8$ is certainly even and by theorem, \cref{thm:indicator reduced},  $$\nu(\hat{\chi})=\begin{cases} \zeta_8^{-3i} & \text{ if }\zeta_8^{2i}=1\\
    0 & \text { if }\zeta_8^{2i}\not=1
    \end{cases}$$
    
    Finally, consider the irreducible (group) representation of $F_x$, $\rho_4$ defined by $\rho_4(a^2)=\zeta_8^4$.

That is, we can let $i=4$. Then $\zeta_8^{2i}=\zeta_8^8=1$ and so we have that $\zeta_8^4=e^{4\frac{2\pi i}{8}}=e^{\pi i}=-1$ and so $$\nu(\hat{\chi})=\zeta_8^{-3i}=\zeta_8^{-12}=(\zeta_8^4)^{-3}=(-1)^{-3}=-1.$$}
\end{example}
 
In fact, we will see that there exists irreps of $J_n$ which have negative indicator for any $n\ge12$ with $4|n$.

Before we conclude this section, we provide two important results that we will use in later sections.

\begin{proposition}\label{prop:observations}
{\it Let $x\in S_{n-1},$ $F_x=\langle a^t\rangle$, $x^{-1}=a^{-x(s)}xa^s\in\mathcal{O}_x$, and $x(t)=jt$ for some $j$ coprime to $\frac{n}{t}$. Let $\hat{\chi}$ be the character induced from an irreducible representation of $F_x$. Then,
\begin{myenum}
    \item \label{prop:observations (1)} Permutations in the same orbit have the same indicators. That is, if $y\in\mathcal{O}_x$ then for a given representation, the induced characters $\hat{\chi}_x$ and $\hat{\chi}_y$ have the same indicator. (This is true even if $x^{-1}\not\in\mathcal{O}_x$). 
    \item \label{prop:observations (2)} If $\mathcal{O}_x$ contains an element of order $2$, then $\nu(\hat{\chi})\ge0$
    \item \label{prop:observations (3)} If $t$ is odd, then $\nu(\hat{\chi})\ge0$.
    \item \label{prop:observations (4)} If $n=2m$ for $m$ odd then $\nu(\hat{\chi})\ge0$.
\end{myenum}}
\end{proposition}

\no{\it Proof.} 
\begin{myenum}
    \item Note \cpref{prop:equivalence of sets}{prop:equivalence of sets (2)} tells us that if $y\in\mathcal{O}_x$ then $F_x=F_y$. Thus, any (group) irrep of $F_x$ is equivalent to a group irrep of $F_y$. However, \cref{thm:indicator reduced} uses values dependent on the outputs of $x$ and so we must show that in any case, the values used to compute the indicator in \cref{thm:indicator reduced} are the same. Let $y\in\mathcal{O}_x$.
    \begin{myenum}
        \item  First, \cpref{lem:orbit and stab inv equiv}{lem:orbit and stab inv equiv (2)} tells us that $x^{-1}\not\in\mathcal{O}_x$ if and only if $y^{-1}\not\in\mathcal{O}_y=\mathcal{O}_x$ and so if $x^{-1}\not\in\mathcal{O}_x$ then $y^{-1}\not\in\mathcal{O}_x$ so all indicators are $0$ and certainly they are all the same.
        \item Now, assume $x^{-1}\in\mathcal{O}_x$. Then $y^{-1}\in\mathcal{O}_y=\mathcal{O}_x$. Then, \cpref{prop:equivalence of sets}{prop:equivalence of sets (2)} tells us that $x(t)=y(t)$.
        \begin{myenum}
            \item If $\frac{n}{t}$ is odd, then $x(t)+t=y(t)+t$ and since this is the only value used in theorem \cref{thm:indicator reduced} when computing the indicator, we have that the set of indicators induced from $y$ is the same as the set of indicators induced from $x$ and that they are induced from the same irreducible group representations of $F_x=F_y$.
            \item If $\frac{n}{t}$ is even, then let $s\le t$ be such that $x^{-1}=a^{-x(s)}xa^s$ and let $y=y_l=a^{-x(l)}xa^l\in\mathcal{O}_x.$ Then \cpref{lem:y in orbit x}{lem:y in orbit x (1)} gives that $y^{-1}=a^{-x((x(l)+s)\mod t)}xa^{(x(l)+s)\mod t}$. Let $1\le q\le\frac{n}{t}$ and $1\le p\le t$ be such that $x(l)+s=qt+p$. Then, $y(p-l)=x(p)-x(l)$ and so $$y^{-1}=a^{-x(p)}xa^p=a^{-x(p)}(a^{x(l)}ya^{-l})a^p=a^{-(x(p)-x(l))}ya^{p-l}=a^{-y(p-l)}ya^{p-l}$$ and so $y^{-1}=a^{-y(p-l)}ya^{p-l}\in\mathcal{O}_y$. Finally, $l=x^{-1}(x(l))=x(x(l)+s)-x(s)$ and so \begin{align*}
    y(p-l)+p-l&=x(p)-x(l)+p-l\\
    &=x(x(l)+s-qt)-x(l)+x(l)+s-qt-l\\
    &=x(x(l)+s)-x(qt)+s-qt-l\\
    &=l+x(s)+s-q(x(t)+t)-l\\
    &=x(s)+s-q(x(t)+t)
\end{align*}

Fix $i$ and consider the irreducible group representation $\rho_{x,i}=\rho_{y,i}=\rho_i$ of $F_x$ defined by $\rho_i(a^t)=(\zeta_\frac{n}{t})^i$. Then, $$(\zeta_{\frac{n}{t}})^{-i\frac{y(p-l)+p-l}{t}}=(\zeta_{\frac{n}{t}})^{-i\frac{x(s)+s-q(x(t)+t)}{t}}=(\zeta_{\frac{n}{t}})^{-i\frac{x(s)+s}{t}}((\zeta_{\frac{n}{t}})^{i\frac{x(t)+t}{t}})^q.$$

Namely, either the indicator of the associated character $\nu(\hat{\chi}_i)=0$ or, if $(\zeta_\frac{n}{t})^{i\frac{x(t)+t}{t}}=1$, then $$(\zeta_{\frac{n}{t}})^{-i\frac{y(p-l)+p-l}{t}}=(\zeta_{\frac{n}{t}})^{-i\frac{x(s)+s}{t}}$$ and so in either case, the indicator of $\nu(\hat{\chi}_{x,i})=\nu(\hat{\chi}_{y,i}).$ 
        \end{myenum}
    \end{myenum}
\item If $y\in\mathcal{O}_x$ has order $2$, then $s=t$ and so if $(\zeta_{\frac{n}{t}})^{i\frac{y(t)+t}{t}}=1$ then clearly $(\zeta_{\frac{n}{t}})^{-i\frac{y(s)+s}{t}}=((\zeta_{\frac{n}{t}})^{i\frac{y(t)+t}{t}})^{-1}=1^{-1}=1$ as well. By (1), the indicator of the representation induced by $y$ is the same as that of $x$ and so the indicators are all nonnegative.
\item If $t$ is odd, then $|\mathcal{O}_x|=\frac{|F|}{|F_x|}=\frac{n}{\frac{n}{t}}=t$ and $x^{-1}\in\mathcal{O}_x$, then every $y\in\mathcal{O}_x$ also has its inverse in $\mathcal{O}_x$ and so $\mathcal{O}_x$ contains an involution. By (2), all indicators are nonnegative.
\item If $n=2m$ where $m$ is odd, then either $t$ is odd, in which case (3) tells us $\nu(\hat{\chi})\ge0$, or $\frac{n}{t}$ is odd, in which case \cref{thm:indicator reduced} tells us $\nu(\hat{\chi})\ge0.$\qed\\
\end{myenum}

\begin{proposition}\label{prop:totally orthogonal}
{\it The only $n$ for which $J_n$ is totally orthogonal is $n=2.$}
\end{proposition}

\no{\it Proof.} 
\begin{myenum}
    \item $J_n$ can only be defined for $n\ge 2$ since its structure is derived from $S_{n-1}$ and $S_m$ is not defined for $m<1$. Letting $n=2$, gives that $S_1=\{(1)\}$ so the only possible permutation is the identity $x=(1)$ and so $F_x=F$. In this case, we have that $t=1$. Now, we just plug into \cref{thm:indicator reduced}. $u_1=\frac{x(1)+1}{1}=2$ so if $\zeta_2$ is a primitive square root of unity, then $\zeta_2^{2i}=1$ for all $i=0,1$, we have that $\nu(\hat{\chi}_{(1),i})=+1$.
    \item Now, let $n>2$. Again, consider the identity permutation $(1)\in S_{n-1}$. Then of course, $F_x=F$ again, and $u_1=2$, but now $\zeta_n^{2i}=1$ if and only if $2i=0\mod n$. If $n$ is odd, $i=0$ is the only solution, and if $n$ is even then $i=0,\frac{n}{2}$ are both solutions. Since $n>2$, there exists $0\le i\le n-1$ such that $2i\not=0$ and so for these choices, $\nu(\hat{\chi}_{(1),i})=0$. Therefore, not all the indicators are $+1$ and so $J_n$ is not totally orthogonal for any $n>2.$ \qed\\
\end{myenum}

\section{\bf Explicitly Counting}\label{sec:explicitly counting}

Of course, the goal for the project is to try and count explicitly the number of irreps of $J_n$ which have indicator $+1,-1$, and $0$ for a given dimension. This result is very difficult to obtain directly since there are $(n-1)!$ permutations and we must check the stabilizer of each permutation in order to determine how many associated representations are induced from $x$ (see \cref{thm:[KMM] inducing modules}). However, using our new understanding of the group actions which describe the structure of $J_n$ and our new reduced formulas from the previous section, we are able to remove a lot of extraneous computation.

In this section, we show explicitly how to describe all irreps of $J_n$ of dimension $t$ if $t$ is odd. Not only do we present formulas for counting the total number of such irreps, but we describe explicitly how to construct the set of permutations $x$ which we will later show produce nonzero indicators and what these irreps look like. 

Throughout this section, all arithmetic will be done modulo $n$ unless otherwise specified. For example, when we write $x(2t)=2x(t)$ we mean $x(2t\mod n)=2x(t)\mod n.$

For us, an involution is any permutation $x\in S_{n-1}$ which is its own inverse (including the identity).

Additionally, we will frequently invoke a result from elementary number theory: The equation $ax= b\mod n$ has either no solutions, or gcd$(a,n)$ solutions. Specifically, the equation $ax= 0\mod n$ always has gcd$(a,n)$ solutions.

\subsection{Structure of Permutations with Given Stabilizer}\label{sec:perms from stab}

As we have seen from the previous sections, if $t$ is odd, then all $t$-dimensional irreps of $J_n$ have nonnegative indicators (\cref{prop:observations}). Furthermore, the only way an (odd) $t$-dimensional irrep of $J_n$ can have an indicator of $+1$ is if the permutation $x$ which induces that irrep shares its orbit with an involution.

Therefore, describing the odd dimensional irreps as having indicator either $+1$ or $0$ is equivalent to understanding and counting the orbits under the matched pair action which contain involutions.

\begin{proposition}\label{prop:number of x with Fx=<a^t>}
{\it Let $M_{n/t}$ be the set of $x\in S_{n-1}$ such that $F_x=\langle a^t\rangle$. Let $\varphi(n)$ be the Euler Totient Function which is the number of integers coprime to $n$. Then \begin{align}
    |M_{n/1}|&=\varphi(n)\label{eq:M1}\\
    |M_{n/t}|&=\varphi\left(\frac{n}{t}\right)\left(\frac{n}{t}\right)^{t-1}(t-1)!-\sum_{\substack{s|t\\ s\not=t}}|M_{n/s}|\label{eq:Mt}
\end{align} }
\end{proposition}

\begin{proof} Let $\varphi(n)=\{1\le j\le n\,|\,j\text{ coprime to }n\}$ be the Euler Totient function.

The proof is simple counting. \cref{cor:x(i)=ix(1) for Fx=F} tells us that $F_x=F$ if and only if $x(u)=ux(1)$, \cref{lem:x(t)=kt some k} says that $x(1)=j$ for some $j$ coprime to $n$, and \cref{lem:x deterimined by remainders mod t for Fx=<a^t>} states that $x$ is determined by where it sends $1$. Thus, the number of permutations $x\in S_{n-1}$ with full stabilizer is given by the number of choices of $x(1)$, which is the number of integers coprime to $n.$ Therefore, $|M_{n/1}|=\varphi(n)$. 

Now, let $t>1$ be a divisor of $n.$ Again, \cref{lem:x deterimined by remainders mod t for Fx=<a^t>} tell us that if we want to count all permutations $x\in S_{n-1}$ with stabilizer $F_x=\langle a^t\rangle$, then to fully determine $x$, we need only understand where it sends $1,2,3,...,t$. Since $x(t)=jt$ where $j$ is coprime to $\frac{n}{t}$ by \cref{lem:x(t)=kt some k}, we know there are $\varphi(\frac{n}{t})$ choices for where $x$ can send $t.$ 

Now, where can $x$ send $1$? Well it certainly cannot send it to any multiple of $t$ since then $x$ would no longer be a bijection. Thus, it can send $1$ to any number $u$ where $1\le u\le n-1$ and $u$ is not a multiple of $t$. There are $n-\frac{n}{t}=(t-1)\frac{n}{t}$ such options. 

Recall that once we know $x(1)$, we also know $x(1+t)$, $x(1+2t)$,...,$x(1+(\frac{n}{t}-1)t)$. Namely, once we choose $x(1)$ we have fixed $\frac{n}{t}$ outputs.

Similarly, $x$ can send $2$ to any available value. Since all multiples of $t$ and now all values of the form $qt+x(1)$ for $1\le q\le\frac{n}{t}$ are taken. So there are $n-2\frac{n}{t}=(t-2)\frac{n}{t}$ possible options for $x(2)$.

Inductively we obtain: $$\underbrace{\varphi\left(\frac{n}{t}\right)}_{\text{choices for }x(t)}\cdot\underbrace{(t-1)\frac{n}{t}}_{\text{choices for }x(1)}\cdot\underbrace{(t-2)\frac{n}{t}}_{\text{choices for }x(2)}\cdots \underbrace{\frac{n}{t}}_{\text{choices for }x(t-1)}.$$

Now, define $x(qt+u)=qx(t)+x(u)$ and because this is a necessary condition we have that $a^t\in F_x$ by \cref{lem:x deterimined by remainders mod t for Fx=<a^t>}. 

However, we will have over counted here. Because this now counts all $x\in S_{n-1}$ for which $a^t\in F_x$, and it is not necessarily true that $F_x=\langle a^t\rangle$ for each of these permutations. Namely, we must now remove all $x$ for which $F_x=\langle a^s\rangle$ where $s$ divides $t$. That is, we have over counted by exactly $|M_{n/s}|$ for each $s$ which is a divisor of $t$. 

Finally, this gives the desired result, \begin{align*}
    |M_{n/t}|&=\varphi\left(\frac{n}{t}\right)(t-1)\frac{n}{t}(t-2)\frac{n}{t}\cdots \frac{n}{t}-\sum_{\substack{s | t \\ s\not= t}}M_{n/s}\\
    &=\varphi\left(\frac{n}{t}\right)\left(\frac{n}{t}\right)^{t-1}(t-1)!-\sum_{\substack{s|t\\ s\not=t}}|M_{n/s}|\qedhere
\end{align*}
\end{proof}

\begin{remark}\label{remark3 counting M}
{\it For example, when $n=12$ \begin{align*}
    |M_{12/1}|&=4\\
    |M_{12/2}|&=2(6)1!-4=8\\
    |M_{12/3}|&=2(4)^22!-4=60\\
    |M_{12/6}|&=1(2)^55!-4-8-60=3,768\\
    |M_{12/12}|&=1(1)^99!-4-8-60-3768=359,040
\end{align*}}
\end{remark}

In fact, \cref{prop:number of x with Fx=<a^t>} gives a recursive formula for explicitly constructing all permutaitons $x$ which have stabilizer $F_x=\langle a^t\rangle$ for a given $t$.

Using this formula, we will show in \cref{sec:limit behavior} that for a fixed $t$, the ratio $|M_{n/t}|/(n-1)!\to0$ as $n\to\infty.$ This shows that it is rare for a permutation to have stabilizer $\langle a^t\rangle$ for large $n.$

Now, let us describe explicitly how to write $x$ in terms of a permutation in a smaller symmetric group.

\begin{lemma}\label{lem:x generated from smaller sigma}
{\it Let $x\in S_{n-1}$ with $F_x=\langle a^t\rangle$. Then there exists an integer $j$ which is coprime to $\frac{n}{t}$ (and so therefore invertible modulo $\frac{n}{t}$), a permutation $\sigma_x\in S_{t-1}$ and integers $1\le u_i\le \frac{n}{t}$ such that \begin{align*}
    x(i)&=u_it+\sigma_x(i)\mod n\qquad \text{ for all }1\le i\le t-1\\
    x(qt)&=qjt\mod n\qquad\text{ for all } 1\le q\le \frac{n}{t}\\
    x(qt+i)&=qx(t)+x(i) \mod n
\end{align*}

Furthermore, \begin{align*}
    x^{-1}(i)&=-ju_{\sigma_x^{-1}(i)}t+\sigma_x^{-1}(i)\mod n\qquad \text{ for all }1\le i\le t-1\\
    x^{-1}(t)&=j^{-1}t\mod n\\
    x^{-1}(qt+i)&=qx^{-1}(t)+x^{-1}(i)\mod n
\end{align*}

We call $\sigma_x$ a \textit{remainder permutation} for $x$.}
\end{lemma}

\no{\it Proof.} Let $x\in S_{n-1}$ with stabilizer $F_x=\langle a^t\rangle$. We have already established in \cref{lem:x deterimined by remainders mod t for Fx=<a^t>} that $x$ is determined by where it sends $1,...,t$. We now claim that $x$ defines a permutation modulo $t$. From, \cref{lem:x(t)=kt some k}, there exists a $j$ coprime to $\frac{n}{t}$ (and so invertible modulo $\frac{n}{t}$) with $x(t)=jt.$ Now, if $i,i'<t$ and $x(i)=x(i')\mod t$, then $x(i)-x(i')$ is a multiple of $t$, say $x(i)-x(i')=qt$ for some value $q$. Then $$x(j^{-1}qt+i')=x(j^{-1}qt)+x(i')=jj^{-1}qt+x(i')=qt+x(i')=x(i)-x(i')+x(i')=x(i).$$ Therefore, because $x$ is a bijection, $j^{-1}qt+i'=i\mod n$ and so $i-i'$ is a multiple of $t$. However, we assumed that both $i$ and $i'$ were positive integers smaller than $t$ so their difference is only a multiple of $t$ if their difference is $0$. Thus, $x\mod t$ actually defines a permutation in $S_{t-1}$.

Specifically, for each $x$, there exists an associated remainder permutation $\sigma_x\in S_{t-1}$ defined by $\sigma_x(i):=x(i)\mod t$ for all $i<t$ and a value $j$ which is coprime to $\frac{n}{t}$ such that $x(t)=jt.$ Furthermore, $x(i)=tu_i+\sigma_x(i)$ for all $1\le i\le t-1$. 

Now, if we let $x^{-1}(i)=tl_i+\sigma_{x^{-1}}(i)$, then we obtain \begin{align*}
    i&=x(x^{-1}(i))=x(tl_i+\sigma_{x^{-1}}(i))\\
    &=l_ix(t)+x(\sigma_{x^{-1}}(i))\\
    &=l_ijt+tu_{\sigma_{x^{-1}}(i)}+\sigma_x(\sigma_{x^{-1}}(i))
\end{align*}

Now, note that $t$ divides $l_ijt+tu_{\sigma_{x^{-1}}(i)}$, however, $\sigma_x,\sigma_{x^{-1}}\in S_{t-1}$ and so they can only output values strictly smaller than $t$. This gives us two relations, first, $\sigma_x(\sigma_{x^{-1}}(i))=i$ for all $i$, and so $\sigma_x^{-1}=\sigma_{x^{-1}}$ which is very desirable, but also $l_ijt=-tu_{\sigma_x^{-1}(i)}$ for all $i$, which we can rewrite to obtain $$l_i= -j^{-1}u_{\sigma_x^{-1}(i)}\mod\frac{n}{t}\qquad\text{ for all }i=1,...,t-1.\eqno\qed$$

\begin{corollary}\label{cor:involution generated from smaller sigma}
{\it Let $x\in S_{n-1}$ have stabilizer $F_x=\langle a^t\rangle$. Then there exists integers $j,u_i$ for $1\le i\le t-1$ and remainder permutation $\sigma_x\in S_{t-1}$ as from \cref{lem:x generated from smaller sigma} such that \begin{align*}
    x(i)&=u_it+\sigma_x(i)\mod n\qquad \text{ for all }1\le i\le t-1\\
    x(qt)&=qjt\mod n\qquad\text{ for all } 1\le q\le \frac{n}{t}\\
    x(qt+i)&=qx(t)+x(i) \mod n
\end{align*} 

Then, if $x$ is an involution, $\sigma_x$ is as well and $u_i=-ju_{\sigma_x(i)}$ for all $1\le i\le t-1$. 

Conversely, every choice of $j$ coprime to $\frac{n}{t}$, involution $\sigma\in S_{t-1}$, and integers $u_i$ satisfying $u_i=-ju_{\sigma(i)}$ will generate an involution $x\in S_{n-1}$ which is stabilized by $a^t$.
}
\end{corollary}

\no{\it Proof.} \cref{lem:x generated from smaller sigma} gives that $\sigma_{x^{-1}}=\sigma_x^{-1}$ and so if $x^2=1$ then $\sigma_x^2=1$.

Note because $x^{-1}=x$, $j^{-1}=j$, and  $l_i=u_i$ for all $i$. Therefore, $u_i=-j^{-1}u_{\sigma_x^{-1}(i)}=-ju_{\sigma_x(i)}\mod\frac{n}{t}$ for all $i$. This says that for every fixed point of $\sigma_x$, we require that $u_i(j+1)= 0\mod\frac{n}{t}$. 

The converse follows by direct computation. \qed\\

Note that knowing $\sigma_x$ is an involution does {\it not} guarantee that $x$ is as well. For example, if $t=3$, then every $\sigma\in S_2$ has order $2$, but certainly $x$ need not be. 

It must be stressed here that the remainder permutation $\sigma_x$ plays a vital role in understanding (and thus counting) involutions $x$ with certain properties. 

While we do not explore it in this work, we believe that irreps in $J_n$ are in some way connected to and possible even extended from irreps in $J_t$ (for $t$ dividing $n$) and that remainder permutations are the key to understanding this connection.

\subsection{Counting Involutions with Certain Properties}\label{sec:certain properties}

Here we focus on counting formulas for involutions in $x\in S_{n-1}$ with stabilizer $F_x=\langle a^t\rangle$ and certain additional properties. Specifically, we will show how to count involutions with a given number of fixed points and involutions whose orbits contain a given number of involutions.

We will see that the ideas presented in this section and the next will be crucial to constructing a counting formula for the number of irreps of odd dimension with indicator $+1$. 

Although they may seem very different, the ideas used to construct these counting formulas build upon each other. That is, we must understand what involutions with stabilizer $F_x=\langle a^t\rangle$ in order to understand involutions with stabilizer $F_x=\langle a^t\rangle$ and $r$ fixed points, and we must understand involutions with stabilizer $F_x=\langle a^t\rangle$ and $r$ fixed points in order to understand involutions with stabilizer $F_x=\langle a^t\rangle$ and $r$ fixed points smaller than $t.$

\begin{center}

    \begin{mytable}[Big Table of Commonly Cited Values and Sets] \label{fig:table of sets and values}\end{mytable} 

 {\tabulinesep=1.2mm
\begin{tabu}{ c|c} 
 \hline
 Number & Set or Value \\
 \hline
 \hline
(1) & $E_{n/t}=\displaystyle\left\{1\le j\le \frac{n}{t}\,\big|\,j^2=1\mod\frac{n}{t}\right\}$ \\
 \hline
 (2) & $K_{j,n/t}=\displaystyle\left\{1\le u\le \frac{n}{t}\,\big|\,u(j+1)=0\mod\frac{n}{t}\right\}$ \\
 \hline
(3) & $\alpha_{j,n/t}=\displaystyle\text{gcd}\left(j+1,\frac{n}{t}\right)=|K_{j,n/t}|$ \\
 \hline
(4) & $P_{j,n/t}=\displaystyle\left\{q(j-1)\mod\frac{n}{t}\,\big|\,1\le q\le \frac{n}{t}\right\}$ \\
 \hline
 (5) & $P_{j,n/t}^{(c)}=\displaystyle\left\{1\le u\le \frac{n}{t}\,\big|\,u\not\in P_{j,n/t}\right\}\bigcap K_{j,n/t}$ \\
   \hline
  (6) & $\delta^{P^c}=\displaystyle \begin{cases} |P_{j,n/t}^{(c)}| & \text{ if } P_{j,n/t}^{(c)}\not=\varnothing\\ 1 & \text{ otherwise }\end{cases}$ \\
 \hline
 (7) & $\overline{K}_{j',j_{\sigma_x},n,t,s}=${\footnotesize$\displaystyle\left\{1\le u\le \frac{n}{t}\,\big|\,\left(u\frac{t}{s}+m_i\right)(1+j')= 0\mod\frac{n}{s}\text{ for some }m_i\in P_{j_{\sigma_x},t/s}\right\}$} \\
  \hline
 (8) & $\overline{K}_{j',j_{\sigma_x},n,t,s}^{(c)}=${\footnotesize $\displaystyle\left\{1\le u\le \frac{n}{t}\,|\,\left(u\frac{t}{s}+m_i\right)(1+j')= 0\mod\frac{n}{s}\text{ for some }m_i\in P_{j_{\sigma_x},t/s}^{(c)}\right\}$}\\
   \hline
 (9) & $\delta^{K^c}=\displaystyle\begin{cases} |\overline{K}_{j',j_{\sigma_x},n,t,s}^{(c)}| &\text{ if }\overline{K}_{j',j_{\sigma_x},n,t,s}^{(c)}\not=\varnothing\\ 1&\text{otherwise}\end{cases}$ \\
 \hline
  (10) & $\overline{E}_{j_{\sigma_x},n,t,s}=\displaystyle\left\{j_{\sigma_x}+\overline{m}\frac{t}{s}\,\big|\,1\le \overline{m}\le \frac{n}{t}\right\}\bigcap E_{n/s}$ \\
 \hline
 (11) & $\beta_{j,n/t}=\displaystyle\text{gcd}\left(j-1,\frac{n}{t}\right)$ \\
 \hline
 (12) &  $\delta_{j,n/t,r}=\displaystyle\begin{cases} 1 & \text{ if }\beta_{j,n/t}|r\text{ and }r\le t\beta_{j,n/t}\\ 0&\text{ otherwise}\end{cases}$ \\
 \hline
 (13) & $m_{r,j,n/t}=\displaystyle\frac{r}{\beta_{j,n/t}}$\\
 \hline
 (14) & $K_{j,n/2}'=\left\{u\,|(u+1)(j+1)=0\mod\frac{n}{2}\,\right\}$\\
 \hline
 (15) & $\overline{\delta}_{u,j,n/2}=\begin{cases} 1 & \text{ if }j\not=2u+1\mod\frac{n}{2}\\ 0 & \text{ otherwise}\end{cases}$\\
 \hline
 (16) & $\delta_{i,u,n/2}^0=\begin{cases} 1 & \text{ if }-i(u+1)=0\mod\frac{n}{2}\\ 0 & \text{ otherwise}\end{cases}$\\
 \hline
 (17) & $\delta_{i,u,n/2}^{\not=0}=\begin{cases} 1 & \text{ if }-i(u+1)\not=0\mod\frac{n}{2}\\ 0 & \text{ otherwise}\end{cases}$\\
\end{tabu}}
\end{center}
\noindent{\small Note: A table listing the majority of sets and values which are used heavily in the statements in this section.}

\begin{proposition}\label{prop:number of ord 2 elements for given stab} 
{\it Let $T_{n/t}$ be the set of involutions $x\in S_{n-1}$ such that $F_x=\langle a^t\rangle$. Let $E_{n/t}$ be as from (1) and $\alpha_{j,n/t}$ be as from (3) in \cref{fig:table of sets and values} (\nameref{fig:table of sets and values}). Then \begin{align}
    |T_{n/1}|&=|E_{n/1}|\label{eq:T1}\\
    |T_{n/t}|&=\sum_{j\in E_{n/t}}\sum_{l=0}^{\left\lfloor\frac{t-1}{2}\right\rfloor}(\alpha_{j,n/t})^{t-1-2l}\left(\frac{n}{t}\right)^l\frac{(t-1)!}{(t-1-2l)!2^ll!}-\sum_{\substack{s|t \\ s\not=t}}|T_{n/s}|\label{eq:Tt}\\
    |T_{n/n}|&=\sum_{l=0}^{\frac{n-1}{2}}\frac{(n-1)!}{(n-1-2l)!2^ll!}-\sum_{\substack{s|n \\ s\not=n}}|T_{n/s}|\label{eq:Tn}
\end{align} }
\end{proposition}

\no{\it Proof.} Let $x\in S_{n-1}$ be an involution with stabilizer $F_x=\langle a^t\rangle$ and let $$E_{n/t}:=\left\{1\le j\le n\,|\,j^2= 1\mod\frac{n}{t}\right\}.$$ From \cref{cor:involution generated from smaller sigma}, let $\sigma_x\in S_{t-1}$, write $x(i)=tu_i+\sigma_x(i)$ for all $1\le i\le t-1$ and $x(t)=jt$ for some $j\in E_{n/t}$ which determines $x$ completely. 

Therefore, the number of $x$ satisfying $x^2=1$ can be counted by fixing a value $j\in E_{n/t}$, then for each involution $\sigma\in S_{t-1}$, we count the number of fixed points of $\sigma$ which requires counting the number of transpositions present in the disjoint cycle decomposition of $\sigma.$

Note, there are $\frac{(t-1)!}{(t-1-2l)!2^ll!}$ permutations $\sigma\in S_{t-1}$ with $\sigma^2=1$ and $\sigma$ being a product of $l$ disjoint cycles. This is because we choose any two integers between $1$ and $t-1$ to put in a transposition (so $\binom{t-1}{2}=\frac{(t-1)(t-2)}{2}$, then we choose any two of the remaining $t-3$ integers to pair up in a second transposition so $\frac{(t-3)(t-4)}{2}$), and we repeat this $l$ times which gives $$\frac{(t-1)(t-2)(t-3)\cdots(t-1-2l)}{2^ll!}$$ where we must divide by $l!$ because the order in which we obtain each transposition does not matter, that is, $(1\pp 2)(3\pp 4)=(3\pp 4)(1\pp 2)$ and there are $l!$ ways to re-order the product of transpositions.

Now, if $x$ has remainder permutation $\sigma$, then from \cref{lem:x generated from smaller sigma}, we know that $u_i= -ju_{\sigma(i)}\mod\frac{n}{t}$ and so each fixed point of $\sigma$ (of which there are $t-1-2l$) forces $u_i(j+1)= 0\mod\frac{n}{t}$ which contributes gcd$(j+1,\frac{n}{t})$ possible choices for $u_i$ where $\sigma(i)=i$ (see last paragraph of the proof of \cref{lem:x generated from smaller sigma}). Furthermore, each transposition $(i\pp i')$ in $\sigma$ which pairs $u_i$ with $u_{i'}$ and so contributes $\frac{n}{t}$ choices for $u_i$ (which then determines $u_{i'}$). 

Let $\alpha_{j,n/t}:=\text{gcd}\left(j+1,\frac{n}{t}\right)$. Therefore, for a fixed $j$, and fixed $\sigma$ comprised of $l$ disjoint transpositions, there are $(\alpha_{j,n/t})^{t-1-2l}\left(\frac{n}{t}\right)^l$ choices for the $u_i$ and so if we multiply this by the number of such $\sigma$, which is $\frac{(t-1)!}{(t-1-2l)!2^ll!}$ and then sum over all $j\in E_{n/t}t$, and subtract those we have already counted, we obtain our result

$$|T_{n/t}|=\underbrace{\sum_{j\in E_{n/t}}}_{\text{choices for }x(t)}\underbrace{\sum_{l=0}^{\left\lfloor\frac{t-1}{2}\right\rfloor}}_{\substack{\#\text{ of fixed} \\ \text{points of }\sigma_x}}\underbrace{(\alpha_{j,n/t})^{t-1-2l}}_{\substack{\text{choices for }x(i) \\ \text{ when }\sigma_x(i)=i}}\underbrace{\left(\frac{n}{t}\right)^l}_{\substack{\text{choices for }x(i) \\ \text{ when }\sigma_x(i)\not=i}}\underbrace{\frac{(t-1)!}{(t-1-2l)!2^ll!}}_{\text{ choice for }\sigma_x}-\underbrace{\sum_{\substack{s|t \\ s\not=t}}|T_{n/s}|}_{\text{remove over count}}$$

In the case where $t=1$, then $l=0$ and all powers are $0$ so we are simply summing $1$ for every $j\in E_{n/1}$.

Of course, if $t=n$, all sets trivialize and so we simply count all involutions in $S_{n-1}$ and then remove those we have already counted.
\qed\\

Again, as with \cref{prop:number of x with Fx=<a^t>}, explicitly writing out the set $T_{n/t}$ is done recursively by using the formula given.

Note that we can now verify the statements made at the end of the discussion under \cref{thm:indicator reduced} that if $n=p^r$ for $p$ and odd prime and $r\in\mathbb{N}$, then there exists only one nontrivial involution $x\in S_{n-1}$ with $F_x=F$. Since $E_{p^r/1}=\{1,-1\}$, we have only two permutations in $$T_{p^r/1}=\{(1),(1\pp p^r-1)(2\pp p^r-2)\cdots(\frac{p^r-1}{2}\pp\frac{p^r+1}{2})\}.$$

\begin{lemma}\label{lem:fixed points of x relating to q and sigma}
{\it Let $x\in S_{n-1}$ be an involution with $F_x=\langle a^t\rangle$. Let $j\in E_{n/t}$ as from \cref{fig:table of sets and values} (\nameref{fig:table of sets and values}) and write $x(t)=jt$ and for all $1\le i<t$, and $x(i)=u_it+\sigma_x(i)$ for some involution $\sigma_x\in S_{t-1}$ and constants $1\le u_i\le \frac{n}{t}$. Let $$\beta_{j,n/t}=\text{gcd}\left(j-1,\frac{n}{t}\right)=\text{gcd}\left(1-j,\frac{n}{t}\right).$$ Then, $x(qt+i)=qt+i$ for $\beta_{j,n/t}$ different values $q$ if and only if $\sigma_x(i)=i$ and $u_i= q(1-j)\mod\frac{n}{t}$ for some $q$.}
\end{lemma}

\no{\it Proof.} 
\begin{enumerate}
    \item[$\implies$] Let $x(qt+i)=qt+i$ be a fixed point of $x$. Then this says that $$x(qt+i)=qx(t)+x(i)=qjt+u_it+\sigma_x(i)=qt+i.$$ Because $\sigma_x(i)<t$ and $i<t$, this forces $\sigma_x(i)=i$. Thus, subtracting we obtain that $$u_it=qt-qjt\implies u_i= q(1-j)\mod\frac{n}{t}.$$ 
\item[$\impliedby$] Let $\sigma_x(i)=i$ and $u_i= q(1-j)\mod\frac{n}{t}$ for some $q$. Now, since there exists a $q$ for which the equation $u_i= q(1-j)\mod\frac{n}{t}$ has a solution, there exists a total of $\beta_{j,n/t}$ solutions and so there are $\beta_{j,n/t}$ possible values $q$ for which $$x(qt+i)=qjt+u_it+\sigma_x(i)=qjt+q(1-j)t+i=qt+i.\eqno\qed$$ 
\end{enumerate}

\begin{proposition}\label{prop:number of ord 2 elements for given stab with certain fix} 
{\it Let $R_{n/t,r}$ be the set of involutions $x\in S_{n-1}$ such that $F_x=\langle a^t\rangle$ and $x$ has exactly $r$ fixed points (including $x(n)=n$ as one of the fixed points). Let $E_{n/t}$, $\beta_{j,n/t}$, $m_{r,j,n/t}$, $\delta_{j,n/t,r}$, $P_{j,n/t}$, $\delta^{P^c}$ as from \cref{fig:table of sets and values} (\nameref{fig:table of sets and values}). Let $\delta_r$ be the Kronecker-Delta step function. Then \begin{align}
  |R_{n/1,r}|&=\sum_{j\in E_{n/1}}\delta_r(\beta_{j,n/1})\label{eq:R1}\\
  |R_{n/t,r}|&=\sum_{j\in E_{n/t}}\sum_{l=0}^{\lfloor\frac{t-m_{r,j,n/t}}{2}\rfloor}\delta_{j,n/t,r}\delta_{j,n/t,r}^l\binom{t-1-2l}{m_{r,j,n/t}-1}|P_{j,n/t}|^{m_{r,j,n/t}-1}(\delta^{P^c})^{t-2l-m_{r,j,n/t}}\nonumber\\
  &\qquad\qquad\qquad\qquad\qquad\qquad\qquad\qquad\left(\frac{n}{t}\right)^l\frac{(t-1)!}{(t-1-2l)!2^ll!}-\sum_{\substack{s|t\\s\not=t}}|R_{n/s,r}|\label{eq:Rt}\\
  |R_{n/n,r}|&=\frac{(n-1)!}{(r-1)!2^{\frac{n-r}{2}}\left(\frac{n-r}{2}\right)!}-\sum_{\substack{s|n\\s\not=n}}|R_{n/s,r}|\label{eq:Rn}\\
  |R_{n/t,r}|&=0\text{ if }n-r\text{ is not even}.\nonumber
\end{align} }
\end{proposition}

\no{\it Proof.} Let $x\in S_{n-1}$ be an involution with stabilizer $F_x=\langle a^t\rangle$. For some $\sigma_x\in S_{t-1}$, write $x(i)=tu_i+\sigma_x(i)$ for all $1\le i\le t-1$ and $x(t)=jt$ for some $j\in E_{n/t}$ which determines $x$ completely as from \cref{cor:involution generated from smaller sigma}.

Then, assume $x$ has $r$ fixed points (including $x(n)=n$ as a fixed point).

Since $x(qt)=jqt=qt$ if and only if $q(j-1)=0\mod\frac{n}{t}$, we immediately get $\beta_{j,n/t}:=$gcd$(j-1,\frac{n}{t})$ fixed points for $x$ from the given choice of $j.$

Additionally, from \cref{lem:fixed points of x relating to q and sigma}, we know that for a fixed $1\le i<t$, $x(qt+i)=qt+i$ for $\beta_{j,n/t}$ different values of $q$ if and only if $\sigma_x(i)=i$ and $u_i=q(1-j)\mod\frac{n}{t}$.

Thus, the total number of fixed points of $x$ is a multiple of $\beta_{j,n/t}$ and so $\beta_{j,n/t}$ must divide $r$. To ensure the existence of an $x$ with the given choice of $j$ and $r$, the value $\delta_{j,n/t,r}$ must be included in our sum where $$\delta_{j,n/t,r}:=\begin{cases} 1 & \text{ if }\beta_{j,n/t}|r\text{ and }r\le t\beta_{j,n/t}\\ 0 &\text{ otherwise }\end{cases}.$$

Furthermore, letting $m_{r,j,n/t}:=\frac{r}{\beta_{j,n/t}}$, the ratio minus one, $m_{r,j,n/t}-1$ tells us how many $u_i$ for $1\le i<t$ must be of the form $q(1-j)$ for some $q$. Let $$P_{j,n/t}:=\left\{q(1-j)\mod\frac{n}{t}\,|\,1\le q\le\frac{n}{t}\right\}.$$

Then $m_{r,j,n/t}-1$ of the $u_i$ come from $P_{j,n/t}$, and $t-1-(m_{r,j,n/t}-1)=t-m_{r,j,n/t}$ of the remaining $u_i$ come from its compliment $$P_{j,n/t}^c:=\left\{u\,|\,u\not=q(1-j)\mod\frac{n}{t}\text{ for all }q\right\}.$$ However, recall that if $\sigma_x(i)=i$, then $u_i(j+1)= 0\mod\frac{n}{t}$ is required by \cref{cor:involution generated from smaller sigma}. Of course, if $u_i\in P_{j,n/t}$, then $u_i=q(j-1)$ for some $q$ and so certainly $u_i(j+1)=q(j-1)(j+1)=q(j^2-1)=0$ since $j^2=1$. However, to guarantee this for $u_i\in P_{j,n/t}^c$, we must actually require that $u_i\in P_{j,n/t}^{(c)}:=P_{j,n/t}^c\cap K_{j,n/t}$ where $$K_{j,n/t}:=\left\{1\le u\le \frac{n}{t}\,|\,u(j+1)= 0\mod\frac{n}{t}\right\}.$$ 

In summary: 
\begin{myenum}
    \item If $\sigma_x(i)=i$, then either:
    \begin{myenum}
        \item $u_i\in P_{j,n/t}$ and so contributes $\beta_{j,n/t}$ fixed points to $x$
        \item or $u_i\in P_{j,n/t}^{(c)}$ and so contributes $0$ fixed points to $x$
    \end{myenum}
    \item If $\sigma_x(i)\not=i$ then $x(qt+i)\not=qt+i$ for any $q$ and so we can allow $u_i$ to take any value between $1$ and $\frac{n}{t}$ and this will not add any fixed points to $x$.
\end{myenum}

Let $t-1-2l$ be the number of fixed points of $\sigma_x$ for some $l$. Then, since $m_{r,j,n/t}-1$ of the $u_i$ come from $P_{j,n/t}$, $\sigma_x$ must have at least $m_{r,j,n/t}-1$ fixed points which means $t-1-2l\ge m_{r,j,n/t}-1$ and so $t-m_{r,j,n/t}\ge 2l$ so finally $l\le \frac{t-m_{r,j,n/t}}{2}$.

Additionally, for the remaining $t-1-2l-(m_{r,j,n/t}-1)=t-2l-m_{r,j,n/t}$ fixed points of $\sigma_x$, $u_i$ must be taken from $P_{j,n/t}^{(c)}$, however, there is an edge case we must address here.

Assume $P_{j,n/t}^{(c)}=\varnothing$. Then either $P_{j,n/t}^c=\varnothing$ (which is possible) or $K_{j,n/t}\subset P_{j,n/t}$. In the first case, we get $P_{j,n/t}=\{1,2,...,n/t\}$ and since $P_{j,n/t}\subset K_{j,n/t}$, we get $K_{j,n/t}=P_{j,n/t}$. In the second case, $K_{j,n/t}\subset P_{j,n/t}$ and so again $K_{j,n/t}=P_{j,n/t}$. In either case, if $i$ is a fixed point of $\sigma_x$, and so $\sigma_x(i)=i$, then because $u_i\in K_{j,n/t}=P_{j,n/t}$ by \cref{cor:involution generated from smaller sigma}, $x$ has $\beta_{j,n/t}$ fixed points of the form $qt+i$ for some $q$.  

Namely, if $P_{j,n/t}^{(c)}=\varnothing$, every fixed point of $\sigma_x$ contributes exactly $\beta_{j,n/t}$ fixed points to $x$ and so the number of fixed points of $\sigma_x$ is equal to $m_{j,n/t}-1$. That is $t-1-2l=m_{j,n/t}-1$ and so $t-m_{j,n/t}=2l$. 

All this is to say that if $P_{j,n/t}^{(c)}=\varnothing$, then we must have $t-m_{j,n/t}=2l$ so we introduce $$\delta_{j,n/t,r}^l=\begin{cases} 0&\text{ if }|P_{j,n/t}^{(c)}|=\varnothing\text{ and }2l\not=t-m_{j,n/t}\\1&\text{otherwise}\end{cases}$$

However, this condition is messy to write down as we do not want to be stuck with a $0^0$ indeterminate in our sum so to avoid this, the set $P_{j,n/t}^{(c)}$ is never used. Instead, we include the messy conditional $$\delta^{P^c}:=\begin{cases} |P_{j,n/t}^{(c)}|&\text{ if }P_{j,n/t}^{(c)}\not=\varnothing\\ 1 &\text{ otherwise}\end{cases}.$$

Now, we are at last ready to fully describe all involutions with stabilizer $F_x=\langle a^t\rangle$ with and $r$ fixed points.

In the base case, if $t=1$, then $x(i)=ij=i$ if and only if $i(j-1)=0\mod n$. So $x$ has $r$ fixed points if and only if the number of solutions to this equation is $r$. Since $i=n$ is always a solutions, the number of solutions to this equation is $\beta_{j,n/1}=$gcd$(j-1,n)$. Let $\delta_r$ be the step function so $$\delta_r(u)=\begin{cases} 1 & \text{ if }r=u\\ 0 & \text{ otherwise}\end{cases}
$$

Then $$|R_{n/1,r}|=\sum_{j\in E_{n/1}}\delta_r(\beta_{j,n/1}).$$ 

Finally,
$$|R_{n/t,r}|=\underbrace{\sum_{j\in E_{n/t}}}_{\text{choices for }x(t)}\underbrace{\sum_{l=0}^{\lfloor\frac{t-m_{r,j,n/t}}{2}\rfloor}}_{\substack{\#\text{ of fixed}\\ \text{points of }\sigma_x\\}}\underbrace{\delta_{j,n/t,r}}_{x\text{ exists}}\underbrace{\delta_{j,n/t,r}^l}_{x\text{ matches }\sigma_x}\underbrace{\binom{t-1-2l}{m_{r,j,n/t}-1}}_{\substack{\text{ choose }u_i\\ \text{ in }P_{j,n/t}}}\,\,\,\underbrace{|P_{j,n/t}|^{m_{r,j,n/t}-1}}_{\substack{\sigma_x(i)=i\text{ and }\\ u_i=q(1-j)}}\underbrace{(\delta^{P^c})^{t-2l-m_{r,j,n/t}}}_{\substack{\sigma_x(i)=i\text{ and }\\ u_i\not=q(1-j)}}$$ $$\qquad\qquad\qquad\qquad\qquad\qquad\qquad\underbrace{\left(\frac{n}{t}\right)^l}_{\sigma_x(i)\not=i}\underbrace{\frac{(t-1)!}{(t-1-2l)!2^ll!}}_{\text{ choices for }\sigma_x} -\underbrace{\sum_{\substack{s|t\\s\not=t}}R_{n/s,r}}_{\text{ remove over count}}$$

Note that this agrees with our base case since if $t=1$ we know that $1\ge m_{r,j,n/t}$ and that $m_{r,j,n/t}$ is an integer so $m_{r,j,n/t}=1$. Thus, necessarily $r=\beta_{j,n/t}$ (or of course no such $x$ exists) and so $l=0$. So the sum becomes $$\sum_{j\in E_{n/1}}\delta_{j,n/t,r}=\sum_{j\in E_{n/1}}\delta_r(\beta_{j,n/t})$$ as desired. 

Furthermore, If $t=n$, all sets trivialize and so we simply count all involutions with $r$ fixed points in $S_{n-1}$ and then remove those we have already counted.

Note also that if $n-r$ is not even, then $x$ cannot be an involution in $S_{n-1}$ with $r$ fixed points, and so we say that $|R_{n/t,r}|=0$ if $n-r$ is not even.
\qed\\

\begin{lemma}\label{lem:P and K relationship} 
{\it Either $P_{j,n/t}=K_{j,n/t}$ and $\alpha_{j,n/t}\beta_{j,n/t}=\frac{n}{t}$ or $|P_{j,n/t}|=\frac{1}{2}|K_{j,n/t}|=|P_{j,n/t}^{(c)}|$ and $\alpha_{j,n/t}\beta_{j,n/t}=2\frac{n}{t}$}
\end{lemma}

\no{\it Proof.} Recall that $\beta_{j,n/t}:=\text{gcd}\left(1-j,\frac{n}{t}\right)$ and $P_{j,n/t}:=\{q(1-j)\mod\frac{n}{t}\}$. We claim $|P_{j,n/t}|=\frac{n}{t\beta_{j,n/t}}$. This is because $q(1-j)=q'(1-j)$ if and only if $(q-q')(1-j)=0\mod\frac{n}{t}$ if and only if $q=q'\mod\frac{n}{t\beta_{j,n/t}}$ which proves the claim.

Recall now that $P_{j,n/t}\subset K_{J,n/t}\subset\{1,2,...,\frac{n}{t}\}$ is always true and so the results follow by checking the cases where $P_{j,n/t}^{(c)}=\varnothing$ and $P_{j,n/t}^{(c)}\not=\varnothing$ separately and using basic properties of the gcd.   \qed\\

\begin{remark}\label{remark:equal}
By definition $\sum_{r=1}^n|R_{n/t,r}|=|T_{n/t}|$ must be true, however, using some careful computation it can be shown that both equations from \cref{prop:number of ord 2 elements for given stab} and \cref{prop:number of ord 2 elements for given stab with certain fix} satisfy this relation.
\end{remark}

\subsection{Prerequisites for Counting Indicators}\label{sec:prereq}

Now, lets say we want to count the number of $x\in S_{n-1}$ which have indicators of a certain value. Note that if $x^2=1$, then $x=x^{-1}\in\mathcal{O}_x$ and so by \cref{prop:observations}, all indicators are nonnegative. Now, what determines whether the indicator of $x$ is zero or not? Well, $\nu(\hat{V})=1$ if $\zeta_n^{i\frac{x(t)+t}{t}}=1$ and $0$ otherwise.

Therefore, knowing the value of $x(t)$ tells us exactly which (and so how many) $i$ exist for which $\nu(\hat{V})\not=0$. Now, assume that $t$ is odd. Then $\mathcal{O}_x$ is closed under inversion if and only if it contains an element $y$ of order $2$. Since $\mathcal{O}_y=\mathcal{O}_x$, we may relabel so as to only consider orbits $\mathcal{O}_x$ where $x$ is an involution.

And so, in the case where $t$ is odd, the strategy is this: 

\begin{myenum}
    \item $x^{-1}\in\mathcal{O}_x$ if and only if $\mathcal{O}_x$ contains an element of order $2$.
    \item The number of involutions (excluding $x$ itself) in $\mathcal{O}_x$ is given by the number of fixed points of $\sigma_x\in S_{t-1}$ which generates $x$ (by \cpref{lem:y in orbit x}{lem:y in orbit x (2)}). 
    \item The number of representations of $J_n$ induced form $x$ which have positive indicator is given by the number of values $i$ for which $\zeta_n^{i\frac{x(t)+t}{t}}=1$ (by \cref{thm:indicator reduced}).
\end{myenum} 

Namely, if we want to count the nonzero indictors of irreps of $J_n$ of a fixed odd dimension $t$, we must be able to count the number of orbits which contain exactly $r$ involutions.

\begin{proposition}\label{prop:number of ord 2 elements for given stab with certain fixed mod t} 
{\it Let $X_{n/t,r}$ be the set of involutions $x\in S_{n-1}$ with $F_x=\langle a^t\rangle$, and such that $\mathcal{O}_x$ has $r\not=0$ involutions. Let $T_{n/1}$ be as from \cref{prop:number of ord 2 elements for given stab}, $R_{n/t,r}$ as from \cref{prop:number of ord 2 elements for given stab with certain fix}, and we pull the entries (1), (2), (4), (6), (7), (9), (10), (11), (12), (13) from \cref{fig:table of sets and values} (\nameref{fig:table of sets and values}).
Then \begin{align}
    |X_{n/1,r}|&=|T_{n/1}|=|E_{n/1}|\label{eq:X1}\\
    |X_{n/t,r}|&=\sum_{j\in E_{n/t}}(\alpha_{j,n/t})^{r-1}\left(\frac{n}{t}\right)^\frac{t-r}{2}\frac{(t-1)!}{(r-1)!2^\frac{t-r}{2}\left(\frac{t-r}{2}\right)!}-\sum_{\substack{s|t\\s\not=t}}|C_{n,t,s,r}|\label{eq:Xt}\\
    |X_{n/n,r}|&=\frac{(n-1)!}{(r-1)!2^{\frac{n-r}{2}}\left(\frac{n-r}{2}\right)!}-\sum_{\substack{s|n\\s\not=n}}|C_{n,n,s,r}|\label{eq:Xn}\\
    |X_{n/t,r}|&=0\text{ if }t-r\text{ is negative or odd}\nonumber
\end{align}
    where \begin{align}
    |C_{n,t,1,r}|&=\sum_{j_{\sigma_x}\in E_{t/1}}\delta_r(\beta_{j_{\sigma_x},t/1})|\overline{E}_{j_{\sigma_x},n,t,1}|\label{eq:C1}\\
    |C_{n,t,s,r}|&=\sum_{j_{\sigma_x}\in E_{t/s}}\sum_{j'\in \overline{E}_{j_{\sigma_x},n,t,s}}\sum_{l=0}^{\lfloor\frac{s-m_{r,j_{\sigma_x},t/s}}{2}\rfloor}\delta_{j_{\sigma_x},t/s,r}\delta_{j_{\sigma_x},t/s,r}^l\binom{s-1-2l}{m_{r,j_{\sigma_x},t/s}-1}|P_{j_{\sigma_x},t/s}|^{m_{r,j_{\sigma_x},t/s}-1}\nonumber\\
    &\qquad(\delta^{P^c})^{s-2l-m_{r,j_{\sigma_x},t/s}}\frac{(s-1)!}{(s-1-2l)!2^ll!}|\overline{K}_{j',j_{\sigma_x},n,t,s}|^{m_{r,j_{\sigma_x},t/s}-1}(\delta^{K^c})^{s-2l-m_{r,j_{\sigma_x},t/s}}\left(\frac{n}{s}\right)^l\nonumber\\
    &\qquad\qquad\qquad\qquad\qquad\qquad\qquad\qquad\qquad\qquad-\sum_{\substack{p|s\\p\not=d}}|C_{n,t,p,r}|\label{eq:Ct}\\
    |C_{n,n,s,r}|&=|R_{n/s,r}|\label{eq:Cn}
\end{align}}
\end{proposition}

\no{\it Proof.} Once again, we count. Assume $x\in S_{n-1}$ is an involution with $F_x=\langle a^t\rangle$ such that $\mathcal{O}_x$ contains exactly $r$ involutions. Our goal is to fully describe the conditions necessary to construct $x$ and then count the possible ways to achieve those conditions.

We know that there exists integers $u_i$ and permutation $\sigma_x\in S_{t-1}$ such that for all $1\le i\le t-1$, $x(i)=tu_i+\sigma_x(i)$. Furthermore, we know that $\sigma_x$ has exactly $r$ fixed points (including $\sigma_x(t)=t$). This is because \cpref{lem:y in orbit x}{lem:y in orbit x (2)} tells us that if we write $\mathcal{O}_x=\{x=y_t,y_1,y_2,...,y_{t-1}\}$, then $y_i\in\mathcal{O}_x$ has order $2$ if and only if $i=x(i)\mod t=\sigma_x(i)$. Since $\mathcal{O}_x$ is assume to have $r$ involutions (including $x$), $\sigma_x$ must then have $r$ fixed points. 

Now, we can adapt the formulas from \cref{prop:number of ord 2 elements for given stab} here, however, there will be a few key differences. Let $j\in E_{n/t}$ be fixed. Then recall that the number of involutions $\sigma\in S_{t-1}$ which are a product of $l$ disjoint cycles is given by $\frac{(t-1)!}{(t-1-2l)!2^ll!}$. In this case, $\sigma$ has $t-1-2l$ fixed points. In our case we want $t-1-2l=r-1$ (we now exclude $\sigma_x(t)$ as a fixed point since we are viewing $\sigma_x$ as strictly a permutation in $S_{t-1}$) and so we obtain that $l=\frac{t-r}{2}$. Note that $t-r$ must be even (since this is the number of permutations $y\in\mathcal{O}_x$ which are not involutions) this value is an integer. Finally, we obtain that there are $\frac{(t-1)!}{(r-1)!2^\frac{t-r}{2}\left(\frac{t-r}{2}\right)!}$ involutions $\sigma\in S_{t-1}$ with $r-1$ fixed points. 

Thus, there should be $$\sum_{j\in E_{n/t}}(\alpha_{j,n/t})^{r-1}\left(\frac{n}{t}\right)^\frac{t-r}{2}\frac{(t-1)!}{(r-1)!2^\frac{t-r}{2}\left(\frac{t-r}{2}\right)!}$$ involutions $x\in S_{n-1}$ with stabilizer $\langle a^t\rangle$ which have $r$ fixed points mod $t$. This should seem familiar as it is the same formula as from \cref{prop:number of ord 2 elements for given stab} with $l$ fixed. 

However, we have--unfortunately--over counted. What this value really gives us, is the number of involutions $x\in S_{n-1}$ which have $r$ fixed points mod $t$ and which are \textit{stabilized} by $a^t$. However, their actual stabilizer set could be larger than $\langle a^t\rangle.$

To compensate for this, we must remove those permutations which have smaller stabilizer. However, we cannot just do this recursively. Note that we are counting the number of fixed points of $x$ modulo $t$, which we will not do if we simply replace every $t$ with an $s$. Rather, we must work harder here. 
Certainly, if $x$ is actually stabilized by $a^s$ for some $s|t$, then there exists a $\tau_x\in S_{s-1}$ such that $x(i)=l_is+\tau_x(i)$ for all $i<s$. This is to say that we have two equivalent ways of writing $x(i)$. Namely, for all $i<s$ we have \begin{align*}
    x(i)&=tu_i+\sigma_x(i)\\
    x(i)&=sl_i+\tau_x(i).
\end{align*}

Namely, $\sigma_x(i)=sl_i-tu_i+\tau_x(i)\mod n$ and so $\sigma_x(i)=\tau_x(i)\mod s$ for all $i<s$ and $\sigma_x(s)=s\mod s$. Note that this actually tells us that $\sigma_x$ is generated from $\tau_x$ in the same way that $x$ is generated from $\sigma_x$. That is, $\sigma_x$ is stabilized by $a_t^s$ where $a_t=(1\pp2\pp\cdots\pp t)$ is the standard $t$-cycle and so there exists integers $1\le m_i\le \frac{t}{s}$ such that $\sigma_x(i)=m_is+\tau_x(i)$. We can now write (for all $i<s)$ \begin{align*}
    x(i)&=tu_i+\sigma_x(i)\qquad 1\le u_i\le \frac{n}{t},\text{ and }\sigma_x\in S_{t-1}\\
    x(i)&=sl_i+\tau_x(i)\qquad 1\le l_i\le \frac{n}{s},\text{ and }\tau_x\in S_{s-1}\\
    \sigma_x(i)&=sm_i+\tau_x(i)\qquad 1\le m_i\le \frac{t}{s}.\\
\end{align*} Note, our goal is to write all possible $l_i$. Furthermore, there exists $j_{\sigma_x}\in E_{t/s}$ such that $\sigma_x(s)=j_{\sigma_x}s$. Now, using again that $\sigma_x(i)=l_is-u_it+\tau_x(i)$, we have that $$l_is=m_is+u_it\implies l_i= u_i\frac{t}{s}+m_i\mod\frac{n}{s}.$$

Note that $x(s)=tu_s+\sigma_x(s)=tu_s+j_{\sigma_x}s=s(\frac{t}{s}u_s+j_{\sigma_x})$ and so if we hope to have $x$ be an involution stabilized by $a^s$ we must have $\frac{t}{s}u_i+j_{\sigma_x}\mod\frac{n}{s}\in E_{n/s}$. Let $$\overline{E}_{j_{\sigma_x},n,t,s}:=\left\{j_{\sigma_x}+\overline{m}\frac{t}{s}\,|\,1\le \overline{m}\le \frac{n}{t}\right\}\bigcap E_{n/s}.$$ Note here that if $x(s)=j's$ then $x(t)=\frac{t}{s}x(s)=\frac{t}{s}j's=tj'=tj\mod n$.

Now, after choosing a $j'\in\overline{E}_{j_{\sigma_x},n,t,s}$ such that $x(s)=j's$, our goal is to construct all such $x$ are stabilzied by $a^s$. This requires determining $\tau_x$, $m_i$ and $u_i$ such that $\sigma_x(i)=sm_i+\tau_x(i)$ for all $i<s$ and $x(i)=tu_i+\sigma_x(i)$. If we determine all such values, then we can let $l_i=u_i\frac{t}{s}+m_i\mod\frac{n}{s}$ and this will give us exactly a permutation $x$ stabilized by both $a^t$ and $a^s$ with remainder permutations $\sigma_x$ and $\tau_x$. Then, for all $i<s$ we have $$x(i)=tu_i+\sigma_x(i)=tu_i+sm_i+\tau_x(i)=s(\frac{t}{s}u_i+m_i)+\tau_x(i)=sl_i+\tau_x(i).$$

Note now that $x(\sigma_x(i))=u_{\sigma_x(i)}t+i$ and $x(\sigma_x(i))=x(m_is+\tau_x(i))=m_ij's+l_{\tau_x(i)}s+i$ where we are assuming $x$ is stabilized by $a^s$ and that $x(s)=j's$. Therefore, we also have that $$l_{\tau_x(i)}s=u_{\sigma_x(i)}t-m_ij's$$

Finally, recall from \cref{cor:involution generated from smaller sigma} that $u_{\sigma_x(i)}t=-ju_it\mod n$ for all $1\le i\le t-1$ and that $jt=x(t)=\frac{t}{s}x(s)=\frac{t}{s}j's=tj'$ so substituting we get $$l_{\tau_x(i)}s=-ju_it-m_ij's\implies l_{\tau_x(i)}=-j'u_i\frac{t}{s}-m_ij's\mod\frac{n}{s}.$$ If $\tau_x(i)=i$, then these two must be equal so $m_is+u_it=-ju_it-m_ij's=-j'tu_i-m_ij's$ which gives that \begin{align*}
    u_i\frac{t}{s}+m_i&=-j'u_i\frac{t}{s}-m_ij's\mod\frac{n}{s}\\
    \implies u_i(1+j')\frac{t}{s}+m_i(1+j')&=0\mod\frac{n}{s}\\
    \implies\left(\frac{t}{s}u_i+m_i\right)(1+j')&= 0\mod\frac{n}{s}.
\end{align*}

Note that if $\sigma_x(i)=i$ for $i<s$, then $u_i(j+1)=0\mod\frac{n}{t}$. However, this will always be the case since if $\sigma_x(i)=i$, then necessarily $\tau_x(i)=i$ and there is no $m_i$ so $m_i=0$. Namely, \begin{align*}
m_is+u_it&=-j'tu_i-m_ij's\\
\implies u_it&=-jtu_i\qquad j't=jt\mod n\\
\implies u_i(1+j)&=0\mod\frac{n}{t}
\end{align*} as desired. 

Namely, if we let $$\overline{K}_{j',j_{\sigma_x},n,t,s}:=\left\{u\,|\,(u\frac{t}{s}+m_i)(1+j')= 0\mod\frac{n}{s}\text{ for any }m_i\in P_{j_{\sigma_x},t/s}\right\}$$ then $\{0\}\subset\overline{K}_{j',j_{\sigma_x},n,t,s}\subset K_{j,n/t}$ where $j=j'\mod \frac{n}{t}$. Let $$\overline{K}_{j',j_{\sigma_x},n,t,s}^{(c)}:=\left\{u\,|\,(u\frac{t}{s}+m_i)(1+j')= 0\mod\frac{n}{s},m_i\in P_{j_{\sigma_x},t/s}^{(c)}\right\}.$$ Note that it is possible for $\overline{K}_{j',j_{\sigma_x},n,t,s}^{(c)}$ to be empty which is why must introduce $$\delta^{K^c}:=\begin{cases} |\overline{K}_{j',j_{\sigma_x},n,t,s}^{(c)}| &\text{ if }\overline{K}_{j',j_{\sigma_x},n,t,s}^{(c)}\not=\varnothing\\ 1&\text{otherwise}\end{cases}$$

This gives the necessary requirements for $x.$ Namely, we must have that if $\tau_x(i)=i$, there is a relationship between the possible $u_i$ for which $x(i)=tu_i+\sigma_x(i)$ and possible $m_i$ for which $\sigma_x(i)=sm_i+\tau_x(i)$. However, if $\tau_x(i)\not=i$, then no such relationship is required and so the values $l_i$ such that $x(i)=sl_i+\tau_x(i)$ may be anything smaller than $\frac{n}{s}$ and their only dependence is that $l_{\tau_x(i)}=-j'l_i\mod\frac{n}{s}$. 

Then we sum.

$$C_{n,t,s,r} = \underbrace{\sum_{j_{\sigma_x}\in E_{t/s}}}_{\text{choices for }\sigma_x(s)}\underbrace{\sum_{j'\in \overline{E}_{j_{\sigma_x},n,t,s}}}_{\text{choices for }x(s)}\underbrace{\sum_{l=0}^{\lfloor\frac{s-m_{r,j_{\sigma_x},t/s}}{2}\rfloor}}_{\substack{\# \text{ of fixed} \\ \text{ points of }\tau_x}}\,\,\,\underbrace{\delta_{j_{\sigma_x},t/s,r}}_{\sigma_x\text{ exists}}\,\,\,\underbrace{\delta_{j_{\sigma_x},t/s,r}^l}_{x\text{ and }\sigma_x\text{ agree}}\,\,\underbrace{\binom{s-1-2l}{m_{r,j_{\sigma_x},t/s}-1}}_{\substack{\text{ choose }m_i\\ \text{ in }P_{j_{\sigma_x},t/s}}}\qquad\,\qquad\,$$
$$\qquad\underbrace{|P_{j_{\sigma_x},t/s}|^{m_{r,j_{\sigma_x},t/s}-1}}_{\substack{\tau_x(i)=i\text{ and }\\ m_i=q(1-j_{\sigma_x})}}\underbrace{(\delta^{P^c})^{s-2l-m_{r,j_{\sigma_x},t/s}}}_{\substack{\tau_x(i)=i\text{ and }\\ m_i\not=q(1-j_{\sigma_x})}}\underbrace{\frac{(s-1)!}{(s-1-2l)!2^ll!}}_{\text{ choices for }\tau_x}$$ $$\qquad\qquad\qquad\underbrace{|\overline{K}_{j',j_{\sigma_x},n,t,s}|^{m_{r,j_{\sigma_x},t/s}-1}}_{\substack{u_i \text{ satisfying req.}\\ \text{ when }\tau_x(i)=i\text{ and }m_i\in P_{j_{\sigma_x},t/s}}}\underbrace{(\delta^{K^c})^{s-2l-m_{r,j_{\sigma_x},t/s}}}_{\substack{u_i \text{ satisfying req.}\\ \text{ when }\tau_x(i)=i\text{ and }m_i\in P_{j_{\sigma_x},t/s}^c}}\underbrace{\left(\frac{n}{s}\right)^l}_{\tau_x(i)\not=i}-\underbrace{\sum_{\substack{p|s\\p\not=d}}C_{n,t,p,r}}_{\text{ remove over count}}$$

Specifically, if $s=1$, then $\sigma_x$ has full stabilizer and so $\sigma_x(i)=ij_{\sigma_x}$ for all $i.$ Thus, since $\sigma_x$ has $r$ fixed points, $\beta_{j_{\sigma_x},n/t}$ which is the number of $\sigma_x$ must equal $r$. Therefore, $$C_{n,t,1,r}=\sum_{j_{\sigma_x}\in E_{t/1}}\sum_{j'\in \overline{E}_{j_{\sigma_x},n,t,1}}\delta_r(\beta_{j_{\sigma_x},t/1})=\sum_{j_{\sigma_x}\in E_{t/1}}\delta_r(\beta_{j_{\sigma_x},t/1})|\overline{E}_{j_{\sigma_x},n,t,1}|.$$ 

In the case where $t=n$, then we are just counting inovlutions $x$ with $r$ fixed points and we then remove all overcounts.


This is the number of involutions $x$ which can be generated by a permutation $\sigma_x\in S_{t-1}$ with $r$ fixed points, and which are stabilized by $a^s$ for $s|t$. \qed\\

\begin{proposition}\label{prop:orbit with r involutions length t}
{\it Let $O_{n/t,r}$ be the set of orbits of length $t$ containing $r$ involutions. Then \begin{equation}|O_{n/t,r}|=\begin{cases} \frac{1}{r}|X_{n/t,r}|&\text{ if }r\not=0\\ \frac{1}{t}|M_{n/t}|-\sum_{r=1}^t\frac{1}{r}|X_{n/t,r}|&\text{ if }r=0\end{cases}.\label{ex:Ot}\end{equation} Where $X_{n/t,r}$ is from \cref{prop:number of ord 2 elements for given stab with certain fixed mod t} and $M_{n/t}$ is from \cref{prop:number of x with Fx=<a^t>}.}
\end{proposition}

\no{\it Proof.} If $r\not=0$, then we have seen an involution $x$ with stabiliser $\langle a^t\rangle$ is in an orbit containing $r$ total involutions if and only if it has $r$ fixed points smaller than $t$. Namely, the number of orbits containing $r$ involutions is $\frac{1}{r}|X_{n/t,r}|.$

If $r=0$, then $O_{n/t,0}$ is the number of orbits containing no involutions. This value can be computed by taking the total number of orbtis containing $t$ permutations which given by $\frac{1}{t}|M_{n/t}|$ and subtracting those orbits which contain at least one involution. Thus, we subtract off $\frac{1}{r}|X_{n/t,r}|$ for all $1\le r\le t$.
\qed\\

Lets take the time now to show that our formula here, unwieldy as it is, aligns perfectly with that of \cite{jm} in \cref{thm:[JM] counting all irrepts prime p} in the case where $n=p$.

\begin{proposition}\label{prop:our formula matches p}
{\it Our formula agrees with those of \cite{jm} in \cref{thm:[JM] counting all irrepts prime p}. That is, \begin{myenum}
    \item \label{prop:our formula matches p (1)} $m_{p,1}=\sum_{r=1}^p|O_{p/p,r}|$\\
    \item \label{prop:our formula matches p (2)} $m_{p,1}+m_{p,0}=\frac{1}{p}|M_{p/p}|$
    \item \label{prop:our formula matches p (3)} $p+1$ of the $p(p-1)$ $1$-dimensional irreps of $J_p$ have indicator $+1$
\end{myenum}}
\end{proposition}

\no{\it Proof.} 
\begin{myenum}
    \item Recall, that $m_{p,1}:=\frac{i_p-1}{p}-1$ is the number of orbits $\mathcal{O}_x$ of length $t=p$ containing at least one involution where $i_p$ is the number of involutions (including the identity) in $S_p.$ Note also that in the proof of \cref{prop:number of ord 2 elements for given stab}, we show that $$i_p=\sum_{l=0}^{\lfloor\frac{p}{2}\rfloor}\frac{p!}{(p-2l)!2^ll!}.$$

We would like to show that $$m_{p,1}=\sum_{r=1}^pO_{p/p,r}=\sum_{r=1}^p\frac{1}{r}|X_{p/p,r}|.$$

Now, we simplify, \begin{align*}
    |X_{p/p,r}|&=\frac{(n-1)!}{(r-1)!2^{\frac{n-r}{2}}\left(\frac{n-r}{2}\right)!}-|C_{p,p,1,r}|\\
    &=\frac{(p-1)!}{(r-1)!2^{\frac{p-r}{2}}\left(\frac{p-r}{2}\right)!}-\sum_{j_{\sigma_x}\in E_{p/1}}\delta_r(\beta_{j_{\sigma_x},p/1})|\overline{E}_{j_{\sigma_x},p,p,1}|\\
    &=\frac{(p-1)!}{(r-1)!2^{\frac{p-r}{2}}\left(\frac{p-r}{2}\right)!}-\delta_r(\beta_{1,p/1})|\overline{E}_{1,p,p,1}|-\delta_r(\beta_{p-1,p/1})|\overline{E}_{p-1,p,p,1}|\\
    &=\frac{(p-1)!}{(r-1)!2^{\frac{p-r}{2}}\left(\frac{p-r}{2}\right)!}-\delta_r(p)|\overline{E}_{1,p,p,1}|-\delta_r(1)|\overline{E}_{p-1,p,p,1}|\qquad\beta_{j,p/1}=\text{ gcd}(j-1,p)\\
    &=\frac{(p-1)!}{(r-1)!2^{\frac{p-r}{2}}\left(\frac{p-r}{2}\right)!}-\delta_r(p)-\delta_r(1)\qquad |\overline{E}_{j,p,p,1}|=1\text{ for }j=1,p-1\\
\end{align*}

Therefore, \begin{align*}
    \sum_{r=1}^p\frac{1}{r}|X_{p/p,r}|&=\sum_{r=1}^p\frac{1}{r}\left[\frac{(p-1)!}{(r-1)!2^{\frac{p-r}{2}}\left(\frac{p-r}{2}\right)!}-\delta_r(p)-\delta_r(1)\right]\\
    &=\sum_{r=1}^p\left[\frac{(p-1)!}{r!2^{\frac{p-r}{2}}\left(\frac{p-r}{2}\right)!}-\frac{\delta_r(p)}{r}-\frac{\delta_r(1)}{r}\right]\\
    &=\left[\sum_{r=1}^p\frac{(p-1)!}{r!2^{\frac{p-r}{2}}\left(\frac{p-r}{2}\right)!}\right]-\frac{1}{1}-\frac{1}{p}\\
    &=\left[\sum_{l=0}^\frac{p-1}{2}\frac{(p-1)!}{(p-2l)!2^ll!}\right]-1-\frac{1}{p}\qquad\text{substituting }l=\frac{p-r}{2}\\
    &=\frac{1}{p}\left[\sum_{l=0}^\lfloor\frac{p}{2}\rfloor\frac{p!}{(p-2l)!2^ll!}\right]-1-\frac{1}{p}\qquad\text{since }p\text{ is odd }\left\lfloor\frac{p}{2}\right\rfloor=\frac{p-1}{2}\\
    &=\frac{1}{p}i_p-1-\frac{1}{p}\\
    &=\frac{i_p-1}{p}-1
\end{align*}

as desired.  Of course, \cref{thm:[JM] indicator when Fx=1} tells us that each permutation $x$ satisfying $F_x=F$ and $x^{-1}\in\mathcal{O}_x$ induces a single irrep with indicator $+1$. Thus, there are $pm_{p,1}$ total $p$-dimensional irreps with indicator $+1$.

\item The number of orbits of length $p$ is given by $$\frac{1}{p}|M_{p/p}|==\frac{1}{p}((p-1)!-\varphi(p))=(p-1)\frac{(p-2)!-1}{p}$$ which also aligns with \cref{thm:[JM] counting all irrepts prime p}.

\item Finally, $|M_{p/1}|=\varphi(p)=p-1$, so there are $p(p-1)$ $1$-dimensional irreps of $J_p$. Of these, $|T_{p/1}|=|E_{p/1}|=2$. 

Thus, there are $p(p-3)$ irreps of $J_p$ of dimension $1$ with indicator $0$ (induced from non-involutions). Of the two involutions with full stabilizer, the identity induces the trivial irrep which has indicator $+1$ and the rest have indicator $0$, and a non-identity $x$ which satisfies $x(1)=p-1$ and so plugging into \cref{thm:indicator reduced}, gives all $p$ irreps have indicator $+1$. That is, there are $p+1$ total $1$-dimensional irreps of $J_p$ with indicator $+1$ and the rest have indicator $0$. \qed\\
\end{myenum}

Let us summarize what we have found so far. We have seen that any permutation $x\in S_{n-1}$ with stabilizer $\langle a^t\rangle$ has a remainder permutation $\sigma_x\in S_{t-1}$ such that $x(i)\mod t=\sigma_x(i)$ for all $1\le i<t$. Furthermore, $\sigma_x$ shares certain properties with $x$. That is, if $x^{-1}=a^{-x(s)}xa^s\in\mathcal{O}_x$, then $\sigma_x^{-1}=\sigma_{x^{-1}}=a_t^{-\sigma_x(s)}\sigma_xa_t^s\in\mathcal{O}_{\sigma_x}$ where $a_t$ is the standard $t$ cycle in $S_t$.

We have used remainder permutations to help us determine whether $\mathcal{O}_x$ contains an involution. Most helpfully, we showed that if $x$ is an involution, then the permutation $y_l=a^{-x(l)}xa^l\in\mathcal{O}_x$ is an involution if and only if $\sigma_x(l)=l$ (\cref{lem:y in orbit x}).

This connection between permutations in $S_n$ and $S_t$ was an unexpected one and we believe the matched pair action connects $S_n$ with $S_t$ where $t|n$ in more ways than we have been able to uncover in this paper.

\section{Counting Irreps with a given dimension}\label{sec:counting irreps}

Here we state one of our greatest results. Given an $n$ and any odd divisor $t$, we can explicitly write down and count the number of $t$-dimensional irreps of $J_n$ with indicator $+1$ and indicator $0.$

\subsection{Counting the Number of odd dimensional Irreps with Indicator \texorpdfstring{$+1$}{TEXT}}\label{sec:indicator formula (odd)}

\begin{proposition}\label{prop:orbit with r involutions length t and fixed j}
{\it Let $O_{n/t,r,j}$ be the set of orbits of length $1<t<n$ containing $r$ involutions where for each permutation $y$ in any of these orbits, $y(t)=jt$. Recall entries (1), (2), (4), (6), (7), (9), (10), (11), (12), (13) from \cref{fig:table of sets and values} (\nameref{fig:table of sets and values}). Then, \begin{equation}r|O_{n/t,r,j}|=(\alpha_{j,n/t})^{r-1}\left(\frac{n}{t}\right)^\frac{t-r}{2}\frac{(t-1)!}{(r-1)!2^\frac{t-r}{2}\left(\frac{t-r}{2}\right)!}-\sum_{\substack{s|t\\s\not=t}}\tilde{C}_{n,t,s,r,j}\label{ex:Otj}\end{equation}
    where \begin{align}
    \delta_{j,j',n/t}'&=\begin{cases} 1 &\text{ if } j't=jt\mod n\\ 0&\text{ otherwise}\end{cases}\nonumber\\
    \tilde{C}_{n,t,1,r,j}&=\sum_{j_{\sigma_x}\in E_{t/1}}\sum_{j'\in \overline{E}_{j_{\sigma_x},n,t,1}}\delta_{j,j',n/t}'\delta_r(\beta_{j_{\sigma_x},t/1,r})\label{ex:C't}\\
    \tilde{C}_{n,t,s,r,j}&=\sum_{j_{\sigma_x}\in E_{t/s}}\sum_{j'\in \overline{E}_{j_{\sigma_x},n,t,s}}\sum_{l=0}^{\lfloor\frac{s-m_{r,j_{\sigma_x},t/s}}{2}\rfloor}\delta_{j,j',n/t}'\delta_{j_{\sigma_x},t/s,r}\binom{s-1-2l}{m_{r,j_{\sigma_x},t/s}-1}|P_{j_{\sigma_x},t/s}|^{m_{r,j_{\sigma_x},t/s}-1}\nonumber\\
    &\qquad\qquad(\delta^{P^c})^{s-2l-m_{r,j_{\sigma_x},t/s}}\frac{(s-1)!}{(s-1-2l)!2^ll!}|\overline{K}_{j',j_{\sigma_x},n,t,s}|^{m_{r,j_{\sigma_x},t/s}-1}(\delta^{K^c})^{s-2l-m_{r,j_{\sigma_x},t/s}}\left(\frac{n}{s}\right)^l\nonumber\\
    &\qquad\qquad\qquad\qquad\qquad\qquad-\sum_{\substack{p|s\\p\not=d}}\tilde{C}_{n,t,p,r,j}\label{ex:C'tp}
\end{align} }
\end{proposition}

\no{\it Proof.} This is a direct adaptation of \cref{prop:number of ord 2 elements for given stab with certain fixed mod t}. The only change is that we are assuming a fixed $j$ here and so we omit the first sum over the set $E_{n/t}$. Additionally, in the overcount function we must ensure that we only remove those $x$ which have the same $j$ value. Namely, since our choice of $j'$ such that $x(s)=j's$ forces $x(t)=j't$, we must have that $j't=jt$ for our fixed $j.$

Therefore, we include the value $$\delta_{j,j',n/t}'=\begin{cases} 1 &\text{ if } j't=jt\mod n\\ 0&\text{ otherwise}\end{cases}$$ in our overcount function.
\qed\\

\begin{theorem}\label{thm:number of positive indicators of dimension} 
{\it Let $t$ be an odd divisor of $n.$ Let $I_{n/t}^{(+1)}$ be the number of representations of $J_n$ of dimension $t$ with indicator $+1$. Recall (3) from \cref{fig:table of sets and values} (\nameref{fig:table of sets and values}) and $O_{n/t,r,j}$ as from \cref{prop:orbit with r involutions length t and fixed j}. Then \begin{align}
    I_{n/1}^{(+1)}&=\sum_{j\in E_{n/1}}\alpha_{j,n/1}\label{ex:I1+}\\
    I_{n/t}^{(+1)}&=t\sum_{j\in E_{n/t}}\sum_{r=1}^t\alpha_{j,n/t}O_{n/t,r,j}\label{ex:It+}\\
    I_{n/n}^{(+1)}&=n\sum_{r=1}^nO_{n/n,r}\label{ex:In+}
\end{align} Finally, the number of irreps of $J_n$ of dimension $t$ with indicator $0$ is given by \begin{equation}I_{n/t}^{(0)}=\frac{n}{t}|M_{n/t}|-I_{n/t}^{(+1)}\label{ex:It0}\end{equation} where $M_{n/t}$ is from \cref{prop:number of x with Fx=<a^t>}.}
\end{theorem}

\no{\it Proof.} Let $t$ be an odd divisor of $n$ and $x$ a permutation with stabilizer $F_x=\langle a^t\rangle$ such that $x(t)=jt$ and $\mathcal{O}_x$ contains at least one involution. 

Let $\alpha_{j,n/t}=\text{gcd}(j+1,\frac{n}{t})$, this is the number of $i$ such that $\zeta_n^{i\frac{x(t)+t}{t}}=\zeta_n^{i\frac{jt+t}{t}}=\zeta_n^{i(j+1)}=1$. Namely, $\alpha_{j,n/t}$ is the number of (group) irreps $\rho_{x,i}$ such that the indicator of the induced character $\nu(\hat{\chi}_{x,i})=+1$ from \cref{thm:indicator reduced}.

Since \cpref{prop:observations}{prop:observations (1)} tells us that each $y\in\mathcal{O}_x$ gives the same number of indicators of value $+1$, we have $t\alpha_{j,n/t}$ total indicators of value $+1$ coming from this orbit.  

Then, the number of representations $\hat{V}$ of $J_n$ of dimension $t$ with indicator $+1$ is given by $$t\sum_{j\in E_t}\sum_{r=1}^t\alpha_{j,n/t}O_{n/t,r,j}.$$

Of course, if $t=n$ then each $x$ with stabilizer $F_x=\{1\}$ either satisfies that $x^{-1}\in\mathcal{O}_x$ and so induces a single irrep of $J_n$ with indicator $+1$, or $x^{-1}\not\in\mathcal{O}_x$ and so $x$ induces an irrep of indicator $0$ (see \cref{thm:[JM] indicator when Fx=1}). Since $t=n$ is odd by assumption, any orbit closed under inversion must contain at least one involution. Thus, the number of irreps of dimension $n$ with indicator $+1$ is given by counting the number of orbits containing at least one involution and multiplying by $n.$ This gives $$I_{n/n}^{(+1)}=n\sum_{r=1}^nO_{n/n,r}.$$

Finally, since there are $|M_{n/t}|$ permutations with stabilizer $\langle a^t\rangle$ (proved in \cref{prop:number of x with Fx=<a^t>}) and each of these generates $\frac{n}{t}$ irreps of $J_n$, there are $\frac{n}{t}|M_{n/t}|$ total irreprs of $J_n$ of dimension $t.$ Since when $t$ is odd, there are no irreps with indicator $-1$ (from \cref{prop:observations}), we conclude that $I_{n/t}^{(0)}=\frac{n}{t}|M_{n/t}|-I_{n/t}^{(+1)}.$
\qed\\

\cref{thm:number of positive indicators of dimension} gives an explicit method of constructing the irreps of dimension $t$ of $J_n$ which have indicator $+1$ in the case where $t$ is odd. Recall that $|M_{n/t}|$ is the number of permutations $S_{n-1}$ which have stabilizer $F_x=\langle a^t\rangle$ and that this value is equivalent to the number of permutations in $S_{n-1}$ which induce an irrep of $J_n$ of dimension $t$ (see \cref{prop:number of x with Fx=<a^t>} and \cref{thm:[KMM] inducing modules} and \cref{lem:[KMM] dimension of induced rep}).

\begin{corollary}\label{cor:statements}
{\it Let $t$ be an odd divisor of $n$. Then, \begin{myenum}
\item every involution $x\in S_{n-1}$ with $F_x=\langle a^t\rangle$ induces at least one $t$-dimensional irrep of $J_n$ with indicator $+1$
\item {\it NOT} every $t$-dimensional irrep of $J_n$ with indicator $+1$ is induced from an involution
\item every $t$-dimensional irrep of $J_n$ with indicator $+1$ is induced from a permutation whose orbit contains at least one involution
\end{myenum}
}
\end{corollary}

\subsection{The Case \texorpdfstring{$t=2$}{TEXT}}\label{sec:t=2}

Now, we turn our attention to the otherwise completely illusive even case. Using some computation, we are able to derive formulas when $t=2$, although due to the increased difficulty when considering even $t$, we do not believe this is a viable strategy for deriving relationships in general.

\begin{proposition}\label{prop:number x with t=2}
{\it Let $n$ be even and let $x\in S_{n-1}$ and $F_x=\langle a^2\rangle$ (so $t=2$). Let $x(2)=2j$ where $j$ satisfies $j^2= 1\mod \frac{n}{2}$. Then, the number of $x$ such that $x^{-1}\in\mathcal{O}_x$ with $x^{-1}=a^{-x(s)}xa^s$ is the number $u$ satisfying
\begin{myenum}
    \item \label{prop:number x with t=2 (1)} $j\not= 2u+1\mod \frac{n}{2}$
    \item and either \label{prop:number x with t=2 (2)} \begin{myenum}
        \item \label{prop:number x with t=2 (2) case 1} $u(j+1)= 0 \mod\frac{n}{2}$ in which case $s=t$ (equivalently, $s=0$)
        \item \label{prop:number x with t=2 (2) case 2} or $(u+1)(j-1)= 0\mod \frac{n}{2}$ in which case $s=1$
    \end{myenum} 
\end{myenum}

In fact, the or is exclusive and so the number of $x\in S_{n-1}$ with $F_x=\langle a^2\rangle$ is the number of solutions to $u$ to $u(j+1)= 0 \mod\frac{n}{2}$ (which counts the number of involutions $x$) plus the number of solutions $(u+1)(j-1)= 0\mod \frac{n}{2}$ (which counts the number of non-involutions $x$). Here we are excluding all $u$ for which $j\not= 2u+1\mod \frac{n}{2}$ in both cases. 

Furthermore, all such $x$ are defined by $x(1)=2u+1$ for any $u$ satisfying the above. }
\end{proposition}

\no{\it Proof.} \cref{lem:x deterimined by remainders mod t for Fx=<a^t>} tells us that we need only understand $x(1)$ and $x(2)$ to fully determine $x$. Additionally, $x(2)=2j$ where $j^2= 1\mod \frac{n}{2}$ by \cref{lem:x(t)=kt some k}. 

Now, because $x$ sends even numbers to even numbers and sends odd numbers to odd numbers, we can let $x(1)=2u+1$ for some $u$. 

\begin{enumerate}[label=\textbf{(\roman*)}]
    \item Now, assume that $x^2=1$. we have that \begin{align*}
    1&=x^2(1)=x(x_1)=x(2u+1)=ux(2)+x(1)=2uj+x_1\\
    \implies x_1&=1-2uj=2u+1\\ 
    \implies 2u(j+1)&= 0 \mod n\\ 
    \implies u(j+1)&= 0\mod\frac{n}{2}.
\end{align*}

In fact, any $u$ satisfying this formula will guarantee that $x^2=1$ since it ensures that $x^2(1)=1$ and we already have that $x^2(2)=2$. However, similarly to the proof of \cref{prop:number of x with Fx=<a^t>}, we will have over counted. Namely, we must remove all $x$ such that $x^2=1$ and $F_x=\langle a\rangle$. Namely, we cannot have that $$x(2)=2j=2x(1)=2(2u+1)=4u+2\\\implies j= 2u+1\mod\frac{n}{2}$$ so we throw out all $u$ satisfying this formula. 
\item 
We proceed similarly with the second case: $x^{-1}=a^{-x(1)}xa$. This time, we note that \begin{align*}
    1&=x^{-1}(2u+1)=ux^{-1}(2)+x^{-1}(1)\\
    &=ux(2)+x^{-1}(1)=2uj+x^{-1}(1)\\
    &=2uj+x(2)-x(1)=2uj+2j-2u-1\\ 
\end{align*} Which implies that $$0=2uj+2j-2u-2=2j(u+1)-2(u+1)=2(u+1)(j-1).$$ Thus, $(u+1)(j-1)=0\mod\frac{n}{2}$. Again, all $u$ and $j$ satisfying this relation will induce an element $x$ for which $x^{-1}(i)=x(i+1)-x(1)$. Of course, once again, we must remove the case where $x(2)=2j=2x(1)$. 
\end{enumerate}

Since these are the only two cases, and since both are satisfied by the same $u$ exactly when $j= 2u+1\mod\frac{n}{2}$, which we have already ruled out, we obtain all permutations $x$ with stabilizer $F_x=\langle a^2\rangle$ such that $x^{-1}\in\mathcal{O}_x$ broken down into the two disjoint cases where $x^2=1$ and $x^2\not=1.$ 
\qed\\

Actually, \cref{prop:number x with t=2} gives us a full description of what permutations give what indicators.

Note that the way we prove \cref{prop:number x with t=2} is by checking cases for $x$. Since $x(2)$ must be even, we force $x(1)=2u+1$ for some $u$. Then, setting $x=x^{-1}$ or assuming $x\not=x^{-1}$ we rule out options for $u$.

The next smallest even value $t$ can take is $4$. But here our cases explode. $x(4)$ must certainly be a multiple of $4$, but can have any remainder permutation in $S_3$. For example, we could let $x(1)=4u_1+2$ and $x(2)=4u_2+1$ and $x(3)=4u_3+3$ meaning $x$ would have remainder permutation $\sigma_x=(1\pp 2)\in S_3$. From each of these six cases we would need to work through possible choices of $u_1,u_2,u_3$ which lead to either $x$ being an involution or not. This would result in twelve total cases to consider.

Clearly this method cannot scale, since for any given even $t$, we would need to consider all $(t-1)!$ possible ways to write $x$ then double it to check the conditions under which $x$ is an involution or not an involution.

However, we can still get some insight by at least understanding $t=2$.

\begin{theorem}\label{thm:classifying indicators when t=2}
{\it Let $n$ be even and let $x\in S_{n-1}$ and $F_x=\langle a^2\rangle$ (so $t=2$). Let $I_{n/2}^{(+1)}, I_{n/2}^{(-1)}, I_{n/2}^{(0)}$ be the total number of irreps of $J_n$ dimension $2$ which have indicator $+1,-1,0$ respectively. Recall (1), (2), (3), (14), (15), (16), (17) in \cref{fig:table of sets and values} (\nameref{fig:table of sets and values}).

Then, \begin{align}
   I_{n/2}^{(+1)}&=\sum_{j\in E_{n/2}}\sum_{u\in K_{j,n/2}}2\overline{\delta}_{u,j,n/2}\left(\alpha_{j,n/2}+\sum_{i\in K_{j,n/2}}\delta_{i,u,n/2}^0\right)\label{I2+}\\
   I_{n/2}^{(-1)}&=\sum_{j\in E_{n/2}}\sum_{u\in K_{j,n/2}'}\sum_{i\in K_{j,n/2}}2\overline{\delta}_{u,j,n/2}\delta_{i,u,n/2}^{\not=0}\label{I2-}\\
   I_{n/2}^{(0)}&=\frac{n}{2}|M_{n/2}|-I_{n/2}^{(+1)}-I_{n/2}^{(-1)}\label{I20}
\end{align}

 }
\end{theorem}

\no{\it Proof.} Let $$\overline{\delta}_{u,j,n/2}=\begin{cases} 1 & \text{ if }j\not=2u+1\mod\frac{n}{2}\\ 0 & \text{ otherwise}\end{cases}$$ and $$K_{j,n/2}'=\left\{u\,|(u+1)(j-1)=0\mod\frac{n}{2}\,\right\}.$$

We break into two cases. 
\begin{myenum}
    \item If $x^2=1$ then, because $u(j+1)=0\mod\frac{n}{2}$ always has a solution (namely, $u=0$) there are $\alpha_{j,n/2}=\text{gcd}(j+1,\frac{n}{2})$ possible $u$ which satisfy this relation. Let $$K_{j,n/2}:=\left\{u\,|\,u(j+1)=0\mod\frac{n}{2}\right\}.$$ Again, $|K_{j,n/2}|=\alpha_{j,n/2}$ Thus, there are $z_{j,n/2}:=\frac{n}{2}-\alpha_{j,n/2}$ possible $u$ which satisfy $(u+1)(j-1)=0\mod\frac{n}{2}$.

Now, \cref{prop:observations}, tells us that if $\mathcal{O}_x$ contains an involution, $\nu(\hat{\chi})\ge0$. Thus, if $s^2=1$ (so we are in the first case), irreps induced from $x$ either have indicator $+1$ or $0$. Using \cref{thm:indicator reduced}, the indicator is $+1$ for all values $1\le i\le \frac{n}{2}$ such that $i(j+1)=0\mod\frac{n}{2}$. There are exactly $\alpha_{j,n/2}$ such $i$. Let $\mathcal{O}_x=\{x,y\}$ (both are involutions in this case), then $x$ contributes $\alpha_{j,n/2}$ indicators which are $+1$ and so does $y.$ Namely, there are $$\underbrace{\sum_{j\in E_{n/2}}}_{\text{choices for }x(2)}\underbrace{\sum_{u\in K_{j,n/2}}}_{\text{choices for }x(1)}\underbrace{\overline{\delta}_{u,j,n/2}}_{x^{-1}\in\mathcal{O}_x}\underbrace{2}_{x\text{ and }y}\underbrace{\alpha_{j,n/2}}_{\substack{\text{ choices for }i\\ \text{ where } i(j+1)=0\mod\frac{n}{2}}}$$ indicators with value $+1$ which are contributed to by involutions. 
    \item If $x^2\not=1$, then because $t=2$ $\mathcal{O}_x=\{x,x^{-1}\}$ so $x^{-1}=a^{-x(1)}xa$. Now, $x(1)+1=2u+2=2(u+1)$ and so $u_2=u+1$ as in \cref{thm:indicator reduced}. Now, for any $i$, if $i(j+1)=0\mod\frac{n}{2}$, then either $-i(u+1)=0\mod\frac{n}{2}$ in which case the indicator is $+1$ for that irrep, or $-i(u+1)=\frac{n}{4}\mod\frac{n}{2}$ in which case the indicator is $-1$. Let $$\delta_{i,u,n/2}^0=\begin{cases} 1 & \text{ if }-i(u+1)=0\mod\frac{n}{2}\\ 0 & \text{ otherwise}\end{cases}$$ $$\delta_{i,u,n/2}^{\not=0}=\begin{cases} 1 & \text{ if }-i(u+1)\not=0\mod\frac{n}{2}\\ 0 & \text{ otherwise}\end{cases}$$

Of course, if $i(j+1)\not=0\mod\frac{n}{2}$, then the indicator is $0.$

Finally, the number of $+1$ indicators coming from this case is $$\underbrace{\sum_{j\in E_{n/2}}}_{\text{choices for }x(2)}\underbrace{\sum_{u\in K_{j,n/2}'}}_{\text{choices for }x(1)}\underbrace{\sum_{i\in K_{j,n/2}}}_{\substack{\text{choices for }i\\ \text{ where } i(j+1)=0\mod\frac{n}{2}}}\underbrace{2}_{x\text{ and }x^{-1}}\underbrace{\overline{\delta}_{u,j,n/2}}_{\text{ ensure }j\not=2u+1}\underbrace{\delta_{i,u,n/2}^0}_{\text{ ensure }-i(u+1)=0}$$

And the number of $-1$ indicators come from the remaining values $i$ for which $i(j+1)=0$ but $-i(u+1)\not=0$. That is, $$\underbrace{\sum_{j\in E_{n/2}}}_{\text{choices for }x(2)}\underbrace{\sum_{u\in K_{j,n/2}'}}_{\text{choices for }x(1)}\underbrace{\sum_{i\in K_{j,n/2}}}_{\substack{\text{choices for }i\\ \text{ where } i(j+1)=0\mod\frac{n}{2}}}\underbrace{2}_{x\text{ and }x^{-1}}\underbrace{\overline{\delta}_{u,j,n/2}}_{\text{ ensure }j\not=2u+1}\underbrace{\delta_{i,u,n/2}^{\not=0}}_{\text{ ensure }-i(u+1)\not=0}$$
\end{myenum}

Finally, we have that the total number of irreps with indicator $+1$ is \begin{align*}
    I_{n/2}^{(+1)}&=\sum_{j\in E_{n/2}}\sum_{u\in K_{j,n/2}}\overline{\delta}_{u,j,n/2}2\alpha_{j,n/2}+\sum_{j\in E_{n/2}}\sum_{u\in K_{j,n/2}'}\sum_{i\in K_{j,n/2}}2\overline{\delta}_{u,j,n/2}\delta_{i,u,n/2}^0\\
    &=\sum_{j\in E_{n/2}}\sum_{u\in K_{j,n/2}}2\overline{\delta}_{u,j,n/2}\left(\alpha_{j,n/2}+\sum_{i\in K_{j,n/2}}\delta_{i,u,n/2}^0\right)\\
\end{align*}

and the minus ones come only from the second case

$$I_{n/2}^{(-1)}=\sum_{j\in E_{n/t}}\sum_{u\in K_{j,n/2}'}\sum_{i\in K_{j,n/2}}2\overline{\delta}_{u,j,n/2}\delta_{i,u,n/2}^{\not=0}$$

and the rest must be zero $$I_{n/2}^{(0)}=\frac{n}{2}|M_{n/2}|-I_{n/2}^{(+1)}-I_{n/2}^{(-1)}\eqno\qed$$

And so we now give a summary of the results we were able to obtain. Fix a positive integer $n.$

\begin{center}
  \begin{mytable}[Table of Obtainable Results]\label{fig:table of obtainable}\end{mytable} 
  
{\begin{tabular}{ c|c|c|c } 
 \hline
 Statement & Conditions & Result & Citation\\
 \hline
 \hline
\multicolumn{1}{m{4cm}|}{List all permutations with a given stabilizer size} & Fix $t|n$ & \multicolumn{1}{m{4.5cm}|}{ There are $|M_{n/t}|$ permutations with stabilizer $F_x=\langle a^t\rangle$} & \cref{prop:number of x with Fx=<a^t>}\\
 \hline
 \multicolumn{1}{m{4cm}|}{List all involutions with a given stabilizer size} & Fix $t|n$ & \multicolumn{1}{m{4.5cm}|}{ There are $|T_{n/t}|$ involutions with stabilizer $F_x=\langle a^t\rangle$} & \cref{prop:number of ord 2 elements for given stab}\\
 \hline
   \multicolumn{1}{m{4cm}|}{List all involutions with given stabilizer size and $r$ fixed points} & Fix $t|n$ and $r\le n$ & \multicolumn{1}{m{4.5cm}|}{ There are $|R_{n/t,r}|$ involutions with $r$ fixed points and stabilizer $F_x=\langle a^t\rangle$} & \cref{prop:number of ord 2 elements for given stab with certain fix}\\
 \hline
  \multicolumn{1}{m{4cm}|}{List all orbits of length $t$ containing $r$ involutions} & Fix $t|n$ and $r\le t$ & \multicolumn{1}{m{4.5cm}|}{ There are $|\frac{X_{n/t,r}}{r}|$ permutations with stabilizer $F_x=\langle a^t\rangle$} & \cref{prop:orbit with r involutions length t}\\
 \hline
   \multicolumn{1}{m{4cm}|}{Count number of $t$-dimensional irreps of $J_n$ with indicator $+1$} & Fix $t|n$ with $t$ odd & \multicolumn{1}{m{4.5cm}|}{There are $I_{n/t}^{(+1)}$ irreps of $J_n$ of dimension $t$ with indicator $+1$} & \cref{thm:number of positive indicators of dimension}\\
 \hline
    \multicolumn{1}{m{4cm}|}{Count number of $t$-dimensional irreps of $J_n$ with indicator $0$} & Fix $t|n$ with $t$ odd & \multicolumn{1}{m{4.5cm}|}{There are $I_{n/t}^{(0)}=\frac{n}{t}|M_{n/t}|-I_{n/t}^{(+1)}$ irreps of $J_n$ of dimension $t$ with indicator $0$} & \cref{thm:number of positive indicators of dimension}\\
 \hline
    \multicolumn{1}{m{4cm}|}{Count number of $t$-dimensional irreps of $J_n$ with indicator $+1$} & Fix $t|n$ with $t$ even & \multicolumn{1}{m{4.5cm}|}{$I_{n/t}^{(+1)}$ is an \textit{underestimate} for total irreps of $J_n$ of dimension $t$ with indicator $+1$} & \cref{thm:number of positive indicators of dimension}\\
 \hline
    \multicolumn{1}{m{4.1cm}|}{Describe indicator of $t$-dimensional irreps of $J_n$} & If $t=2$ & \multicolumn{1}{m{4.5cm}|}{Irreps of dimension $t$ with indicators $+1,-1, 0$ given by $I_{n/2}^{(+1)},I_{n/2}^{(-1)},I_{n/2}^{(0)}$ respectively} & \multicolumn{1}{m{3cm}}{\cref{thm:classifying indicators when t=2} \cref{prop:number x with t=2}}\\
 \hline
\end{tabular}}
\end{center}
\noindent{\small Note: A table summarizing some results which are now obtainable through a counting formula.}\\

 We end this section with a corollary mentioned earlier.
 
\begin{corollary}\label{cor:exists irrep of dim -1}
{\it If $4$ divides $n$ and $n\ge12$, there exists an irrep of $J_n$ of dimension $2$ of with indicator $-1.$}
\end{corollary}

\no{\it Proof.} Assume $n\ge12$. Let $t=2$. Then let $j=1\in E_{n/2}$ for any $n\ge 12$ (in fact, $1\in E_{n/t}$ for any $\frac{n}{t}>1$). Let $u=2.$ Then $(u+1)(j-1)=(3)(0)=0\mod\frac{n}{2}$ for any $n\ge12$. Thus, the permutation $$x=(1\pp5\pp9\pp\cdots \pp n-3)(3\pp 7\pp 11\pp \cdots \pp n-1)$$ satisfies that $\mathcal{O}_x=\{x,x^{-1}\}$ and for $i=\frac{n}{4}$, the (group) irrep \begin{align*}
       \rho_{x,\frac{n}{4}}:\langle a^2\rangle &\to\mathbb{C}^\times\\
       a^2&\mapsto (\zeta_{\frac{n}{2}})^{\frac{n}{4}}
   \end{align*} satisfies that $i(j+1)=2\frac{n}{4}=0\mod\frac{n}{2}$ but $i(u+1)=3\frac{n}{4}\not=0\mod\frac{n}{2}$. Therefore, the induced $J_n$ module $\hat{V}_{x,\frac{n}{4}}$ has indicator $-1.$
 \qed\\

\section{Limiting Behavior}\label{sec:limit behavior}

\cref{thm:number of positive indicators of dimension} suggests that the conditions that a permutation $x\in S_{n-1}$ must meet in order to induce an irrep of $J_n$ with indicator $+1$ are very restrictive. This intuition turns out to be correct and can be formalized in the following theorem. However, to prove it, we must give a few bounding relationships.

\begin{lemma}\label{lem:some bounds}
{\it Let $E_{n/t}$, where $\frac{n}{t}\ge 2$, be as from (1) in \cref{fig:table of sets and values} (\nameref{fig:table of sets and values}), $\varphi(m)$ be the Euler totient function which is the number of integers smaller than $m$ which are coprime to $m$, and finally let $\omega(m)$ is the number of distinct prime divisors of $m$. Then for all $m\ge3$ \begin{myenum}
    \item \label{lem:some bounds (1)} $|E_{n/t}|\le 2^{\omega(n/t)+2}$ 
    \item \label{lem:some bounds (2)} $\omega(m)\le 1.3841\frac{\log m}{\log(\log m)}$
    \item \label{lem:some bounds (3)} $\varphi(m)\ge\frac{1}{\sqrt{2}}\sqrt{m}$ 
    \item \label{lem:some bounds (4)} $\displaystyle\lim_{m\to\infty}\frac{2^{1.3841\frac{\log(m)}{\log(\log (m))}}}{\sqrt{m}}=0$
\end{myenum}}
\end{lemma}

\no{\it Proof.} 
\begin{myenum}
    \item This is a well known result which comes from the Chinese Remainder Theorem. Note the special case if $\frac{n}{t}=2$ then $E_{n/t}=\{1\}$ and $\omega(n/t)=\omega(2)=1$ so the bound provided holds. 
\item  Robin proves this in \cite{r}.
\item Let $m=p_1^{\alpha_1}p_2^{\alpha_2}\cdots p_u^{\alpha_u}$. Then since $\alpha_i\ge 1$ $$\frac{\varphi(m)^2}{m}=\prod_{i=1}^u\frac{(p_i^{\alpha_i-1}(p_i-1))^2}{p_i^{\alpha_i}}=\prod_{i=1}^up_i^{\alpha_i-2}(p_i-1)^2\ge\prod_{i=1}^up_i^{-1}(p_i-1)^2=\prod_{i=1}^u\frac{(p_i-1)^2}{p_i}\ge\frac{1}{2}.$$

Note the last step is because if $p_i>2$, then $(p_i-1)^2=p_i^2-2p_i+1=p_i(p_i-2)+1\ge p_i+1\ge p_i$ and so $\frac{(p_i-1)^2}{p_i}\ge 1$. Of course, if $p_i=2$ then $\frac{(p_i-1)^2}{p_i}=\frac{1}{2}$.

Finally, we get that $$\frac{\varphi(m)^2}{m}\ge\frac{1}{2}\implies \varphi(m)\ge \frac{1}{\sqrt{2}}\sqrt{m}.$$
\item This is elementary calculus. Let $L=\lim_{m\to\infty}\frac{2^{1.3841\frac{\log(m)}{\log(\log (m))}}}{\sqrt{m}}=0$. Then a standard logarithm trick gives
\begin{align*}
    \log L&=\lim_{m\to\infty}\left(1.3841\frac{\log(m)}{\log(\log (m))}\log 2-\log\left(\sqrt{m}\right)\right)\\
    &=\lim_{m\to\infty}-\frac{1}{2}\log(m)=-\infty\\
\end{align*}
and so $L=0.$\qed\\
\end{myenum}

Now, we can make precise what we mean by saying it is rare for a permutation to satisfy the conditions we have discussed so far. 
 
But first, we show that there are many more involutions than permutations with stabilizer $\langle a^t\rangle.$

\begin{lemma}\label{lem:M/inv approaches 0}
{\it Fix $t$ and let $n=mt$ for $m\ge2.$ Let $i_{n-1}$ be the number of involutions in $S_{n-1}$. Then $\frac{|M_{n/t}|}{i_{n-1}}\to0$ as $m\to\infty.$}
\end{lemma}

\no{\it Proof.} Fix $t$ and let $n=mt$ for $m\ge2$. Note that $|M_{n/t}|\le\left(\frac{n}{t}\right)^tt!$. Now, in the proof of \cref{prop:number of ord 2 elements for given stab}, we show that $$i_{n-1}=\sum_{l=0}^{\lfloor\frac{n-1}{2}\rfloor}\frac{(n-1)!}{(n-1-2l)!2^ll!}.$$

In fact, \cite{chm} show that $i_{n-1}\ge\sqrt{n-1}i_{i-2}$ and so inductively we obtain that $i_{n-1}\ge\sqrt{(n-1)!}\ge2^{\frac{n-1}{2}}$. Note that $(n-1)!\ge 2^{n-1}$ for $n\ge 5$ which is easily provable by induction.

Thus, $$\lim_{m\to\infty}\frac{|M_{n/t}|}{i_{n-1}}\le\lim_{m\to\infty}\frac{\left(m\right)^tt!}{2^{\frac{mt-1}{2}}}=0$$ The last limit is zero using L'Hopital's rule.
\qed\\

Additionally, we can show that involutions with a given stabilizer are rare when compared to general permutations with a given stabilizer. 

\begin{lemma}\label{lem:T/M approaches 0}
{\it Fix $t$ and let $n=mt$ for $m\ge2.$ Then $\frac{|T_{n/t}|}{|M_{n/t}|}\to0$ as $m\to\infty.$}
\end{lemma}

\begin{proof}
\begin{myenum}
    \item To start, let $t=1$, then $$\lim_{m\to\infty}\frac{|T_{n/1}|}{|M_{n/1}|}=\lim_{m\to\infty}\frac{|E_{n/1}|}{\varphi(m)}\le\lim_{m\to\infty}\frac{2^{1.3841\frac{\log(m)}{\log(\log (m))}}}{\sqrt{m}}=0$$ by \cref{lem:some bounds}.
    \item 
Now, assume $t>1$. We will find an upper bound for the numerator and a lower bound for the denominator to bound the entire ratio.

First, we set an upper bound for $|T_{n/t}|$. Let $n=mt$ for $m=\frac{n}{t}\ge2$.

\begin{align*}
    |T_{n/t}|&=\sum_{j\in E_{n/t}}\sum_{l=0}^{\left\lfloor\frac{t-1}{2}\right\rfloor}(\alpha_{j,n/t})^{t-1-2l}\left(\frac{n}{t}\right)^l\frac{(t-1)!}{(t-1-2l)!2^ll!}-\sum_{\substack{s|t \\ s\not=t}}|T_{n/s}|\\
    &\le\sum_{j\in E_{n/t}}\sum_{l=0}^{\left\lfloor\frac{t-1}{2}\right\rfloor}(\alpha_{j,n/t})^{t-1-2l}\left(\frac{n}{t}\right)^l\frac{(t-1)!}{(t-1-2l)!2^ll!}\qquad\text{ forget }-\sum_{\substack{s|t \\ s\not=t}}|T_{n/s}|\\
    &\le |E_{n/t}|\sum_{l=0}^{\left\lfloor\frac{t-1}{2}\right\rfloor}\left(\frac{n}{t}\right)^{t-1-l}\frac{(t-1)!}{(t-1-2l)!2^ll!}\qquad\alpha_{j,n/t}\le\frac{n}{t}\\
    &\le |E_{n/t}|\left(\frac{n}{t}\right)^{t-1}t!\\
    &\le \left(\frac{n}{t}\right)^{t-1}t!2^{1.3841\frac{\log m}{\log(\log m)}+2}\qquad\text{\cpref{lem:some bounds}{lem:some bounds (2)}}
\end{align*}

Second, we set a lower bound for $|M_{n/t}|$ for $t>1$. Note that $|M_{n/s}|\le \left(\frac{n}{s}\right)^s(s-1)!$ and so
\begin{align*}
    \frac{n}{t}|M_{n/t}|&=\varphi\left(\frac{n}{t}\right)\left(\frac{n}{t}\right)^{t-1}(t-1)!-\sum_{\substack{s|t\\ s\not=t}}|M_{n/s}|\\
    &\ge \frac{1}{\sqrt{2}}\sqrt{\frac{n}{t}}\left(\frac{n}{t}\right)^{t-1}(t-1)!-\sum_{\substack{s|t\\ s\not=t}}|M_{n/s}|\qquad\text{ \cpref{lem:some bounds}{lem:some bounds (3)}}\\
    &\ge \frac{1}{\sqrt{2}}\sqrt{\frac{n}{t}}\left(\frac{n}{t}\right)^{t-1}(t-1)!-\sum_{\substack{s|t\\ s\not=t}}\left(\frac{n}{s}\right)^s(s-1)!
\end{align*}

Finally, we get \begin{align*}
    \lim_{m\to\infty}\frac{|T_{n/t}|}{|M_{n/t}|}&\le\lim_{m\to\infty}\frac{\left(\frac{n}{t}\right)^{t-1}t!2^{1.3841\frac{\log m}{\log(\log m)}+2}}{\frac{1}{\sqrt{2}}\sqrt{\frac{n}{t}}\left(\frac{n}{t}\right)^{t-1}(t-1)!-\sum_{\substack{s|t\\ s\not=t}}\left(\frac{n}{s}\right)^s(s-1)!}\\
    &= \lim_{m\to\infty}\frac{\left(m\right)^{t-1}t!2^{1.3841\frac{\log m}{\log(\log m)}+2}}{\sqrt{m}\left(m\right)^{t-1}\left(\frac{1}{\sqrt{2}}(t-1)!-\left(\frac{1}{m}\right)^{t-1}\sum_{\substack{s|t\\ s\not=t}}\left(m\frac{t}{s}\right)^s(s-1)!\right)}\\
    &= \lim_{m\to\infty}4\sqrt{2}t\frac{2^{1.3841\frac{\log m}{\log(\log m)}}}{\sqrt{m}}\\
    &=0\qquad\text{\cpref{lem:some bounds}{lem:some bounds (4)}}\qedhere
\end{align*}
\end{myenum}
\end{proof}

We now give one of our most powerful results from this section.

\begin{theorem}\label{thm:percent +1 ind goes to zero} 
{\it Fix $t$ an odd integer. Let $n=mt$ for some integer $m\ge2$. Then the number ratio of irreps of $J_n$ of dimension $t$ with indicator $+1$ to the total number of irreps of dimension $t$ converges to $0$ as $m\to\infty$.}
\end{theorem}

\begin{proof} The number of permutations $S_{n-1}$ which have stabilizer $F_x=\langle a^t\rangle$ is given by $|M_{n/t}|$ by \cref{prop:number of x with Fx=<a^t>}. There are $\frac{n}{t}=|F_x|$ irreducible (group) representations of $F_x$ and by \cref{thm:[KMM] inducing modules} and \cref{lem:[KMM] dimension of induced rep}, each of these group irreps of $F_x$ induce a unique irrep of $J_n$ of dimension $t=[F:F_x]$.

Thus, there are $\frac{n}{t}|M_{n/t}|$ total irreps of $J_n$ of dimension $t.$

From \cref{thm:number of positive indicators of dimension}, $I_{n/t}^{(+1)}$ is the number of irreps of $J_n$ of dimension $t$ with indicator $+1$ and so the ratio is given by $$\frac{I_{n/t}^{(+1)}}{\frac{n}{t}|M_{n/t}|}.$$

To show convegence to $0$, we will place an upper bound on $I_{n/t}^{(+1)}$ and a lower bound on $\frac{n}{t}|M_{n/t}|$ and then show that this ratio converges to zero. We do this in two cases, first for $t>1$ then for $t=1$.

\begin{myenum}
\item When $t=1$, $|M_{n/1}|=\varphi(n)\ge\frac{1}{\sqrt{2}}\sqrt{n}$ and $$I_{n/1}^{(+1)}=\sum_{j\in E_{n/1}}\alpha_{j,n/1}\le n2^{1.3841\frac{\log(m)}{\log(\log (m))}+2}$$ so $$\lim_{n\to\infty}\frac{I_{n/1}^{(+1)}}{n|M_{n/1}|}\le\lim_{n\to\infty}\frac{4n2^{1.3841\frac{\log(n)}{\log(\log (n))}}}{\frac{1}{\sqrt{2}}n\sqrt{n}}=\lim_{n\to\infty}4\sqrt{2}\frac{2^{1.3841\frac{\log(n)}{\log(\log (n))}}}{\sqrt{n}}=0$$
    \item 
First, an upper bound on $I_{n/t}^{(+1)}$. Let $n=mt$ for $m\ge2$ and $t>1$.
\begin{align*}
    I_{n/t}^{(+1)}&=t\sum_{j\in E_{n/t}}\sum_{r=1}^t\alpha_{j,n/t}O_{n/t,r,j}\\
    &\le \frac{t}{r}\sum_{j\in E_{n/t}}\sum_{r=1}^t\alpha_{j,n/t}(\alpha_{j,n/t})^{r-1}\left(\frac{n}{t}\right)^\frac{t-r}{2}\frac{(t-1)!}{(r-1)!2^\frac{t-r}{2}\left(\frac{t-r}{2}\right)!}\qquad\text{ ignore }-\tilde{C}\text{ terms}.\\
    &\le\sum_{j\in E_{n/t}}\sum_{r=1}^t\left(\frac{n}{t}\right)^r\left(\frac{n}{t}\right)^\frac{t-r}{2}\frac{t!}{r!2^\frac{t-r}{2}\left(\frac{t-r}{2}\right)!}\qquad \alpha_{j,n/t}\le\frac{n}{t}\\
    &\le \sum_{r=1}^t2^{\omega(n/t)+2}\left(\frac{n}{t}\right)^\frac{t+r}{2}\frac{t!}{r!2^\frac{t-r}{2}\left(\frac{t-r}{2}\right)!}\qquad\text{ \cpref{lem:some bounds}{lem:some bounds (1)}}\\
    &\le \sum_{r=1}^t2^{1.3841\frac{\log(n/t)}{\log(\log (n/t))}+2}\left(\frac{n}{t}\right)^\frac{t+r}{2}\frac{t!}{r!2^\frac{t-r}{2}\left(\frac{t-r}{2}\right)!}\qquad\text{ \cpref{lem:some bounds}{lem:some bounds (2)}}
\end{align*}

Now, using the same lower bound as from \cref{lem:T/M approaches 0}, we get $$\frac{n}{t}|M_{n/t}|\ge\frac{n}{t}\left( \frac{1}{\sqrt{2}}\sqrt{\frac{n}{t}}\left(\frac{n}{t}\right)^{t-1}(t-1)!-\sum_{\substack{s|t\\ s\not=t}}\left(\frac{n}{s}\right)^s(s-1)!\right).$$

Finally, noting that $\frac{n}{t}=m$ and $\frac{n}{s}=m\frac{t}{s}$, the ratio converges \begin{align*}
    \lim_{m\to\infty}\frac{I_{n/t}^{(+1)}}{\frac{n}{t}|M_{n/t}|}&\le \lim_{m\to\infty}\frac{\sum_{r=1}^t2^{1.3841\frac{\log(m)}{\log(\log (m))}+2}\left(m\right)^\frac{t+r}{2}\frac{t!}{r!2^\frac{t-r}{2}\left(\frac{t-r}{2}\right)!}}{m\left(\frac{1}{\sqrt{2}}\sqrt{m}\left(m\right)^{t-1}(t-1)!-\sum_{\substack{s|t\\ s\not=t}}\left(m\frac{t}{s}\right)^s(s-1)!\right)}\\
    &=\lim_{m\to\infty}\frac{2^{1.3841\frac{\log(m)}{\log(\log (m))}}\left(\left(m\right)^t+\sum_{r=1}^{t-1}4\left(m\right)^\frac{t+r}{2}\frac{t!}{r!2^\frac{t-r}{2}\left(\frac{t-r}{2}\right)!}\right)}{\sqrt{m}\left(m\right)^t\left(\frac{1}{\sqrt{2}}(t-1)!-\left(\frac{t}{n}\right)^{t-1}\sum_{\substack{s|t\\ s\not=t}}\left(m\frac{t}{s}\right)^s(s-1)!\right)}\\
    &=\lim_{m\to\infty}\frac{2^{1.3841\frac{\log(m)}{\log(\log (m))}}\left(m\right)^t\left(1+\sum_{r=1}^{t-1}4\left(m\right)^\frac{r-t}{2}\frac{t!}{r!2^\frac{t-r}{2}\left(\frac{t-r}{2}\right)!}\right)}{\sqrt{m}\left(m\right)^t\left(\frac{1}{\sqrt{2}}(t-1)!-\left(\frac{1}{m}\right)^{t-1}\sum_{\substack{s|t\\ s\not=t}}\left(m\frac{t}{s}\right)^s(s-1)!\right)}\\
    &=\lim_{m\to\infty}\frac{2^{1.3841\frac{\log(m)}{\log(\log (m))}}}{\sqrt{m}\frac{1}{\sqrt{2}}(t-1)!}=0\qquad\text{ \cpref{lem:some bounds}{lem:some bounds (4)}}\qedhere
\end{align*}
\end{myenum}
\end{proof}

\cref{thm:percent +1 ind goes to zero} tells us that \textit{odd} $t$-dimensional irreps of $J_n$ with nonzero indicator become rare as $n$ grows. In fact, we conjecture that for any $t$, not only odd $t$, the number of irreps of $J_n$ of dimension $t$ with nonzero indicator approaches $0.$

We summarize the results of these limit results as follows.

\begin{corollary}\label{cor:limit statements}
{\it Fix $t$ and let $n=mt$. Then for large $m$ (and so large $n$)
\begin{myenum}
    \item $i_{n-1}$ is much less than $(n-1)!$
    \item $|M_{n/t}|$ is much less than $i_{n-1}$
    \item $|T_{n/t}|$ is much less than $|M_{n/t}|$
    \item if $t$ is odd, $I_{n/t}^{(+1)}$ is much less than $|M_{n/t}|$
\end{myenum}}
\end{corollary}

\no{\it Proof.} The only claim we have not proved is the first part. However, this follows from another inequality of \cite{chm} which states that for $n\ge3$, $i_{n-1}\le (\sqrt{n-1}+1)i_{n-2}\le \sqrt{n}i_{n-2}\le\sqrt{n!}$.

Thus, $$\lim_{n\to\infty}\frac{i_{n-1}}{(n-1)!}\le\lim_{n\to\infty}\frac{\sqrt{n!}}{(n-1)!}=\lim_{n\to\infty}\sqrt{\frac{n}{(n-1)!}}=0\eqno\qed$$

Now, using the formulas from \cref{thm:classifying indicators when t=2}, we can again show that the proportion of $2$-dimensional irreps of $J_n$ with nonzero indicator approaches $0.$

\begin{theorem}\label{thm:percent nonzero ind goes to zero} 
{\it Let $n=2m$ for some integer $m\ge2$. Then the number ratio of irreps of $J_n$ of dimension $2$ with either indicator $+1$ or $-1$ (so nonzero indicator) to the total number of irreps of dimension $2$ converges to $0$ as $m\to\infty$.}
\end{theorem}

\no{\it Proof.} We follow the same outline as \cref{thm:percent +1 ind goes to zero}. 

From \cref{thm:classifying indicators when t=2}, we can use the following bounds. \begin{align*}
    I_{n/2}^{(+1)}&\le \sum_{j\in E_{n/2}}\sum_{u\in K_{j,n/t}}2(2\alpha_{j,n/t})\qquad |K_{j,n/t}|=\alpha_{j,n/t}\text{ and set both }\delta\text{ functions to }1\\
    &\le \sum_{j\in E_{n/t}}4(\alpha_{j,n/t})^2\qquad |K_{j,n/t}|=\alpha_{j,n/t}\\
    &\le 2^{\omega(n/2)+4}\left(\frac{n}{2}\right)^2\qquad \text{ Lemma }\cpref{lem:some bounds}{lem:some bounds (1)}
\end{align*}

Then, letting $n=2m$ for some $m\ge2$ and using the same lower bound for $\frac{n}{2}|M_{n/2}|$ we obtain $$\lim_{m\to\infty}\frac{I_{n/2}^{(+1)}}{\frac{n}{2}|M_{n/2}|}\le \lim_{m\to\infty}\frac{2^{\omega(m)+4}m^2}{\sqrt{m}m^2\frac{1}{\sqrt{2}}}\le\lim_{m\to\infty}16\sqrt{2}\frac{2^{1.3841\frac{\log(m)}{\log(\log (m))}}}{\sqrt{m}}=0$$ using \cpref{lem:some bounds}{lem:some bounds (4)}. 

Similarly, $$ I_{n/2}^{(-1)}\le 2^{\omega(n/2)+3}\left(\frac{n}{2}\right)^2$$ and so $$\lim_{m\to\infty}\frac{I_{n/2}^{(-1)}}{\frac{n}{2}|M_{n/2}|}\le\lim_{m\to\infty}8\sqrt{2}\frac{2^{1.3841\frac{\log(m)}{\log(\log (m))}}}{\sqrt{m}}=0.$$

Thus, $$\lim_{m\to\infty}\frac{I_{n/2}^{(+1)}+I_{n/2}^{(-1)}}{\frac{n}{2}|M_{n/2}|}=\lim_{m\to\infty}\frac{I_{n/2}^{(+1)}}{\frac{n}{2}|M_{n/2}|}+\lim_{m\to\infty}\frac{I_{n/2}^{(-1)}}{\frac{n}{2}|M_{n/2}|}=0$$ and so the ratio of irreps of dimension $2$ of $J_{2m}$ which have nonzero indicator to the total number of irreps of dimension $2$ approaches $0$ as $m\to\infty.$
 \qed\\

\newpage

\appendix
\section*{Appendices}
\addcontentsline{toc}{section}{Appendices}
\renewcommand{\thesubsection}{\Alph{subsection}}

\section{Discussing the Python Code}\label{sec:app discussing code}

Here we walk through examples that illustrate how use the theorems and propositions of the previous section. We provide graphs of the various formulas and show some direct computations to aid with visualizing the indicators.

Throughout, we will let $t$ be a divisor of $n$.

The process of writing the recursive counting formulas of \cref{sec:explicitly counting} went hand in hand with directly computing examples using Python. In fact, all counting formulas proved in \cref{sec:explicitly counting} were cross checked against what was reasonable to compute directly using Python.

At first, efforts were made to do direct computation by actually multiplying permutations in Python using a series of combinatorics packages similar to those used in GAP. Using this method, we ran out of processing power around $n=12.$ Using the reformulation of the action described in \cref{sec:reinterpretting the action}, we were able to drastically increase the computations possible. For example, when $n<50$ and $t<9$, we can easily describe all indicators of irreps of $J_n$ of dimension $t.$

We are limited by the factorial growth of this problem and so when doing computations it is important to choose either that $n$ is small ($n<20$ computes quickly) or that $t$ is small. For example, if we fix $t=2$, then it took less than 30 minutes to describe all irreps of $J_n$ of dimension $2$ for $n\le 150$ (see \cref{sec:app using thm classifying indicators when t=2}).

Let us understand why the value of $t$ is so closely tied with computational speed. In Python, each permutation is treated as a tuple. For example, if $n=12$, the permutation $$x = (1\pp3\pp5)(2\pp4)(7\pp9\pp11)(8\pp 10)$$ is written in Python as $$x=[x(1),x(2),...,x(n)]=[3,4,5,2,1,6,9,10,11,8,7,12].$$ Now, using \cref{prop:equivalence of sets}, if we want to compute the orbit of $x$ we simply take shifts of all the values of $x$. For example, the permutation $y_1=a^{-x(1)}xa\in\mathcal{O}_x$ is characterized by the fact that $y_1(i)=x(1+i)-x(1)$ and so using the previous example, we obtain easily that $$y_1=[y_1(1),y_1(2),...,y_1(n)]=[x(2)-x(1),x(3)-x(1),...,x(1)-x(1)]=[1,2,11,10,3,6,7,8,5,4,9,12]$$ we can construct all permutation in the orbit of $x$ by taking shifts of the values of $x$. Namely, given a permutation $x\in S_{n-1}$ it is computationally easy to construct the orbit of $x$.

\begin{center}
    \includegraphics[width=0.75\textwidth]{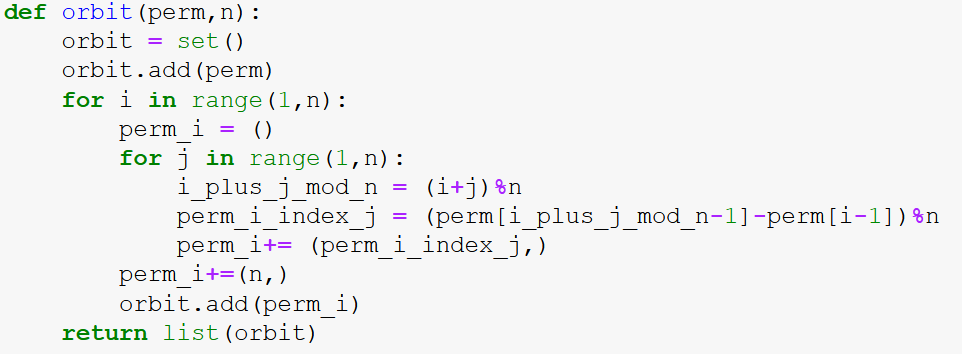}
    \begin{myfigure}\label{fig:code orbit calc}{ Excerpt from python code which calculates the orbit of a given permutation for fixed $n.$ } \end{myfigure}
\end{center}

Now, the smaller $t$ gets, the more restrictions there are on $x$ in order for it to have stabilizer $\langle a^t\rangle$. For example, when $t=1$, $x$ is completely determined by where it sends $1$ and so there are $\varphi(n)$ (the Euler totient function of $n$) permutation which have stabilizer $\langle a\rangle$ (by \cref{prop:number of x with Fx=<a^t>}). If $t=2$ then $x$ is determined by where it sends $1$ and $2$, etc. Thus, as $t$ grows closer to $n$, the number of $x$ with stabilizer $\langle a^t\rangle$ approaches $(n-1)!$.

Lets say now that we want to compute all orbits of size $t$ for a given $n$ which contain $r$ involutions. First, we generate all the permutation in $S_{n-1}$, then we compute all their orbits, then we pick out those of length $t,$ then of those we pick out the ones which contain $r$ involutions. To count the number of involutions in an orbit, we need to manually check if each permutation in the set is its own inverse. As $t$ grows, not only there more permutations within the orbit to compare with their inverses, but also the total number of orbits of length $t$ is approaching $(n-1)!$ which cannot be done efficiently for large $n.$

For this reason, we demanded that $t$ stayed small (usually $t<8$) throughout our computations, especially as $n$ grew large. Note the $y$-axis scaling on many of the graphs is nonlinear. This is because the values differ by so much that they cannot be visualized on the same graph without logarithmically scaling the $y$-axis.

\section{Using \texorpdfstring{\cref{prop:number of x with Fx=<a^t>}}{Proposition \ref{prop:number of x with Fx=<a^t>}}}\label{sec:app using prop number of x with given stab}

Recall from \cref{prop:number of x with Fx=<a^t>}, that $M_{n/t}$ is the set of permutations in $S_{n-1}$ with stabilizer $F_x=\langle a^t\rangle$.

Using the proof of \cref{prop:number of x with Fx=<a^t>}, we can explicitly write down which permutations have stabilizer $\langle a^t\rangle$. 

\begin{center}
    \begin{tikzpicture}[style={font=\small}, node distance = 5.5cm, auto]
     \node[rectangle, scale=0.65, draw, fill=blue!20, text width = 9em, text centered, rounded corners, minimum height = 4em] (xt) {\normalsize choose $x(t)=jt$ for $j\in E_{n/t}$};
     \node[rectangle, scale=0.65, draw, fill=blue!20, text width = 9em, text centered, rounded corners, minimum height = 4em, right of = xt] (x1) {\normalsize choose $x(1)$ in $\{u\,|\,u\not=0\mod t\}$};
     \node[rectangle, scale=0.65, draw, fill=blue!20, text width = 9em, text centered, rounded corners, minimum height = 4em, right of = x1] (x2) {\normalsize choose $x(2)$ in $\{u\,|\,u\not=0, x(1)\mod t\}$};
     \node[rectangle, scale=0.65, text width = 2em, text centered, minimum height = 4em, right of = x2,node distance = 3.5cm] (dots) {\Large $\cdots$};
     \node[rectangle, scale=0.65, draw, fill=blue!20, text width = 9em, text centered, rounded corners, minimum height = 4em, right of = dots,node distance = 3.5cm] (xn-1) {\normalsize choose $x(n-1)$ in $\{u\,|\,u\not=0, x(1), x(2), ... , x(n-2)\mod t\}$};
     \path[draw,->] (xt) -- (x1);
     \path[draw,->] (x1) -- (x2);
     \path[draw] (x2) -- (dots);
     \path[draw,->] (dots) -- (xn-1);
    \end{tikzpicture}
\vspace{-0.25cm}
    \begin{myfigure}\label{fig:chart of constructing x with stab a^t}{ A flow chart illustrating the process of \cref{prop:number of x with Fx=<a^t>}. This will construct all permutations which are stabilized by $a^t$. To get the exact set of permutations which have stabilizer $\langle a^t\rangle$, this process must be done recursively so that all $x$ which are stabilized by $a^s$ for some $s|t$ can be removed from the set.}\end{myfigure} 
\end{center}

\begin{example}\label{ex n=12, find all perms with stab a^t}$\,$

For example, lets say $n=12$ we want to write down all permutations $x\in S_{11}$ with stabilizer $\langle a^2\rangle$. We break this process into steps as follows: we determine all permutations with full stabilizer, we determine all permutations stabilized by $a^2$, then we remove those we have already found.

Since $n=12$ and $t=2$, we will use (5) from \cref{prop:number of x with Fx=<a^t>}.

\begin{myenum}
    \item First need to determine all permutations $x\in S_{11}$ with stabilizer $\langle a\rangle=F$. 
    
    All $x$ with full stabilizer are determined by where they send $1$. Since the number of values coprime to $12$ are $1,5,7,11$ the number of such $x$ is $4$. Furthermore, they are given by \begin{align*}
    j=1: x(i)=i\implies x&=(1)\\
    j=5: x(i)=5i\implies x&=(1\pp5)(2\pp10)(4\pp8)(7\pp 11)\\
    j=7: x(i)=7i\implies x&=(1\pp7)(3\pp9)(5\pp 11)\\
    j=11: x(i)=11i\implies x&=(1\pp11)(2\pp 10)(3\pp 9)(4\pp 8)(5\pp 7)
\end{align*}

\item Now, we compute all permutations with are stabilized by $2$. Since the only $j$ which are coprime to $\frac{12}{2}=6$ are $1$ and $5$, we have that $x(2)=2$ or $x(2)=10$. Then, $x(1)$ can be any odd value and since there are six choices $\{1,3,5,7,9,11\}$, we have $12$ total permutations.

\begin{align*}
    \alignedbox{j=1; x(1)=1: x}{=(1)}\\
    j=1; x(1)=3: x&=(1\pp3\pp5\pp7\pp9\pp11)\\
    j=1; x(1)=5: x&=(1\pp5\pp9)(3\pp7\pp11)\\
        \alignedbox{j=1; x(1)=7: x}{=(1\pp7)(3\pp9)(5\pp11)}\\
    j=1; x(1)=9: x&=(1\pp9\pp5)(3\pp11\pp7)\\
    j=1; x(1)=11: x&=(1\pp11\pp9\pp7\pp5\pp3)\\
    &\\
    j=5; x(1)=1: x&=(2\pp10)(3\pp11)(4\pp8)(5\pp9)\\
    j=5; x(1)=3: x&=(1\pp3)(2\pp10)(4\pp8)(5\pp11)(7\pp9)\\
    \alignedbox{j=5; x(1)=5: x}{=(1\pp5)(2\pp10)(4\pp8)(7\pp11)}\\
    j=5; x(1)=7: x&=(1\pp7)(2\pp10)(3\pp5)(4\pp8)(9\pp11)\\
    j=5; x(1)=9: x&=(1\pp9)(2\pp10)(3\pp7)(4\pp8)\\
    \alignedbox{j=5; x(1)=11: x}{=(1\pp11)(2\pp10)(3\pp9)(4\pp8)(5\pp7)}
\end{align*}

\item Of course, the highlighted permutations were those we found previously to have full stabilizer and so those should be removed from our list. This leaves us with our final total of $M_{12/2}=8$ permutations as counted by \cref{remark3 counting M}. 
\end{myenum}
\end{example}

\begin{center}
    \includegraphics[width=0.95\textwidth]{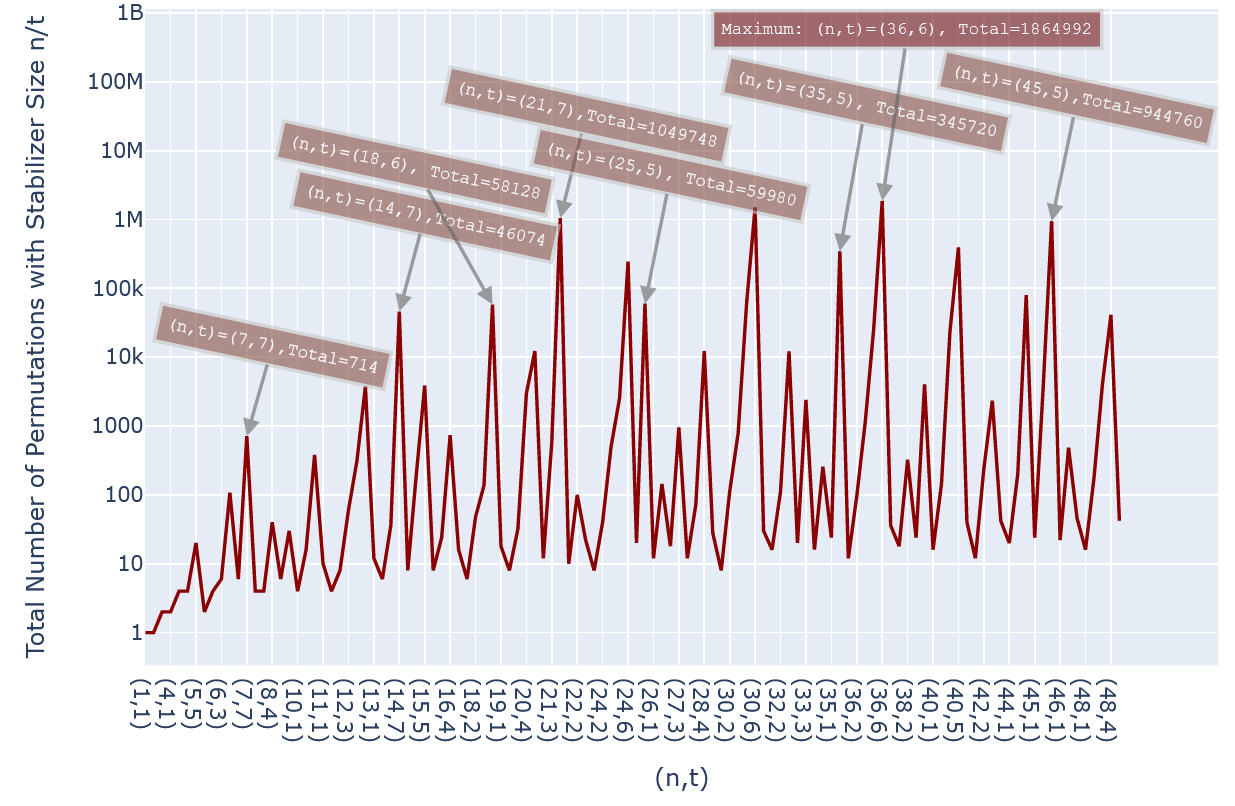}
    \vspace{-0.5cm}
    \begin{myfigure}\label{fig:graph all perms with given stab size} Graph of $M_{n/t}$: number of permutations $x\in S_{n-1}$ with stabilizer $\langle a^t\rangle$. \end{myfigure}
\end{center}

\section{Using \texorpdfstring{\cref{prop:number of ord 2 elements for given stab}}{Proposition \ref{prop:number of ord 2 elements for given stab}}}\label{sec:app using prop number of ord 2 ele for given stab}

Recall from \cref{prop:number of ord 2 elements for given stab}, that $T_{n/t}$ is the set of involutions in $S_{n-1}$ with stabilizer $F_x=\langle a^t\rangle$. 

\begin{center}
    \begin{tikzpicture}[style={font=\small}, node distance = 3.5cm, auto]
     \node[rectangle, scale=0.65, draw, fill=blue!20, text width = 9em, text centered, rounded corners, minimum height = 4em] (xt) {\normalsize choose $x(t)=jt$ for $j\in E_{n/t}$};
     \node[rectangle, scale=0.65, draw, fill=blue!20, text width = 9em, text centered, rounded corners, minimum height = 4em, right of = xt,node distance=5.5cm] (sigma) {\normalsize choose an involution $\sigma_x\in S_{t-1}$};
     \node[rectangle, scale=0.65, draw, fill=blue!20, text width = 9em, text centered, rounded corners, minimum height = 4em, right of = sigma, above of =sigma] (xfix) {\normalsize if $\sigma_x(i)=i$, choose $u_i$ such that $u_i(j+1)=0\mod\frac{n}{t}$};
     \node[rectangle, scale=0.65, draw, fill=blue!20, text width = 9em, text centered, rounded corners, minimum height = 4em, right of = sigma,below of =sigma] (xnotfix) {\normalsize if $\sigma_x(i)\not=i$, choose any $u_i$ which determines $u_{\sigma_x(i)}=-ju_i\mod\frac{n}{t}$};
     \node[rectangle, scale=0.65, draw, fill=blue!20, text width = 9em, text centered, rounded corners, minimum height = 4em, right of = xfix, below of = xfix] (x) {construct $x$};
     \path[draw,->] (xt) -- (sigma);
     \path[draw,->] (sigma) -- (xfix);
     \path[draw,->] (sigma) -- (xnotfix);
     \path[draw,->] (xfix) -- (x);
     \path[draw,->] (xnotfix) -- (x);
    \end{tikzpicture}
    
    \begin{myfigure}\label{fig:chart of constructing involution x with stab a^t}{ A flow chart illustrating the process of \cref{prop:number of ord 2 elements for given stab}. This will construct all involutions which are stabilized by $a^t$. To get the exact set of involutions which have stabilizer $\langle a^t\rangle$, this process must be done recursively so that all $x$ which are stabilized by $a^s$ for some $s|t$ can be removed from the set.}\end{myfigure} 
\end{center}

\begin{example}\label{n=6,t=3 find all inv with stab a^t}$\,$

As an example, say we want to find $T_{6/3}$. As always, the first step is to compute $T_{6/1}$ since these permutations must be removed. Of course, any permutation which is stabilized by $a^3$ and is not stabilized by $a$ must have stabilizer $\langle a^3\rangle.$ Again we break this process into steps.

Since $n=6$ and $t=1$, we will use (6) from \cref{prop:number of ord 2 elements for given stab}

\begin{myenum}
    \item First, we find all $x$ with full stabilizer. $1$ is the only divisor of $3$ not equal to $3$. Each $j\in E_{6/1}=\{1,5\}$ gives a single involution with full stabilizer: \begin{align*}
    j=1: x&=(1)\\
    j=5: x&=(1\pp5)(2\pp4)
\end{align*}

\item Second, we find all permutations stabilized by $a^3.$ There are two involutions in $S_2$ which can act as remainder permutations. They are $\sigma_1=(1)$, $\sigma_2=(1\pp2)$. Here $E_{6/3}=\{1\}$ so $j=1$ is fixed. Now, we group by remainder permutations.
\begin{myenum}
    \item Say $\sigma_x=(1)$ so every $i$ is a fixed point of $\sigma_x$. Then, since $u_i(j+1)=2u_i=0\mod 2=\frac{6}{3}$, $u_i$ can take any value $\{0,1\}$ for $i=1,2$. This gives the possible permutations
    
    \begin{align*}
    \alignedbox{\sigma_x=(1); u_1=0,u_2=0; x(1)=1,x(2)=2: x}{=(1)}\\
    \sigma_x=(1); u_1=0,u_2=1; x(1)=1,x(2)=5: x&=(2\pp5)\\
    \sigma_x=(1); u_1=1,u_2=0; x(1)=4,x(2)=2: x&=(1\pp4)\\
    \sigma_x=(1); u_1=1,u_2=1; x(1)=4,x(2)=5: x&=(1\pp4)(2\pp5)\\
\end{align*}
    
    \item On the other hand, if $\sigma_x=(1\pp2)$ then $\sigma_x$ has no fixed points and so $u_1=-u_2=u_2\mod2$. Again, we use these values to write $x(i)=tu_i+\sigma_x(i)$. This gives the possible permutations
\begin{align*}
    \sigma_x=(1\pp 2); u_1=0,u_2=0; x(1)=2,x(2)=1: x&=(1\pp2)(4\pp5)\\
    \alignedbox{\sigma_x=(1\pp 2); u_1=1,u_2=1; x(1)=5,x(2)=4: x}{=(1\pp5)(2\pp4)}
\end{align*}
\end{myenum}
\item Again, the highlighted involutions are those we already found to have full stabilizer. Remove those leaves us with $4$ involutions with stabilizer $\langle a^3\rangle.$
\end{myenum}
\end{example}

Note how in {\bf (2).(i)} all permutations $x$ are composed of the two $(2\pp 5)$ and $(1\pp 4)$ and their product. These two permutations are in fact all of the involutions of $S_5$ which consist of transpositions $(i\pp i+3)$. Of course $(3\pp 6)$ is not a valid involution in $S_5$.

Throughout our study of the sorts of permutations which satisfy various requirements (such as number of fixed points, size of stabilizer, {\it etc.}), we have stumbled upon many such ``patterns.'' The percieved lack of randomness invites endless questions. Are these patterns indicative of some structure, perhaps induced from $S_{t-1}$?

Of course, the pattern of $(i\pp i+3)$ does not hold for {\bf (2).(ii)}, nor does $(i\pp i+1)$ so are these patterns coincidental?

It is our belief that there is a much easier way to compute the sets defined in \cref{sec:explicitly counting} using some deeper understanding of these patterns. However, we have not explored these questions in this work.

\begin{center}
    \includegraphics[width=0.9\textwidth]{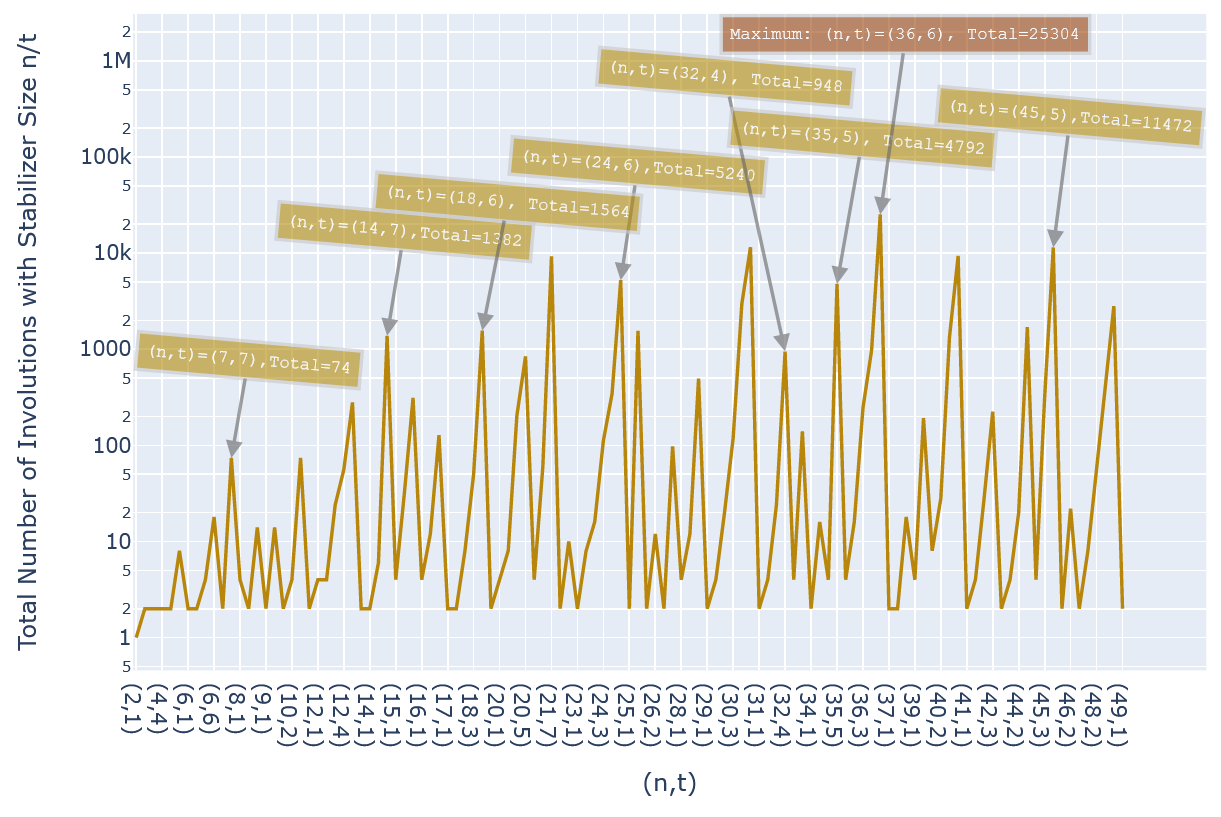}
     \vspace{-0.5cm}
    \begin{myfigure}\label{fig:graph all involutions with given stab size} Graph of $T_{n/t}$: number of involutions $x\in S_{n-1}$ with stabilizer $\langle a^t\rangle$.\end{myfigure}
\end{center}

\section{Using \texorpdfstring{\cref{prop:number of ord 2 elements for given stab with certain fix}}{Proposition \ref{prop:number of ord 2 elements for given stab with certain fix}}}\label{sec:app using prop number of order 2 ele for given stab with fix}

Recall from \cref{prop:number of ord 2 elements for given stab with certain fix}, that $R_{n/t,r}$ is the set of involutions in $S_{n-1}$ with stabilizer $F_x=\langle a^t\rangle$ and exactly $r$ fixed points (including $x(n)=n$ as a fixed point).

\begin{example}\label{n=8,t=4,r=2 all inv with fixed points}$\,$

The more restriction we place on $x$, the more complicated our formulas must be for constructing them. However, we will again showcase how to use \cref{prop:number of ord 2 elements for given stab with certain fix} to construct the set $R_{8/4,2}$ which is the number of involutions in $S_7$ with stabilizer $\langle a^4\rangle$ and $2$ fixed points (including $x(8)=8$ as a fixed point).

Since $n=8$, $t=4$, and $r=2$, we use (10) from \cref{prop:number of ord 2 elements for given stab with certain fix}

First, we need $R_{8/1,2}$ and $R_{8/2,2}$ so the process becomes more complicated here. 

\begin{myenum}
    \item First, we find all involutions with full stabilizer. For $t=1$, $E_{8/1}=\{1,3,5,7\}$. Since $\beta_{1,8/1}=8$, $\beta_{3,8/1}=2$, $\beta_{5,8/1}=4$, $\beta_{7,8/1}=2$ there are only two permutations with $2$ fixed points and full stabilizer coming from $j=3$ and $j=7$. This is because $\delta_{j,n/t,r}=0$ unless $\beta_{j,n/t}$ divides $r$.
\begin{align*}
    j=3: x&=(1\pp3)(2\pp6)(5\pp7)\\
    j=7: x&=(1\pp7)(2\pp6)(3\pp5)
\end{align*}
\item For $t=2$, $E_{8/2}=\{1,3\}$. $\beta_{1,8/2}=4$, $\beta_{3,8/2}=2$ so only $j=3$ is viable here. Since the only $\sigma\in S_1$ is the identity, $u_1\in P_{3,8/2}^{(c)}=\{1,3\}$ and so again we get two permutations

\begin{align*}
    \,\qquad\alignedbox{u_1=1; x(1)=3: x}{=(1\pp3)(2\pp6)(5\pp7)}\\
    \,\qquad\alignedbox{u_1=3; x(1)=7: x}{=(1\pp7)(2\pp6)(3\pp5)}
\end{align*}

Of course, both of these we found to have full stabilizer for $R_{8/2,2}=0.$
\item Now we compute $R_{8/4,2}$. First, $E_{8/4}=\{1\}$ so $j=1$ and $x(4)=4$ is fixed. Now, $\beta_{1,8/4}=\text{gcd}(1-1,\frac{8}{4})=\text{gcd}(0,2)=2$. Thus, $m_{2,1,8/4}=1$ and $\frac{t-m_{2,1,8/4}}{2}=\frac{3}{2}$ so $l=0$ or $l=1$ in the second sum of $R_{n/t,r}$. This tells us that $\sigma_x\in S_3$ can be any involution comprised of either $0$ or $1$ transpositions. In fact, this is all such permutations so we have possible choices for $\sigma_x$ are $(1),(1\pp2),(1\pp3),$ and $(2\pp3)$. 

The sets that we will be pulling from are \begin{align*}
    K_{1,8/4}&=\{0,1\}\\
    P_{1,8/4}&=\{0\}\\
    P_{1,8/4}^{(c)}&=\{1\}
\end{align*}
\begin{myenum}
    \item Say $\sigma_x=(1)$, then this permutation has $t-1-2l=3$ fixed points and of these we choose $m_{2,1,8/4}-1=0$ of the $u_i$ to be in $P_{1,8/4}$ so all $u_i$ are in $P_{1,8/4}^{(c)}$. 
    \item Similarly, if $\sigma_x$ is one of the other transpositions. There is one just one fixed point $i$ of $\sigma_x$ which forces $u_i\in P_{1,8/4}^{(c)}$ for that $i$. The two values swapped by the transposition of $\sigma_x$ equal (via the relation $u_{\sigma_x(i)}=-ju_i\mod\frac{n}{t}$). 
    \item Finally, we can write out all possible involutions. \begin{align*}
    \sigma_x=(1); u_1=1,u_2=1,u_3=1; x(1)=5,x(2)=6,x(3)=7: x&=(1\pp5)(2\pp6)(3\pp7)\\
    &\\
    \sigma_x=(1\pp 2); u_1=0,u_2=0,u_3=1; x(1)=2,x(2)=1,x(3)=7: x&=(1\pp2)(3\pp7)(5\pp6)\\
    \sigma_x=(1\pp 2); u_1=1,u_2=1,u_3=1; x(1)=6,x(2)=5,x(3)=7: x&=(1\pp6)(2\pp5)(3\pp7)\\
    &\\
    \alignedbox{\sigma_x=(1\pp 3); u_1=0,u_2=1,u_3=0; x(1)=3,x(2)=6,x(3)=1: x}{=(1\pp3)(2\pp6)(5\pp7)}\\
    \alignedbox{\sigma_x=(1\pp 3); u_1=1,u_2=1,u_3=1; x(1)=7,x(2)=6,x(3)=5: x}{=(1\pp7)(2\pp6)(3\pp5)}\\
    &\\
    \sigma_x=(2\pp 3); u_1=1,u_2=0,u_3=0; x(1)=5,x(2)=3,x(3)=2: x&=(1\pp5)(2\pp3)(6\pp7)\\
    \sigma_x=(2\pp 3); u_1=1,u_2=1,u_3=1; x(1)=5,x(2)=7,x(3)=6: x&=(1\pp5)(2\pp7)(3\pp6)
\end{align*}
\end{myenum}
\item So there are five involutions in $S_7$ with stabilzier $\langle a^4\rangle$ which have exactly two fixed points.
\end{myenum}

\end{example}

\begin{center}
    \includegraphics[width=0.9\textwidth]{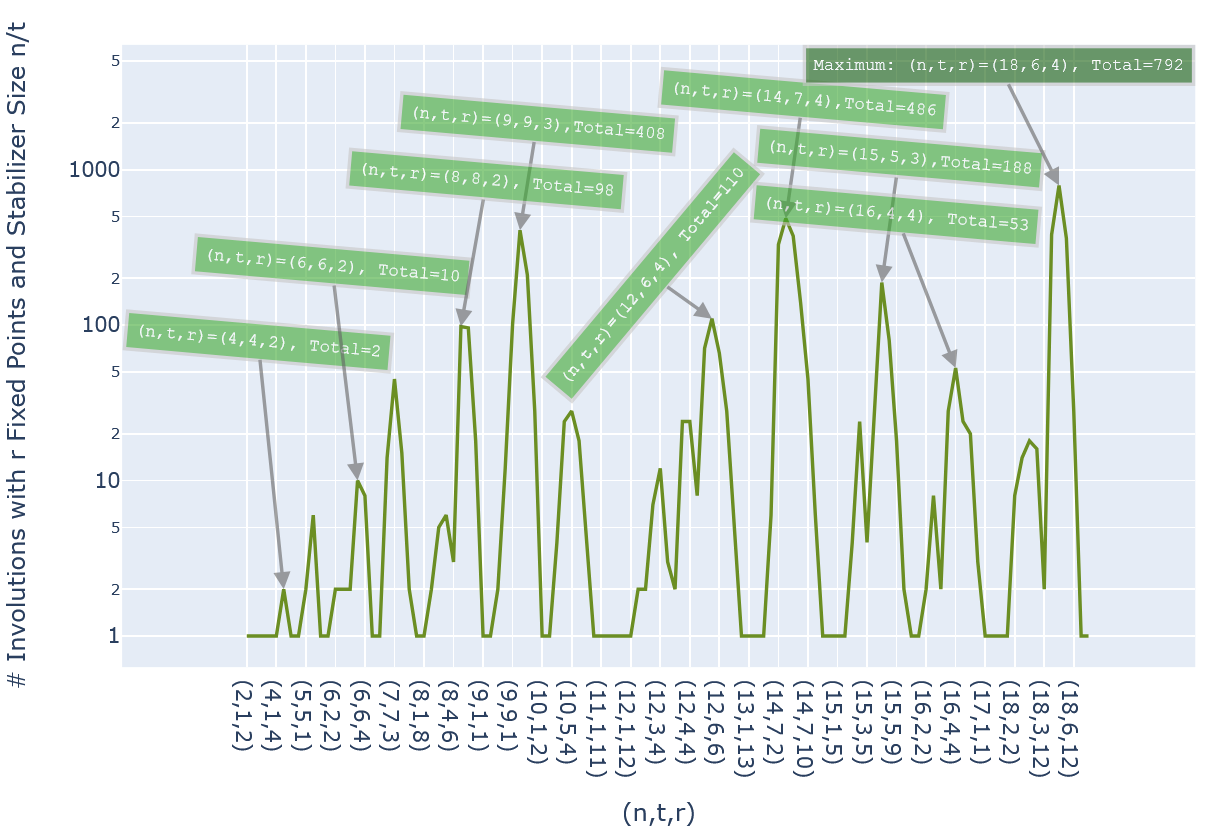}
      \vspace{-0.5cm}
    \begin{myfigure}\label{fig:graph inv with r fixed points} Graph of $R_{n/t,r}$: the number of involutions in $S_{n-1}$ with stabilier $\langle a^t\rangle$ and $r$ fixed points.\end{myfigure}
\end{center}

\section{Using \texorpdfstring{\cref{prop:number of ord 2 elements for given stab with certain fixed mod t}}{Proposition \ref{prop:number of ord 2 elements for given stab with certain fixed mod t}} and Counting Indicators Directly}\label{sec:app using prop number of ord 2 mod t}

Recall from \cref{prop:number of ord 2 elements for given stab with certain fixed mod t}, that $X_{n/t,r}$ is the set of involutions in $x\in S_{n-1}$ with stabilizer $\langle a^t\rangle$ such that $x$ has $r$ fixed points smaller than $t$ (equivalently, such that $\mathcal{O}_x$ contains $r$ involutions).

\begin{example}\label{n=12,t=3,r=1, inv with r inv in orbit}$\,$
Let us again go through an example. Say we want to find $X_{12/3,1}$ the number of involutions whose orbit contains one involution (and so contain $0$ involutions other than $x$ itself). These are all the involutions with $r-1=0$ fixed points smaller than $t=3$. 

Since $n=12$, $t=3,$ and $r=1$, we will use (13) from \cref{prop:number of ord 2 elements for given stab with certain fixed mod t}

\begin{myenum}
    \item First we compute $E_{n/t}$. Since $E_{12/3}=\{1,3\}$ we have two choices for $j\in E_{12/3}$. 
    \item Because there is only one $\sigma\in S_2$ with $r-1=0$ fixed points, we must have $\sigma_x=(1\pp2)$. Because $\sigma_x$ has no fixed points, $1\le u_1\le\frac{12}{3}=4$ can take any value but $u_2=-ju_1\mod\frac{12}{3}=4$. This gives us the following possibilities for $x$
\begin{align*}
    j=1; u_1=0,u_2=0; x(1)=2,x(2)=1: x&=(1\pp2)(4\pp5)(7\pp8)(10\pp11)\\
    j=1; u_1=1,u_2=3; x(1)=5,x(2)=10: x&=(1\pp5)(2\pp10)(4\pp8)(7\pp11)\\
    j=1; u_1=2,u_2=2; x(1)=8,x(2)=7: x&=(1\pp8)(2\pp7)(4\pp11)(5\pp10)\\
    j=1; u_1=3,u_2=1; x(1)=11,x(2)=4: x&=(1\pp11)(2\pp4)(5\pp7)(8\pp10)\\
    &\\
    j=3; u_1=0,u_2=0; x(1)=2,x(2)=1: x&=(1\pp2)(3\pp9)(4\pp11)(5\pp10)(7\pp8)\\
    j=3; u_1=1,u_2=1; x(1)=5,x(2)=4: x&=(1\pp5)(2\pp4)(3\pp9)(7\pp11)(8\pp10)\\
    j=3; u_1=2,u_2=2; x(1)=8,x(2)=7: x&=(1\pp8)(2\pp7)(3\pp9)(4\pp5)(10\pp11)\\
    j=3; u_1=3,u_2=3; x(1)=11,x(2)=10: x&=(1\pp11)(2\pp10)(3\pp9)(4\pp8)(5\pp7)
\end{align*}
\item Now, we must remove those we have over counted. However, for the first time, our recursion does not require us to do the same computation but for a smaller $t$. Because $t$ very much matters here, (we are looking for $x$ with no fixed points below $t$), we use a different formula now to remove overcount. 

We must subtract $C_{12,3,1,2}$. In this case, $E_{3/1}=\{1,2\}$ so we have two choices for $j_{\sigma_x}$. When $j_{\sigma_x}=1$, $\beta_{1,3/1}=3$ which does not divide $r$ so we ignore this case. When $j_{\sigma_x}=2$, $\beta_{2,3/1}=1$, $j_{\sigma_x}+\overline{m}\frac{3}{1}$ are $2,2+3,2+6,$ and $2+9$ so $\overline{E}_{2,12,3,1}=\{5,11\}$. Thus, there are two permutations we must remove: \begin{align*}
    j'=5: x&=(1\pp5)(2\pp10)(4\pp8)(7\pp11)\\
    j'=11: x&=(1\pp11)(2\pp10)(3\pp9)(4\pp8)(5\pp7)
\end{align*}
\item So there are $6$ involutions whose orbit contains one involution which have stabilizer $\langle a^3\rangle$ in $S_{11}.$ Namely, since each of these must be in a different orbit, there are $6$ orbits of length $3$ which contain exactly one involution. Thus, we have also computed $O_{12/3,1}=6$ by \cref{prop:orbit with r involutions length t}.

\end{myenum}
\end{example}

\begin{center}
    \includegraphics[width=0.7\textwidth]{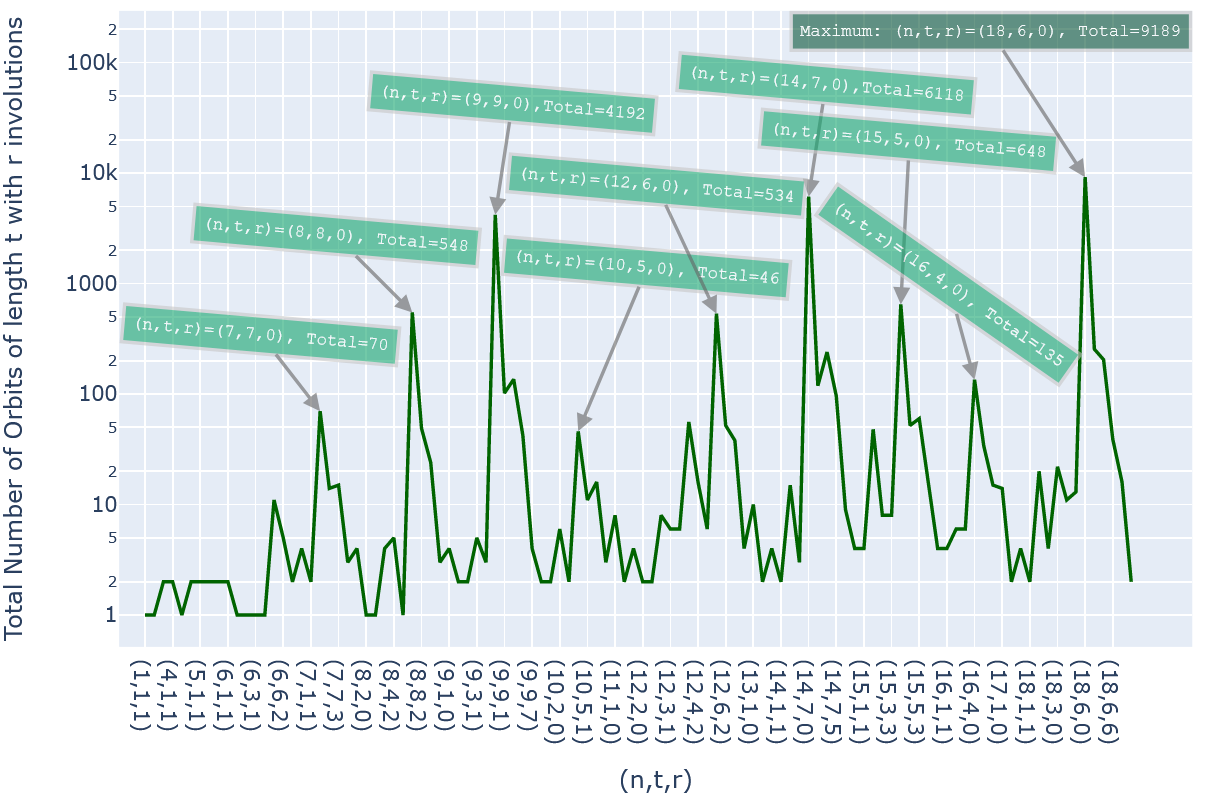}
      \vspace{-0.5cm}
    \begin{myfigure}\label{fig:graph orbits containing r inv} Graph of $O_{n/t,r}$: the number of orbits of length $t$ which contain $r$ involutions.\end{myfigure}
\end{center}

Now, let us explore how using direct computation proved helpful in the case where $t$ is even.

\begin{center}
    \includegraphics[width=0.9\textwidth]{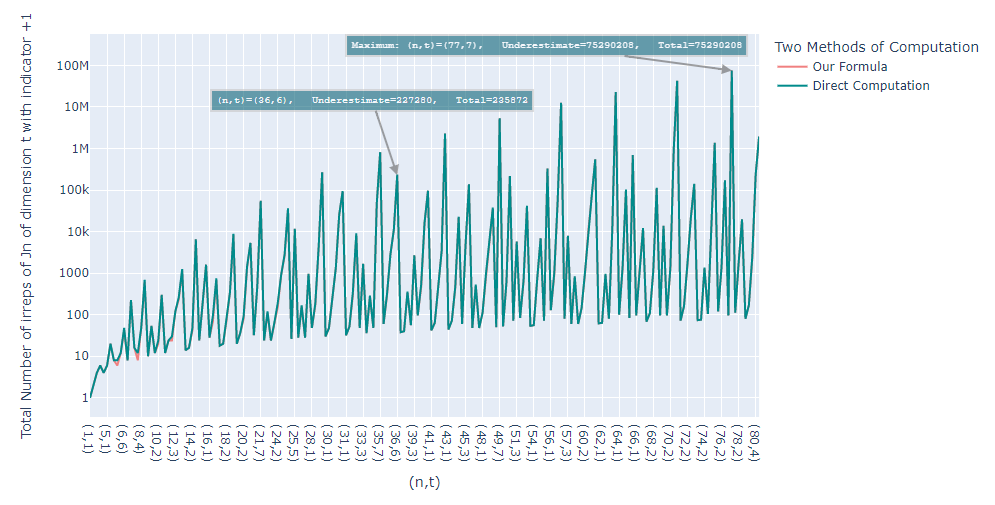}
      \vspace{-0.5cm}
    \begin{myfigure}\label{fig:graph irreps indicator plus one} Graph showing $I_{n/t}^{(+1)}$ the number of irreps of $J_n$ of dimension $t$ with indicator $+1$. Notice that here, when $t$ is even, our formula is an underestimate of the actual value. \end{myfigure}
\end{center}

\begin{center}
    \includegraphics[width=0.9\textwidth]{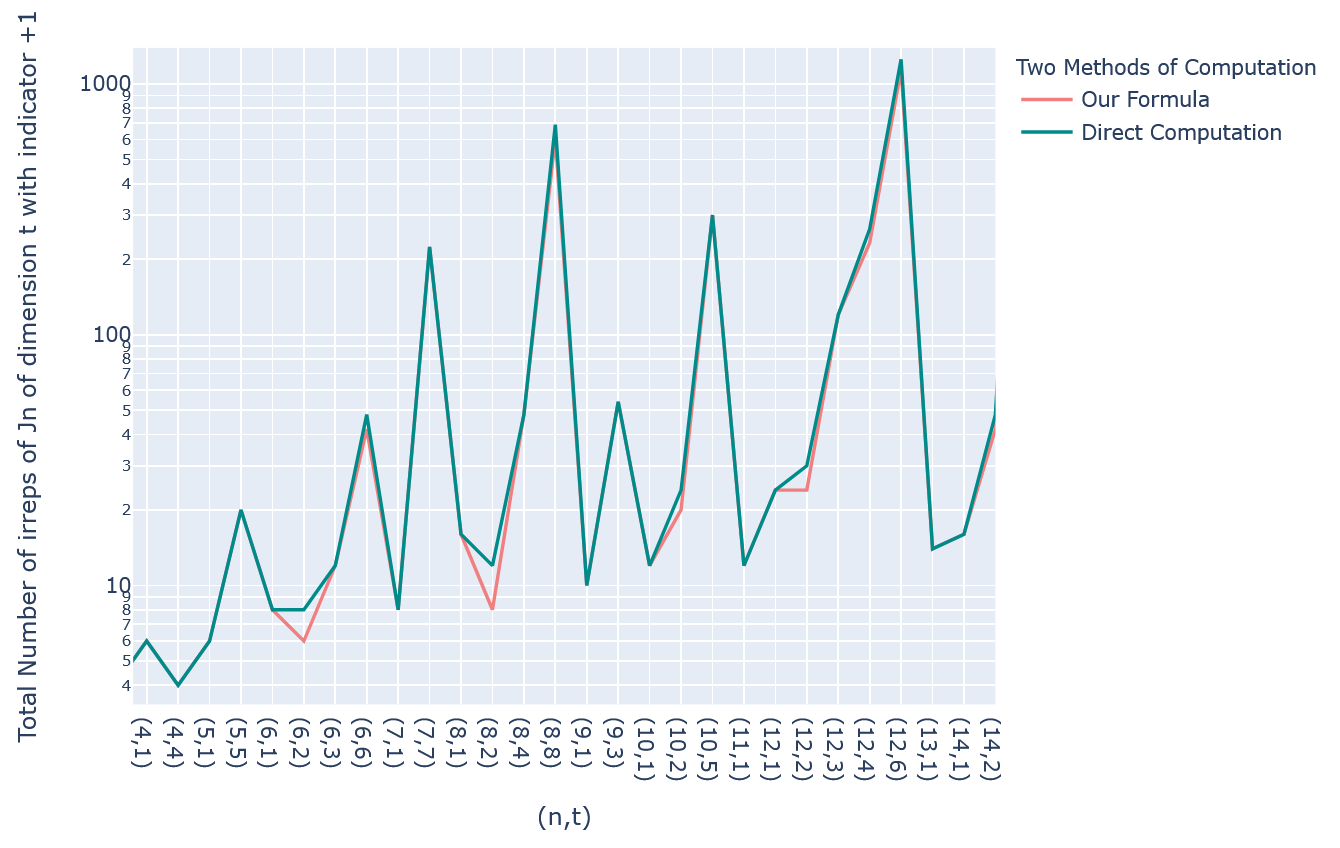}
      \vspace{-0.5cm}
    \begin{myfigure}\label{fig:graph irreps indicator plus one close up} Graph showing a close up of \cref{fig:graph irreps indicator plus one} which highlights the underestimate of our formula whenever $t$ is even.\end{myfigure}
\end{center}

Using Python, it was easy to conjecture that while there exist irreps with indicator $+1$ which are induced from permutations whose orbit does not contain an involution, they are not common. While we have not been able to formalize this in \cref{sec:limit behavior} because we could not come up for a formula to count such irreps, all computational evidence points toward this being true. 

\begin{center}
    \includegraphics[width=0.9\textwidth]{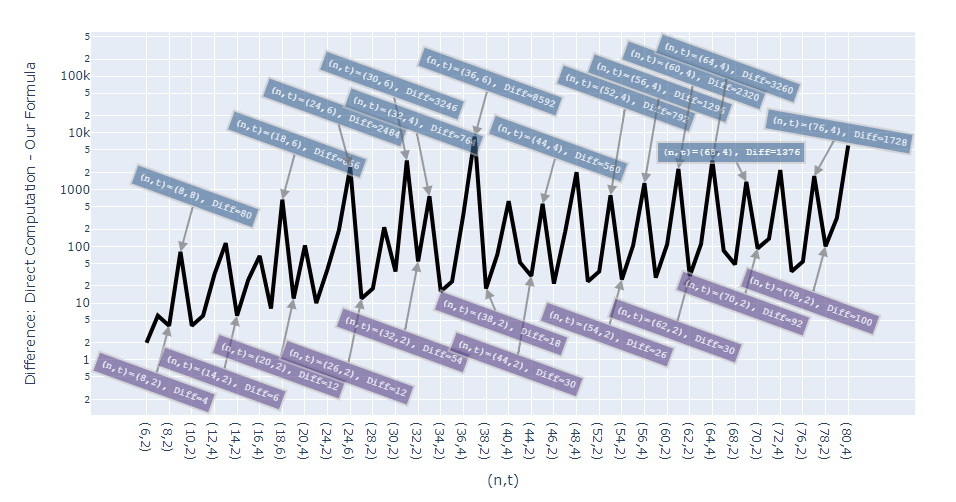}
      \vspace{-0.5cm}
    \begin{myfigure}\label{fig:graph irreps indicator plud one error} Graph showing the difference between direct computation and our formula for irreps of $J_n$ of dimension $t$ (for even $t$) with indicator $+1$ \end{myfigure}
\end{center}

Recall \cref{thm:number of positive indicators of dimension} tells us exactly how many irreps of $J_n$ of dimension $t$ with indicator $+1$ are induced from permutations $x$ whose orbit contains at least one involution. When the dimension $t$ is odd, the indicator can only be $+1$ if the permutation which induces it shares its orbit with an involution (see \cref{prop:observations}). However, when $t$ is even, there may be orbits which are closed under inversion whose elements induce irreps with indicator $+1$.

\cref{fig:graph irreps indicator plud one error} suggests that most of the irreps of $J_n$ of dimension $t$ with indicator $+1$ are induced from permutations $x$ whose orbit contains at least one involution.

For example, \cref{fig:graph irreps indicator plus one} shows that there are $235,872$ irreps of $J_{36}$ of dimension $6$ with indicator $+1$ and of these, $227,280$ are induced from permutations whose orbit contains at least one involution. Namely, only $3.6\%$ of the irreps of $J_{36}$ of dimension $6$ are induced from permutations whose orbit contains no involutions.

Since \cref{thm:percent +1 ind goes to zero} tells us that for fixed {\it odd} $t$, the percent of irreps of $J_{mt}$ of dimension $t$ with indicator $+1$ approaches $0$ as $m\to\infty$, it seems logical to purpose that the same holds true for even $t$.

\begin{conjecture}\label{conj even t percent +1 ind goes to 0}
{\it Fix $t$ an even integer. Let $n=mt$ for some integer $m\ge2$. Then the ratio of irreps of $J_n$ of dimension $t$ with indicator $+1$ to the total number of irreps of dimension $t$ converges to $0$ as $m\to\infty$.}
\end{conjecture}

Of course, \cref{thm:percent nonzero ind goes to zero} tells us this is true when $t=2$, but for general even $t>2$, we have only a {\it lower bound} for the number of irreps with indicator $+1$ and this is not helpful for bounding the ratio from above.

\begin{conjecture}\label{conj even t percent +1 induced form non-inv goes to 0}
{\it Fix $t$ an even integer. Let $n=mt$ for some integer $m\ge 2$. Then the ratio of irreps of $J_n$ of dimension $t$ with indicator $+1$ which are induced from permutations which do not share their orbit with an involution to the total number of irreps of dimension $t$ with indicator $+1$ convertes to $0$ as $m\to\infty$.}
\end{conjecture}

\begin{center}
    \includegraphics[width=0.9\textwidth]{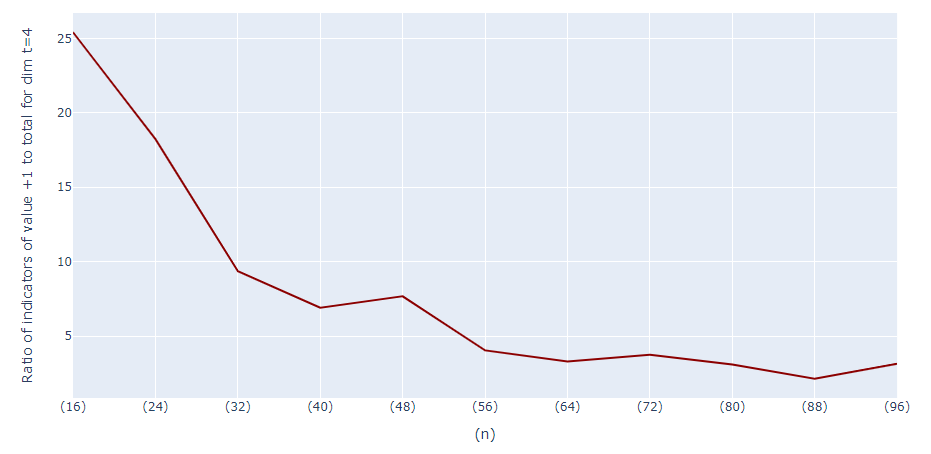}
      \vspace{-0.5cm}
    \begin{myfigure}\label{fig:graph percent +1 t=4} Graph showing percent of $4$-dimensional irreps of $J_n$ with indicator $+1$ to total number of $4$-dimensional irreps of $J_n$ \end{myfigure}
\end{center}

\section{Using \texorpdfstring{\cref{thm:classifying indicators when t=2}}{Proposition \ref{thm:classifying indicators when t=2}}}\label{sec:app using thm classifying indicators when t=2}

\begin{center}
    \begin{tikzpicture}[style={font=\small}, node distance = 3.5cm, auto]
     \node[diamond, scale=0.65, draw, fill=blue!20, text width = 5em, text centered, rounded corners, minimum height = 4em] (j) {\normalsize $x(2)=2j$; choose $j\in E_{n/2}$};
    \node[diamond, scale=0.65, draw, fill=red!20, text width = 5em, text centered, rounded corners, minimum height = 4em, right of=j, above of=j] (kj0) {$x=x^{-1}$; choose $u$ such that $u(j+1)=0$};
    \node[rectangle, scale=0.65, draw, fill=red!20, text width = 5em, text centered, rounded corners, minimum height = 5em, above of=kj0,node distance = 4.5cm] (j2k1) { remove case $\mathcal{O}_x=\{x\}$; ignore $u$ such that $j=2u+1$};
    \node[diamond, scale=0.65, draw, fill=green!20, text width = 5.5em, text centered, rounded corners, minimum height = 4em, below of=j, right of=j] (k+1j0) {$x^{-1}\not=x$; choose $u$ such that \\ {\scriptsize$(u+1)(j-1)=0$}};
    \node[rectangle, scale=0.65, draw, fill=green!20, text width = 5em, text centered, rounded corners, minimum height = 4em, below of=k+1j0,node distance = 4cm] (j2k1also) {remove case $\mathcal{O}_x=\{x\}$; ignore $u$ such that $j=2u+1$};
    \node[diamond, scale=0.65, draw, fill=blue!20, text width = 5em, text centered, rounded corners, minimum height = 4.5em,right of=kj0,below of =kj0] (jnot2k) {\normalsize ensure $j\not=2u+1$, {\scriptsize$x(1)=2u+1$}, $x(2)=2j$};
    \node[diamond, scale=0.65, draw, fill=red!20, text width = 5em, text centered, rounded corners, minimum height = 4em, above of=jnot2k,right of=jnot2k] (xinv) {\normalsize $x^{-1}=x$};
    \node[diamond, scale=0.65, draw, fill=green!20, text width = 5em, text centered, rounded corners, minimum height = 4em, below of=jnot2k,right of=jnot2k] (xinvnot) {$x\not=x^{-1}$ $x^{-1}\in\mathcal{O}_x$};
    \node[diamond, scale=0.65, draw, fill=blue!20, text width = 5em, text centered, rounded corners, minimum height = 4em,right of=xinv, below of =xinv] (i) {\large choose $i$};
    \node[ellipse, scale=0.65, draw, fill=red!20, text width = 5em, text centered, rounded corners, minimum height = 4em, above of=i,right of=i,node distance = 4cm] (indplus) {if $i(j+1)=0$, then $\nu(\hat{\chi}_i)=+1$};
    \node[ellipse, scale=0.65, draw, fill=red!20, text width = 5em, text centered, rounded corners, minimum height = 4em, above of=indplus] (indzero) {if $i(j+1)\not=0$, then $\nu(\hat{\chi}_i)=0$};
    \node[diamond, scale=0.65, draw, fill=green!20, text width = 5em, text centered, rounded corners, minimum height = 4em, below of=i,right of =i] (notindplus) {$i(j+1)=0$};
    \node[ellipse, scale=0.65, draw, fill=green!20, text width = 5em, text centered, rounded corners, minimum height = 4em, above of =notindplus, right of =notindplus,node distance = 2.5cm] (notindplusfin) {$i(u+1)=0$, then $\nu(\hat{\chi}_i)=+1$};
    \node[ellipse, scale=0.65, draw, fill=green!20, text width = 5em, text centered, minimum height = 4em, below of=notindplus,right of =notindplus,node distance = 2.5cm] (notindminusfin) {$i(u+1)\not=0$, then $\nu(\hat{\chi}_i)=-1$};
    \node[ellipse, scale=0.65, draw, fill=green!20, text width = 5em, text centered, minimum height = 4em, below of=notindplus,node distance = 5cm] (notindzero) {$i(j+1)\not=0$, then $\nu(\hat{\chi}_i)=0$};
    \path[draw,->] (j) -- (kj0);
    \path[draw,->] (j) -- (k+1j0);
    \path[draw,->] (kj0) -- (j2k1);
    \path[draw,->] (k+1j0) -- (j2k1also);
    \path[draw,->] (kj0) -- (jnot2k);
    \path[draw,->] (k+1j0) -- (jnot2k);
    \path[draw,->] (jnot2k) -- (xinv);
    \path[draw,->] (jnot2k) -- (xinvnot);
    \path[draw,->] (xinv) -- (i);
    \path[draw,->] (xinvnot) -- (i);
    \path[draw,->] (i) -- (indplus);
    \path[draw,->] (i) -- (indzero);
    \path[draw,->] (i) -- (notindzero);
    \path[draw,->] (i) -- (notindplus);
    \path[draw,->] (notindplus) -- (notindplusfin);
    \path[draw,->] (notindplus) -- (notindminusfin);
    \end{tikzpicture}
    \begin{myfigure}\label{fig:chart of how to determine indicators from x when t=2}{ A flow chart illustrating the process of \cref{thm:classifying indicators when t=2}. The red entries indicate the construction of involutions $x$ and the process of determining which indicators are $+1$ and which are $0$. The green entries indicate the construction of $x$ such that $\mathcal{O}_x=\{x,x^{-1}\}$ (so namely, $x\not=x^{-1}$) and the process of determining which indicators are $+1,-1,0$. The rectangles illustrate that $F_x=F$ if $j=2u+1\mod\frac{n}{2}$, the diamonds illustrate intermediate computation, and the circles illustrate the final conclusion of the indicator.}\end{myfigure} 
\end{center}

\begin{example}\label{ex indicators n=12,t=2}$\,$

For example, let $n=12$. We will to compute $I_{12/2}^{(+1)},I_{12/2}^{(-1)},$ and $I_{12/2}^{(0)}$.

\begin{myenum}
    \item First we compute $E_{n/t}.$ Since $n=12$ and $t=2$, then $E_{12/2}=\{1,5\}$. Now, we go case by case. 
    
    \begin{myenum}
        \item For $j=1$: \begin{myenum}
            \item Again, $j=1$ so $u(j+1)=2u\mod\frac{12}{2}=6$ so $u=0\mod3$ so $u=6,3\mod 6$. Since $1=2(6)+1\mod 6$ and $1=2(3)+1\mod 6$, both values of $u$ are out, so there are no involutions $x$ satisfying $x(2)=2$ with stabilizer $\langle a^2\rangle$.
            \item However, we instead we look at $(u+1)(j-1)=0\mod6$ we get that $u$ can take any value. Since $j=1\not=2u+1$, we cannot use $u=6,u=3$ so there are four $(u=1,2,4,5)$ possible permutations which are not involutions. Explicitly, they are \begin{align*}
    j=1; u=1; x(1)=3: x_1&=(1\pp 3\pp5\pp 7\pp 9\pp11)\\
    j=1; u=2; x(1)=5: x_2&=(1\pp5\pp9)(3\pp7\pp11)\\
    j=1; u=4; x(1)=9: x_3&=(1\pp9\pp5)(3\pp 11\pp7)\\
    j=1; u=5; x(1)=11: x_4&=(1\pp11\pp9\pp7\pp5\pp3)
\end{align*}
        \end{myenum}
        \item Now, consider $j=5$. 
        \begin{myenum}
            \item Then $u(j+1)=0\mod 6$ for any $1\le u\le 6$. Since we cannot have $5=2u+1\mod 6$ we toss out $u=2,5$. This leaves four possible involutions. Now, $K_{5,12/2}=\{1,2,3,4,5,6\}$ since $i(j+1)=0\mod 6$ for any $i$. Thus, all indicators are $+1$ for all of these involutions. The involutions here are \begin{align*}
    j=5; u=6; x(1)=1: x_5&=(2\pp 10\pp)(3\pp11)(4\pp8)\\
    j=5; u=1; x(1)=3: x_6&=(1\pp 3)(2\pp 10)(4\pp 8)(5\pp 11)(7\pp 9)\\
    j=5; u=3; x(1)=7: x_7&=(1\pp 7)(2\pp 10)(3\pp 5)(4\pp 8)(9\pp 11)\\
    j=5; u=4; x(1)=9: x_8&=(1\pp 9)(2\pp 10)(3\pp 7)(4\pp 8)
\end{align*}
\item For the case where $(u+1)(j-1)=4(u+1)=0\mod 6$, we require that $u=-1\mod 3$ so $u=2,5$. However, because for both of these $j=2u+1$ so there are no permutations which are not involutions satisfying that $x(2)=10$ and $F_x=\langle a^2\rangle$.
        \end{myenum}
        \item Finally, we count the indicators of the irreps induced from each $x$. $x_1$ and $x_4$ each contribute $2$ indicators of value $+1$ and $4$ indicators of value $0$. $x_2$ and $x_3$ each contribute $1$ indicator of value $+1$, $1$ indicator of value $-1$, and $4$ indicators of value $0$. $x_5,x_6,x_7$ and $x_8$ each contribute $6$ indicators of value $+1$. In total, we have \begin{align*}
    I_{12/2}^{(+1)}&=2\cdot 2+2\cdot 1+4\cdot 6=4+2+24=30\\
    I_{12/2}^{(-1)}&=2\cdot 0+2\cdot 1+4\cdot 0=2\\
    I_{12/2}^{(0)}&=2\cdot 4+2\cdot 4+4\cdot 0=16
\end{align*}
    \end{myenum}
\end{myenum}

Before concluding, let us go more in depth about the actual representations induced from the permutations we've found. 

Consider the four permutations from \textbf{(1).(i).(b)}. They are grouped into two orbits, where each orbit contains $x$ and $x^{-1}$. Now, $K_{1,12/2}=\{i\,|\,2i=0\mod 6\}=\{0,3\}$. Since $6(u+1)=0$ for any $u=1,2,4,5$ this gives four irreps of $J_n$ of dimension $2$ and indicator $+1$. Similarly, $3(u+1)=0\mod 6$ for $u=1,5$ which adds two more irreps with indicator $+1$. However, for $u=2,4$ we have two irrpes of $J_n$ with dimension $2$ which have indicator $-1$ while the rest are of indicator $0$.

Let's be even more specific. We use the information we found about $u$ and $j$ to construct the associated irreducible representations of $J_n$ which have the indicators described. Consider $x_2=(1\pp5\pp9)(3\pp7\pp11)$ which has stabilizer $\langle a^2\rangle$. Let $\zeta_6$ be a primitive $6^{\text{th}}$ root of unity. There are $6$ irreducible group representations associated to $x_2$ defined by \begin{align*}
    \rho_{x_2,i}:\langle a^2\rangle&\to\mathbb{C}^\times\\
    a^2&\mapsto \zeta_6^i
\end{align*} 

for $1\le i\le 6$. Consider the equivalent $F_x$-module $V_{x_2,i}$ to the representation $\rho_{x_2,i}$. Then the induced $J_n$ module is $$\hat{V}_{x_2,i}=\mathbb{C}\langle a\rangle\otimes_{\mathbb{C}\langle a^2\rangle}V_{x_2,i}$$ which is defined by the action $$(p_y\#a^l)\cdot(a^q\otimes v)=\begin{cases} a^{l+q}\otimes v&\text{ if }x=a^{-y(l+q)}ya^{l+q}\\ 0 &\text{ otherwise} \end{cases}.$$ The associated character $\hat{\chi}_{x_2,i}$ gives rise the following indicators:
\begin{align*}
    \nu(\hat{\chi}_{x_2,1})&=0 & \nu(\hat{\chi}_{x_2,2})&=0\\
    \nu(\hat{\chi}_{x_2,3})&=-1 & \nu(\hat{\chi}_{x_2,4})&=0\\
    \nu(\hat{\chi}_{x_2,5})&=0 & \nu(\hat{\chi}_{x_2,6})&=+1
\end{align*}

$x_3$ induces the same irreps as $x_2$ and so they have the same indicator table. For $x_1$ and $x_4$, $i=3$ gives indicators which are $+1$ instead of $-1$.
\end{example}

 \vspace{-0.5cm}

\begin{center}
    \includegraphics[width=0.9\textwidth]{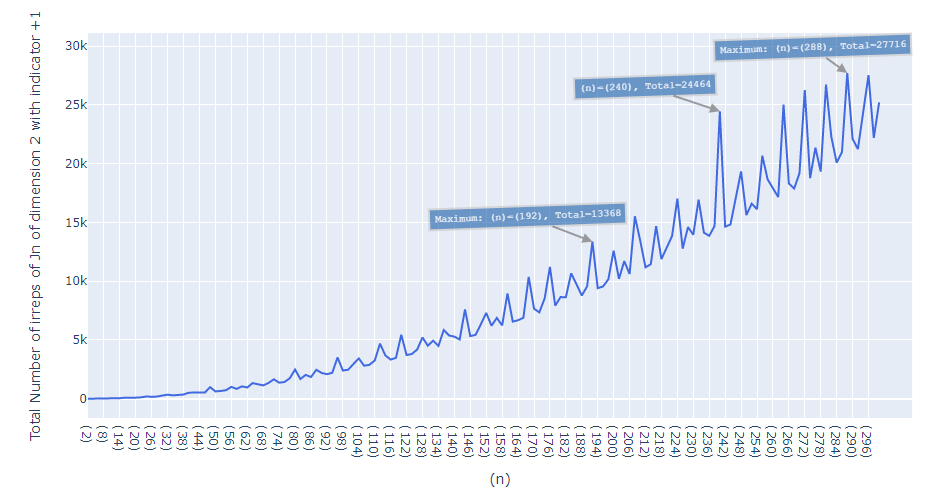}
      \vspace{-0.6cm}
    \begin{myfigure}\label{fig:graph irreps of dim 2 with indicator +1} Graph of $I_{n/2}^{(+1)}$ showing the number of irreps of $J_n$ of dim $2$ with indicator $+1.$\end{myfigure}
\end{center}

\begin{center}
    \includegraphics[width=0.9\textwidth]{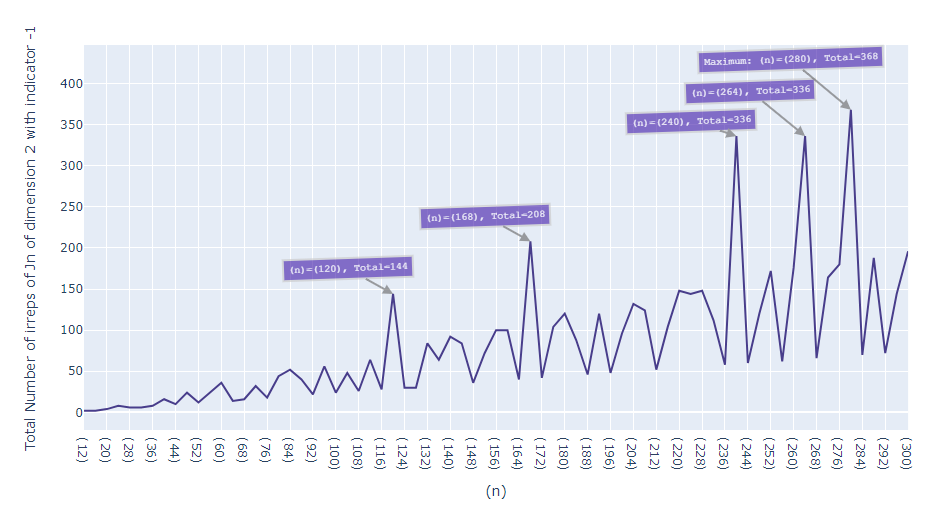}
      \vspace{-0.5cm}
    \begin{myfigure}\label{fig:graph irreps of dim 2 with indicator -1} Graph of $I_{n/2}^{(-1)}$ showing the number of irreps of $J_n$ of dim $2$ with indicator $-1.$\end{myfigure}
\end{center} 

\begin{center}
    \includegraphics[width=0.95\textwidth]{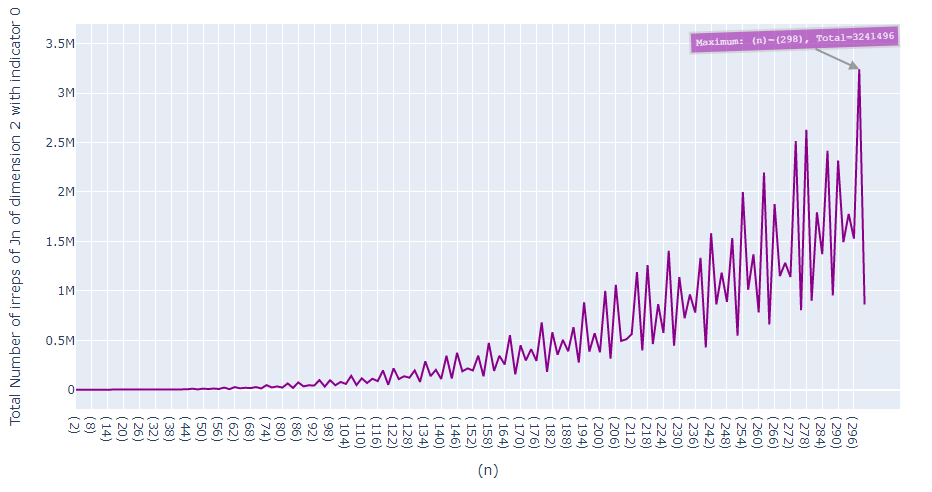}
      \vspace{-0.5cm}
    \begin{myfigure}\label{fig:graph irreps of dim 2 with indicator 0} Graph of $I_{n/2}^{(0)}$ showing the number of irreps of $J_n$ of dim $2$ with indicator $0.$\end{myfigure}
\end{center}

To emphasize the difference in scale, we lay multiple graphs on a single plot.

\begin{center}
    \includegraphics[width=0.95\textwidth]{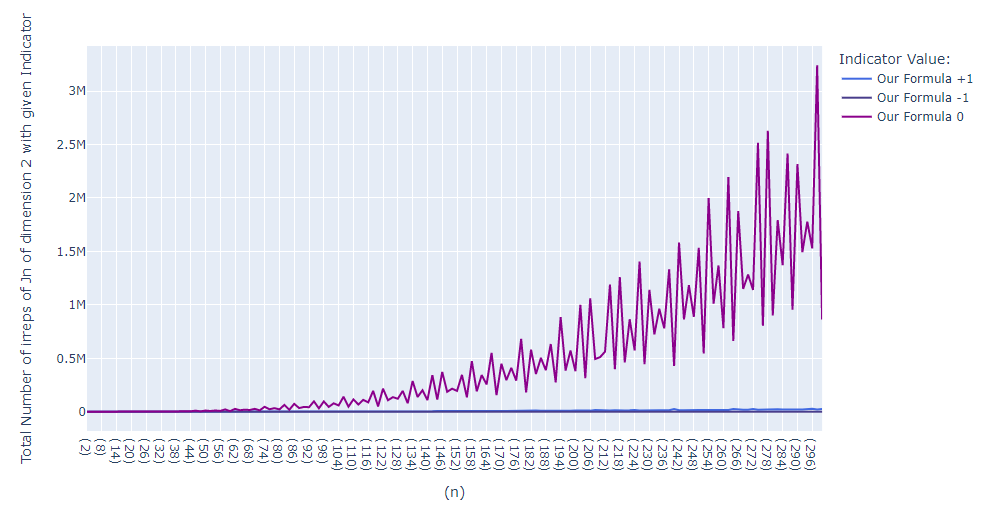}
      \vspace{-0.5cm}
    \begin{myfigure}\label{fig:graph irreps of dim 2 all ind} Plot showing all three of \cref{fig:graph irreps of dim 2 with indicator +1}, \cref{fig:graph irreps of dim 2 with indicator -1}, and \cref{fig:graph irreps of dim 2 with indicator 0} on a single graph.\end{myfigure}
\end{center}

\begin{center}
    \includegraphics[width=0.95\textwidth]{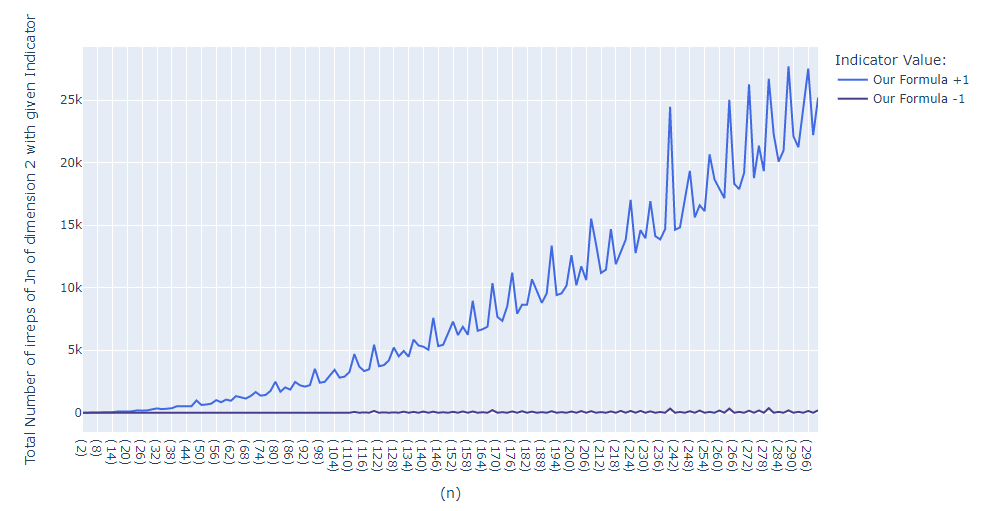}
      \vspace{-0.5cm}
    \begin{myfigure}\label{fig:graph irreps of dim 2 nonzero ind} Plot showing Figures \cref{fig:graph irreps of dim 2 with indicator +1} and \cref{fig:graph irreps of dim 2 with indicator -1} on a single graph.\end{myfigure}
\end{center}

\section{Counting Irreps with Indicator \texorpdfstring{$-1$}{TEXT} Directly}

While we do not have any explicit bounds for the number of irreps with indicator $-1$ as we do for irreps with indicator $+1$, Python was still able to count irreps with indicator $-1$ for some $n$ and $t$.

\begin{mytable}[Table of Negative Indicators] \label{fig:table of negative indicators}\end{mytable}

\noindent \begin{tabular}{c V{3} c|c||c|c||c||c|c|c||c||c|c||c|c||c|c||c||c|c}
\hline
    $n$ & \multicolumn{2}{c||}{$12$} & \multicolumn{2}{c||}{$16$} & \multicolumn{1}{c||}{$20$} & \multicolumn{3}{c||}{$24$} & \multicolumn{1}{c||}{$28$} & \multicolumn{2}{c||}{$32$} & \multicolumn{2}{c||}{$36$} & \multicolumn{2}{c||}{$40$} & \multicolumn{1}{c||}{$44$} & \multicolumn{2}{c}{$48$}\\
    \hline
        $t$ & $2$ & $6$ & $2$ & $4$ & $2$ & $2$ & $4$ & $6$ & $2$ & $2$ & $4$ & $2$ & $6$ & $2$ & $4$ & $2$ & $2$ & $4$\\
        \Xhline{1.2pt}
        $\displaystyle I_{n/t}^{(-1)}$ & $2$ & $42$ & $2$ & $20$ & $4$ & $8$ & $64$ & $816$ &  $6$ & $6$ & $124$ & $8$ & $2976$ & $16$ & $208$ & $10$ & $24$ & $648$\\
        \hline
 \end{tabular}
 
 \noindent{\small Note: Table showing total number of irreps of $J_n$ of dimension $t$ with indicator $-1$. This value was computed directly in Python. Note how $t$ was kept small to make computation possible.}

\begin{center}
    \includegraphics[width=0.95\textwidth]{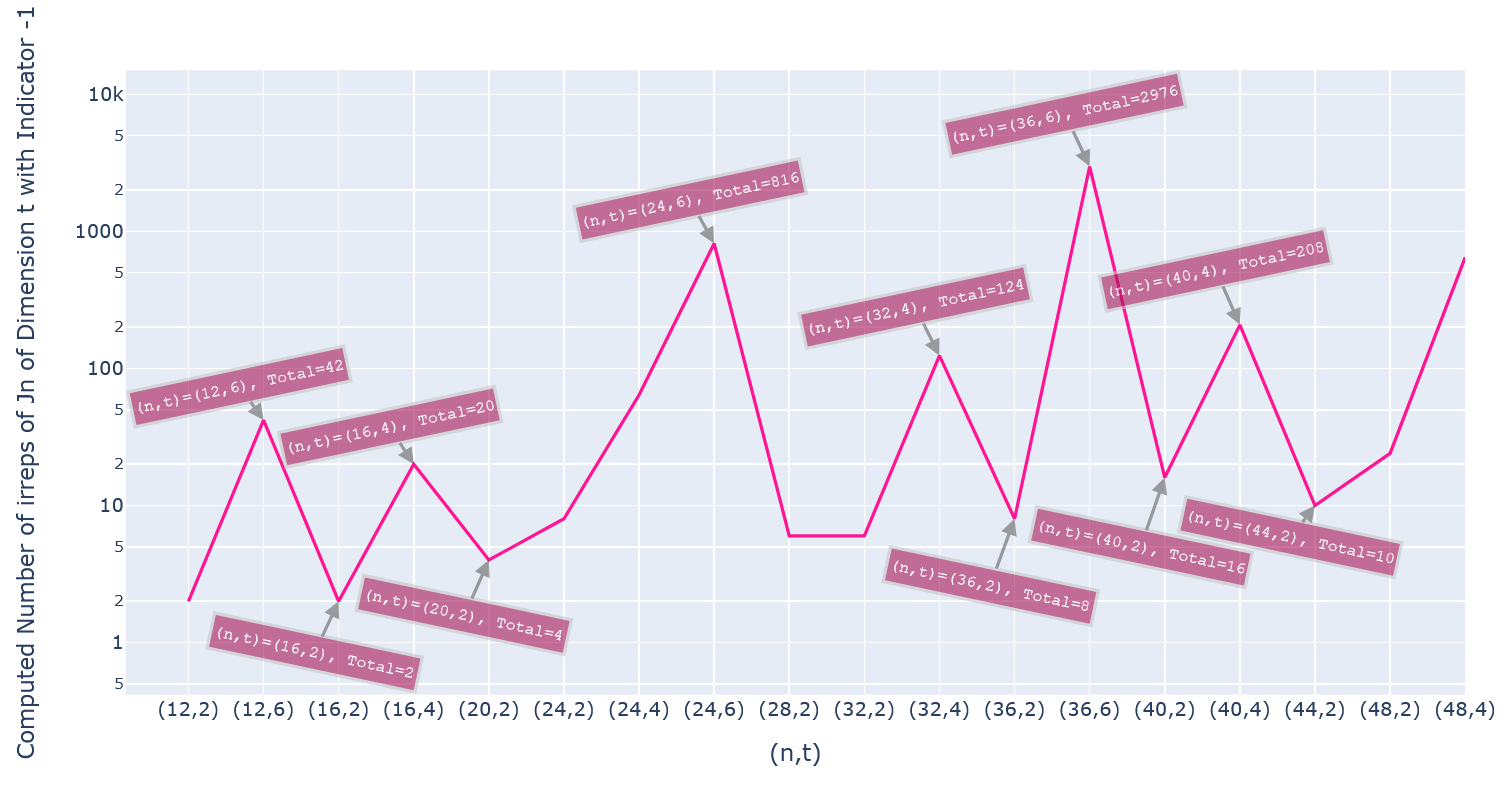}
      \vspace{-0.5cm}
    \begin{myfigure}\label{fig:graph computed irreps dim t ind -1} Graph showing number of irreps of $J_n$ of dimension $t$ with indicator $-1$. This value was computed directly in Python. \end{myfigure}
\end{center}

In comparison with the number of irreps of dimension $t$ with indicator $+1$ in \cref{fig:graph irreps indicator plus one}, the values in \cref{fig:graph computed irreps dim t ind -1} appear phenominally smaller.

Recall from the discussion near the end of Appendix D, that there are $235,872$ irreps of $J_{36}$ of dimension $6$ with indicator $+1$ and of these, $227,280$ are induced from permutations whose orbit contains at least one involution. \cref{fig:graph computed irreps dim t ind -1} shows that there are $2,976$ irreps of $J_{36}$ of dimension $6$ with indicator $-1$, and there are $\frac{36}{6}|M_{36/6}|=11,189,952$ total irreps of $J_{36}$ of dimension $6$.

Namely, when $n=36$ and $t=6$, about $2.1\%$ of the irreps of $J_{36}$ of dimension $6$ have indicator $+1$ and $0.027\%$ have indicator $-1$. Based on our direct analysis, this appears to be a trend, not a fluke and so we end with one final conjecture.

\begin{conjecture}\label{conj even t percent nonzero ind goes to 0}
{\it Fix $t$. Let $n=mt$ for some integer $m\ge2$. Then the number ratio of irreps of $J_n$ of dimension $t$ with nonzero indicator to the total number of irreps of dimension $t$ converges to $0$ as $m\to\infty$.}
\end{conjecture}

\begin{center}
    \includegraphics[width=0.9\textwidth]{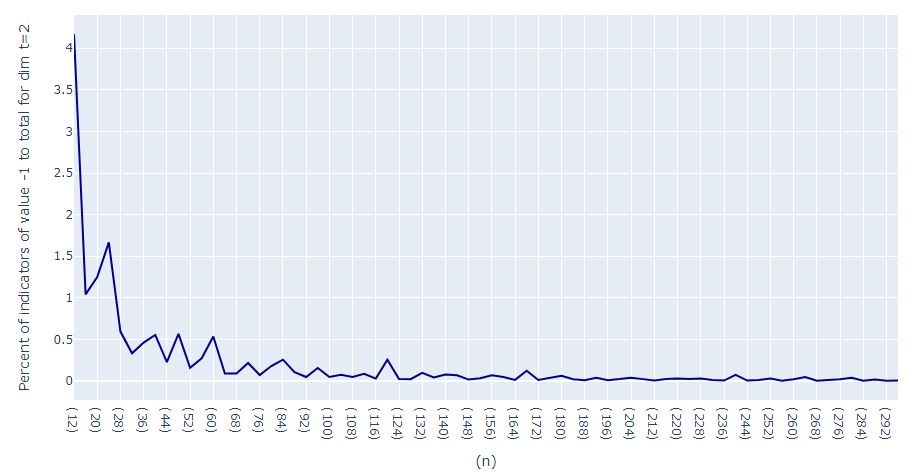}
      \vspace{-0.5cm}
    \begin{myfigure}\label{fig:graph percent -1 t=2} Graph showing percent of $2$-dimensional irreps of $J_n$ with indicator $-1$ to total number of $2$-dimensional irreps of $J_n$ \end{myfigure}
\end{center}

\begin{center}
    \includegraphics[width=0.9\textwidth]{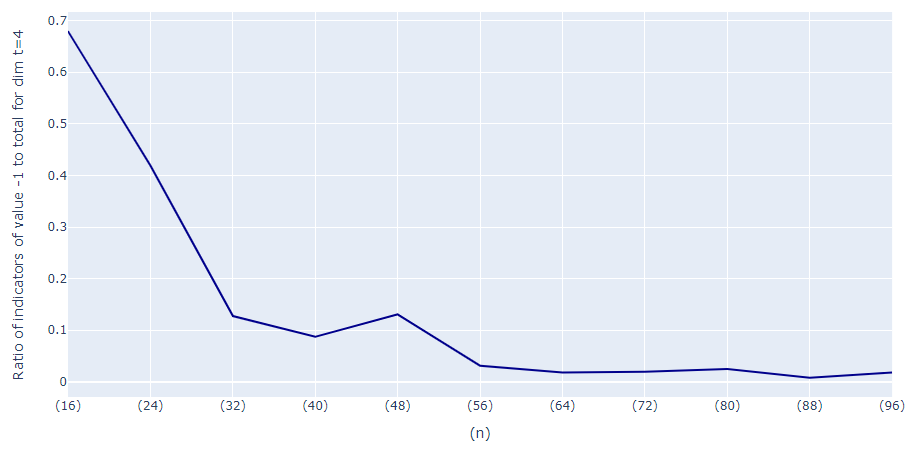}
      \vspace{-0.5cm}
    \begin{myfigure}\label{fig:graph percent -1 t=4} Graph showing percent of $4$-dimensional irreps of $J_n$ with indicator $-1$ to total number of $4$-dimensional irreps of $J_n$ \end{myfigure}
\end{center}

\newpage

\noindent\textbf{Supplementary Information} No supplementary files are included.\\

\noindent\textbf{Acknowledgements} Many thanks to my advisor Susan Montgomery for being a patient and thoughtful advisor. Also thank you to Patrick M. Reardon for teaching me the basics of Python code. This work is part of a PhD thesis at the University of Southern California.\\

\noindent\textbf{Declarations} There are no sources of funding or conflicts of interest in this work. This work is part of a PhD thesis at University of Southern California.\\

\noindent Additionally, the Python notebook used to generate the graphs used in this work is available on GitHub. 

\href{https://github.com/KaylaOrlinsk/USCProject}{https://github.com/KaylaOrlinsk/USCProject}

\newpage


\end{document}